\documentclass[11pt,reqno]{amsart}

\usepackage{geometry}
\usepackage{amsmath, amssymb, amsthm}
\usepackage{mathtools} %
\usepackage{bm} %
\usepackage{physics}
\usepackage{datetime}
\usepackage{microtype}

\usepackage{xcolor} %
\usepackage{graphicx} %
\usepackage{soul} %
\usepackage[colorlinks=true]{hyperref} %
\usepackage[capitalise]{cleveref} %

\DeclareMathOperator{\Prob}{\mathbb{P}}
\DeclareMathOperator{\E}{\mathbb{E}}
\DeclareMathOperator{\Var}{Var}
\DeclareMathOperator{\Cov}{Cov}

\DeclareMathOperator{\lin}{lin}
\DeclareMathOperator{\argmin}{argmin}

\DeclareMathOperator{\supp}{supp}

\newcommand{\R}{\mathbb{R}}
\newcommand{\N}{\mathbb{N}}

\newcommand{\gwass}{\mathsf{W}^{(\sigma)}_1}
\newcommand{\Wp}{\mathsf{W}_p}
\newcommand{\GWp}{\mathsf{W}_p^{(\sigma)}}

\newcommand{\SWp}{\mathsf{S}^{(\sigma)}_p}

\newcommand{\cA}{\mathcal{A}}
\newcommand{\cB}{\mathcal{B}}
\newcommand{\cC}{\mathcal{C}}

\newcommand{\cF}{\mathcal{F}}

\newcommand{\cP}{\mathcal{P}}

\newcommand{\cU}{\mathcal{U}}

\newcommand{\cX}{\mathcal{X}}

\newcommand{\EE}{\mathbb{E}}

\newcommand{\RR}{\mathbb{R}}

\newcommand{\bunderline}[1]{\underline{#1\mkern-3mu}\mkern 3mu }
\newcommand{\opt}{^{\star}}
\newcommand{\Hu}{\dot H^{-1,p}\left(\mu*\gamma_{\sigma}\right)}
\newcommand{\inp}[2]{\left\langle #1,#2\right\rangle}
\newcommand{\dn}{\downarrow}

\theoremstyle{plain} %
\newtheorem{theorem}{Theorem}[section]
\newtheorem{lemma}{Lemma}[section]
\newtheorem{proposition}{Proposition}[section]
\newtheorem{corollary}{Corollary}[section]
\newtheorem{assumption}{Assumption}

\theoremstyle{definition}
\newtheorem{definition}{Definition}[section]
\newtheorem{remark}{Remark}[section]
\newtheorem{example}{Example}[section]

\graphicspath{ {./images/} }

\usepackage{etoolbox}
\usepackage{paralist}
\patchcmd{\subsubsection}{\itshape}{\bfseries}{}{}
\usepackage{bbm}

\begin{document}

\title[Limit distribution theory for smooth Wasserstein distances]{Limit distribution theory for smooth $p$-Wasserstein distances}

\thanks{
Z. Goldfeld is supported by the NSF CRII grant CCF-1947801, in part by the 2020 IBM Academic Award, and in part by the NSF CAREER Award CCF-2046018. 
K. Kato is partially supported by the  NSF grants DMS-1952306 and DMS-2014636.
S. Nietert is supported by the NSF Graduate Research Fellowship under Grant DGE-1650441. }

\author[Z. Goldfeld]{Ziv Goldfeld}

\address[Z. Goldfeld]{
School of Electrical and Computer Engineering, Cornell University.
}
\email{goldfeld@cornell.edu}

\author[K. Kato]{Kengo Kato}

\address[K. Kato]{
Department of Statistics and Data Science, Cornell University.
}
\email{kk976@cornell.edu}

\author[S. Nietert]{Sloan Nietert}

\address[S. Nietert]{
Department of Computer Science, Cornell Universty.}
\email{nietert@cs.cornell.edu}

\author[G. Rioux]{Gabriel Rioux}

\address[G. Rioux]{
Center for Applied Mathematics, Cornell University.}
\email{ger84@cornell.edu}

\begin{abstract}
The Wasserstein distance is a metric on a space of probability measures that has seen a surge of applications in statistics, machine learning, and applied mathematics. However, statistical aspects of Wasserstein distances are bottlenecked by the curse of dimensionality, whereby the number of data points needed to accurately estimate them grows exponentially with dimension. Gaussian smoothing was recently introduced as a means to alleviate the curse of dimensionality, giving rise to a parametric convergence rate in any dimension, while preserving the Wasserstein metric and topological structure. To facilitate valid statistical inference, in this work, we develop a comprehensive limit distribution theory for the empirical smooth Wasserstein distance. The limit distribution results leverage the functional delta method after embedding the domain of the Wasserstein distance into a certain dual Sobolev space, characterizing its Hadamard directional derivative for the dual Sobolev norm, and establishing weak convergence of the smooth empirical process in the dual space. To estimate the distributional limits, we also establish consistency of the nonparametric bootstrap. 
Finally, we use the limit distribution theory to study applications to generative modeling via minimum distance estimation with the smooth Wasserstein distance, showing asymptotic normality of optimal solutions for the quadratic cost. 
\end{abstract}

\keywords{Dual Sobolev space, Gaussian smoothing, functional delta method, limit distribution theory, Wasserstein distance}
\subjclass{60F05, 62E20, and 62G09}

\maketitle

\section{Introduction}

\subsection{Overview}
The Wasserstein distance is an instance of the Kantorovich optimal transport problem \cite{kantorovich1942translocation}, which defines a metric on a space of probability measures. Specifically, for $1 \le p < \infty$, the $p$-Wasserstein distance between two Borel probability measures $\mu$ and $\nu$ on $\R^d$ with finite $p$th moments is defined by 
\begin{equation}\label{eq:OT-problem}
\Wp(\mu,\nu) = \inf_{\pi \in \Pi(\mu,\nu)} \left [ \int_{\R^d \times \R^d} |x-y|^p \, d \pi(x,y) \right]^{1/p},
\end{equation}
where  $\Pi(\mu,\nu)$ is the set of couplings (or transportation plans) of $\mu$ and $\nu$. The Wasserstein distance has seen a surge of applications in statistics, machine learning, and applied mathematics, ranging from generative modeling \cite{arjovsky_wgan_2017, gulrajani2017improved, tolstikhin2018wasserstein}, image recognition \cite{rubner2000earth,sandler2011nonnegative}, and domain adaptation \cite{courty2014domain,courty2016optimal} to robust optimization \cite{gao2016distributionally,mohajerin_robust_2018,blanchet2018optimal} and partial differential equations \cite{jordan1998variational,santambrogio2017}. The widespread applicability of the Wasserstein distance is driven by an array of desirable properties, including its metric structure ($\Wp$ metrizes weak convergence plus convergence of $p$th moments), a convenient dual form, robustness to support mismatch, and a rich geometry it induces on a space of probability measures. We refer to \cite{villani2003,villani2008optimal,ambrosio2005,santambrogio2015} as standard references on optimal transport theory.

However, statistical aspects of Wasserstein distances are bottlenecked by the curse of dimensionality, whereby the number of data points needed to accurately estimate them grows exponentially with dimension. Specifically, for the empirical distribution $\hat{\mu}_n$ of $n$ independent observations from a distribution $\mu$ on $\R^d$, it is known that $\E\big[\Wp(\hat{\mu}_n,\mu)\big]$ scales as $n^{-1/d}$ for $d > 2p$ under moment conditions  \cite{dereich2013,boissard2014mean,fournier2015,weed2019rate, lei2020convergence}. This slow rate renders performance guarantees in terms of $\Wp$ all but vacuous when $d$ is large. It is also a  roadblock towards a more delicate statistical analysis concerning limit distributions, bootstrap, and valid inference. %

Gaussian smoothing was recently introduced as a means to alleviate the curse of dimensionality of empirical $\Wp$ \cite{goldfeld2019,Goldfeld2020GOT,goldfeld2020asymptotic,sadhu2021,nietert21}. For $\sigma > 0$, the smooth $p$-Wasserstein distance is defined as $\GWp(\mu,\nu) := \Wp(\mu*\gamma_\sigma,\nu*\gamma_\sigma)$, where $*$ denotes convolution and $\gamma_\sigma = N(0,\sigma^2 I_d)$ is the isotropic Gaussian distribution with variance parameter $\sigma^2$. For sufficiently sub-Gaussian $\mu$, \cite{goldfeld2019} showed that the expected smooth distance between $\hat{\mu}_n$ and $\mu$ exhibits the parametric convergence rate, i.e., $\E\big[\gwass(\hat{\mu}_n,\mu)\big] = O(n^{-1/2})$ in any dimension. This is a significant departure from the $n^{-1/d}$ rate in the unsmoothed case. \cite{Goldfeld2020GOT} further showed that $\gwass$ maintains the metric and topological structure of $\mathsf{W}_1$ and is able to approximate it within a $\sigma\sqrt{d}$ gap. The structural properties and fast empirical convergence rates were later extended to $p>1$ in \cite{nietert21}. Other follow-up works explored relations between $\GWp$ and maximum mean discrepancies \cite{zhang2021convergence}, analyzed its rate of decay as $\sigma\to \infty$ \cite{chen2021asymptotics}, and adopted it as a performance metric for nonparametric mixture model estimation \cite{han2021nonparametric}.

A limit distribution theory for $\gwass$ was developed in \cite{goldfeld2020asymptotic,sadhu2021}, where the scaled empirical distance $\sqrt{n}\,\gwass(\hat{\mu}_n,\mu)$ was shown to converge in distribution to the supremum of a tight Gaussian process in every dimension $d$ under mild moment conditions. This result relies on the dual formulation of $\mathsf{W}_1$ as an integral probability metric (IPM) over the class of 1-Lipschitz functions. Gaussian smoothing shrinks the function class to that of 1-Lipschitz functions convolved with a Gaussian density, which is shown to be $\mu$-Donsker in every dimension, thereby yielding the limit distribution. Extending these results to empirical $\GWp$ with $p>1$, however, requires substantially new ideas due to the lack of an IPM structure. Consequently, works exploring $\GWp$ with $p>1$, such as \cite{zhang2021convergence,nietert21}, did not contain limit distribution results for it and this question remained largely open. %

The present paper closes this gap and provides a comprehensive limit distribution theory for empirical $\GWp$ with $p > 1$. Our main limit distribution results are summarized in the following theorem, where the `null' refers to when $\mu=\nu$, while `alternative' corresponds to $\mu\neq\nu$. In all what follows, the dimension $d \ge 1$ is arbitrary.

\begin{theorem}[Main results]
\label{thm: main theorem}
Let $1 < p < \infty$, and $\mu,\nu$ be Borel probability measures on~$\R^d$ with finite $p$th moments. Let $\hat{\mu}_n = n^{-1}\sum_{i=1}^n \delta_{X_i}$ and $\hat{\nu}_n = n^{-1}\sum_{i=1}^n \delta_{Y_i}$ be the empirical distributions of independent observations $X_1,\ldots,X_n \sim \mu$ and $Y_1,\ldots,Y_n\sim\nu$. Suppose that~$\mu$ satisfies Condition \eqref{eq:moment-condition} ahead (which requires $\mu$ to be sub-Gaussian). 

\begin{enumerate}[(i)]
    \item (One-sample null case) We have
    \[
    \sqrt{n}\GWp (\hat{\mu}_n,\mu) \stackrel{d}{\to} \sup_{\substack{\varphi \in C_0^{\infty}:\\ \| \varphi \|_{\dot{H}^{1,q}(\mu*\gamma_{\sigma})} 
    \le 1}} G_{\mu}(\varphi), 
    \]
    where $G_{\mu}=(G_{\mu}(\varphi))_{\varphi \in C_{0}^\infty}$ is a centered Gaussian process whose paths are linear and continuous with respect to (w.r.t.) the Sobolev seminorm $\| \varphi \|_{\mspace{-1mu}\dot{H}^{1,q}(\mu*\gamma_{\sigma})}\mspace{-5mu} :=\mspace{-3mu} \| \nabla \varphi \|_{L^q(\mu*\gamma_{\sigma};\R^d)}$. Here $q$ is the conjugate index of $p$, i.e., $1/p+1/q=1$. 
    
    \vspace{2mm}
    \item (Two-sample null case) If $\nu = \mu$, then we have 
    \[
    \sqrt{n}\GWp (\hat{\mu}_n,\hat{\nu}_n) \stackrel{d}{\to}\sup_{\substack{\varphi \in C_0^{\infty}:\\ \| \varphi \|_{\dot{H}^{1,q}(\mu*\gamma_{\sigma})}
    \le 1}} \big [ G_{\mu}(\varphi)-G_{\mu}'(\varphi) \big],
    \]
    where $G_{\mu}'$ is an independent copy of $G_\mu$. 
    \vspace{2mm}
    \item (One-sample alternative case) If $\nu \ne \mu$ and $\nu$ is sub-Weibull, then we have 
    \[
    \sqrt{n}\big(\GWp(\hat{\mu}_n,\nu) - \GWp(\mu,\nu)\big) \stackrel{d}{\to} N \left ( 0, \frac{\Var_{\mu}(g*\phi_{\sigma})}{p^2\big[\GWp(\mu,\nu)\big]^{2(p-1)}} \right ),
    \]
    where $g$ is an optimal transport potential from $\mu*\gamma_{\sigma}$ to $\nu*\gamma_{\sigma}$ for $\Wp^p$, and $\phi_\sigma(x) = (2\pi \sigma^2)^{-d/2}e^{-|x|^2/(2\sigma^2)}$ is the Gaussian density.
    \vspace{2mm}
    \item[(iv)] (Two-sample alternative case)  If $\nu \ne \mu$  and $\nu$ satisfies Condition \eqref{eq:moment-condition}, then we have 
    \[
    \sqrt{n}\big(\GWp(\hat{\mu}_n,\hat{\nu}_n) - \GWp(\mu,\nu)\big) \stackrel{d}{\to} N \left ( 0, \frac{\Var_{\mu}(g*\phi_{\sigma})+\Var_{\nu}(g^c*\phi_{\sigma})}{p^2\big[\GWp(\mu,\nu)\big]^{2(p-1)}} \right ),
    \]
    where $g^c$ is the $c$-transform of $g$ for the cost function $c(x,y)=|x-y|^p$. 
\end{enumerate}
\end{theorem}

Parts (i) and (ii) show that the null limit distributions are non-Gaussian. On the other hand, Parts (iii) and (iv) establish asymptotic normality of empirical $\GWp$ under the alternative. Notably, these result have the correct centering, $\GWp(\mu,\nu)$, which enables us to construct confidence intervals for $\GWp(\mu,\nu)$. 

The proof strategy for \cref{thm: main theorem} differs from existing approaches to limit distribution theory for empirical $\Wp$ for general distributions. In fact, an analog of \cref{thm: main theorem} is not known to hold for classic $\Wp$ in this generality, except for the special case where $\mu,\nu$ are discrete (see a literature review below for details). The key insight is to regard $\GWp$ as a functional defined on a subset of a certain dual Sobolev space. We show that the smooth empirical process converges weakly in the dual Sobolev space and that $\GWp$ is Hadamard (directionally) differentiable w.r.t. the dual Sobolev norm. We then employ the extended functional delta method \cite{shapiro1990,romisch2004} to obtain the limit distribution of one- and two-sample empirical $\GWp$ under both the null and the alternative.

The limit distributions in Theorem \ref{thm: main theorem} are non-pivotal in the sense that they depend on the population distributions $\mu$ and $\nu$, which are unknown in practice. To facilitate statistical inference using $\GWp$, we employ the bootstrap to estimate the limit distributions and prove its consistency for each case of \cref{thm: main theorem}. Under the alternative, the consistency follows from the linearity of the Hadamard derivative. Under the null, where the Hadamard (directional) derivative is nonlinear, the bootstrap consistency is not obvious but still holds. This is somewhat surprising in light of \cite{dumbgen1993,fang2019}, where it is demonstrated that the bootstrap, in general, fails to be consistent for functionals whose Hadamard directional derivatives are nonlinear (cf. Proposition 1 in \cite{dumbgen1993} or Corollary 3.1 in \cite{fang2019}). Nevertheless, our application of the bootstrap differs from \cite{dumbgen1993,fang2019} so there is no contradiction, and the specific structure of the Hadamard derivative of $\GWp$ allows to establish consistency under the null (see the discussion after \cref{prop: bootstrap} for more details). These bootstrap consistency results enable constructing confidence intervals for $\GWp$ and using it to test the equality of distributions. 

As an application of the limit theory, we study implicit generative modeling under the minimum distance estimation (MDE) framework \cite{wolfowitz1957minimum,pollard_min_dist_1980,parr1980minimum}. MDE extends the maximum-likelihood principle beyond the KL divergence and applies to models supported on low-dimensional manifolds \cite{arjovsky_wgan_2017} (whence the KL divergence is not well-defined), as well as to cases when the likelihood function is intractable \cite{gourieroux1993indirect}. For MDE with $\GWp$, we establish limit distribution results for the optimal solution and the smooth $p$-Wasserstein error. Our results hold for arbitrary dimension, again contrasting the classic case where analogous distributional limits for MDE with $\Wp$ are known only for $p=d=1$ \cite{bernton2019}. Remarkably, when $p=2$, the Hilbertian structure of the underlying dual Sobolev space allows showing asymptotic normality of the MDE solution.

\subsection{Literature review}
Analysis of empirical Wasserstein distances, or more generally empirical optimal transport distances, has been an active research area in the statistics and probability theory literature. In particular, significant attention was devoted to rates of convergence and exact asymptotics 
\cite{dudley1969speed,ajtai1984optimal,talagrand1992matching,talagrand1994transportation,dobric1995asymptotics,de2002almost,bolley2007quantitative,dereich2013,barthe2013combinatorial,boissard2014mean,fournier2015,weed2019rate,ambrosio2019pde,ledoux2019,bobkov2019simple,lei2020convergence, chizat2020faster, manole2021, deb2021, manole2021plugin}. As noted before, the empirical Wasserstein distance suffers from the curse of dimensionality, namely, $\E[\Wp(\hat{\mu}_n,\mu)] = O(n^{-1/d})$ whenever $d > 2p$. This rate is known to be sharp in general \cite{dudley1969speed}. 
The recent work by \cite{chizat2020faster, manole2021} discovered that the rate can be improved under the alternative, namely, $\E\big[\big|\Wp(\hat{\mu}_n,\nu)-\Wp(\mu,\nu)\big|\big] = O(n^{-\alpha/d})$  for $d \ge 5$ if $\nu \ne \mu$, where $\alpha = p$ for $1 \le p < 2$ and $\alpha = 2$ for $2 \le p < \infty$. Their insight is to use the duality formula for $\Wp^p$ and exploit regularity of optimal transport potentials. \cite{manole2021} also derive matching minimax lower bounds up to log factors under some technical conditions.

Another central problem that has seen a rapid development is limit distribution theory for empirical Wasserstein distances. However, except for the two special cases discussed next, to the best of our knowledge, there is no proven analog of our Theorem \ref{thm: main theorem} for classic Wasserstein distances, i.e., a comprehensive limit distribution theory for empirical $\Wp$ that holds for general $d$ and $p$. The first case for which the limit distribution is well understood is when $d=1$. Then, $\Wp$ reduces to the $L^p$ distance between quantile functions for $1 \leq p < \infty$, and further simplifies to the $L^1$ distance between distribution functions when $p=1$. %
Building on such explicit expressions, \cite{delbarrio1999} and \cite{delbarrio2005} derived null limit distributions in $d=1$ for $p=1$ and $p=2$, respectively. More recently, under the alternative ($\mu \ne \nu$), \cite{delbarrio2019} derived a central limit theorem (CLT) when $d=1$ and $p \ge 2$. The second case where a limit distribution theory for empirical $\Wp$ is available is when $\mu,\nu$ are discrete. If the distributions are finitely discrete, i.e., $\mu = \sum_{j=1}^m r_j \delta_{x_j}$ and $\nu = \sum_{j=1}^{k}s_j \delta_{y_j}$ for two simplex vectors $r = (r_1,\dots,r_m)$ and $s=(s_1,\dots,s_k)$, then $\Wp(\mu,\nu)$ can be seen as a function of those simplex vectors $r$ and~$s$. Leveraging this, \cite{sommerfeld2018} applied the delta method to obtain limit distributions for empirical $\Wp$ in the finitely discrete case. An extension to countably infinite supports was provided in \cite{tameling2019}, while \cite{delbarrio2021semi} treated the semidiscrete case where $\mu$ is finitely discrete but $\nu$ is general. 

Except for these two special cases, limit distributions for Wasserstein distances are less understood. To avoid repetitions, we focus here our discussion on the one sample case. In \cite{delbarrio2019central}, a CLT for $\sqrt{n}\big(\mathsf{W}_2^2(\hat{\mu}_n,\nu)-\E\big[\mathsf{W}_2^2(\hat{\mu}_n,\nu)\big]\big)$ is derived in any dimension, but the limit Gaussian distribution degenerates to $0$ when $\mu = \nu$; see also \cite{delbarrio2021} for an extension to general $1 < p < \infty$. 
Notably, the centering constant there is the expected empirical Wasserstein distance $\E\big[\mathsf{W}_2^2(\hat{\mu}_n,\nu)\big]$, which in general can not be replaced with the (more natural) population distance $\mathsf{W}_2^2(\mu,\nu)$. The recent preprint \cite{manole2021plugin} addressed this gap and established a CLT for $\sqrt{n}\big(\mathsf{W}_2^2(\tilde{\mu}_n,\nu)-\mathsf{W}_2^2(\mu,\nu)\big)$ for a wavelet-based estimator $\tilde{\mu}_n$ of $\mu$, assuming that the ambient space is $[0,1]^d$ and that $\mu,\nu$ are absolutely continuous w.r.t. the Lebesgue measure with smooth and strictly positive densities. Following arguments similar to  \cite{delbarrio2019central}, they first derive a CLT for $\sqrt{n}\big(\mathsf{W}_2^2(\tilde{\mu}_n,\nu)-\E\big[\mathsf{W}_2^2(\tilde{\mu}_n,\nu)\big]\big)$ and then use the strict positivity of the densities and higher order regularity of optimal transport potentials to control the bias term as  $\E\big[\mathsf{W}_2^2(\tilde{\mu}_n,\nu)\big] - \mathsf{W}_2^2(\mu,\nu) =o(n^{-1/2})$.   

Finally, we note that our proof techniques differ from the aforementioned arguments for classic $\Wp$. %
Specifically, as opposed to the two-step approach of \cite{manole2021plugin} described above, we directly prove asymptotic normality for $\sqrt{n}\big(\GWp(\hat{\mu}_n,\nu) - \GWp(\mu,\nu)\big)$ under the alternative. Their derivation does not apply to our case even when $p=2$ since their bias bound requires that the densities of $\mu$ and $\nu$ be bounded away from zero on their (compact) supports, which fails to hold after the Gaussian convolution. %
Our argument also differs from that of \cite{sommerfeld2018,tameling2019}, even though they also rely on the functional delta method. Specifically, since we do not assume that $\mu,\nu$ are discrete, $\GWp$ can not be parameterized by simplex vectors, and hence the application of the functional delta method is nontrivial. Very recently, an independent work \cite{hundrieser2022} used the extended functional delta method for the supremum functional \cite{carcamo2020directional} to derive limit distributions for classic $\Wp$, with $p\geq 2$, for compactly supported distributions under the alternative in dimensions $d\leq 3$.

\subsection{Organization}
The rest of the paper is organized as follows. 
In Section \ref{sec:background}, we collect background material on Wasserstein distances, smooth Wasserstein distances, and dual Sobolev spaces. In Section \ref{sec: main results}, we prove Theorem \ref{thm: main theorem} and explore the validity of the bootstrap for empirical $\GWp$. \cref{sec: mde} presents applications of our limit distribution theory to MDE with $\GWp$. Proofs for Section \ref{sec: main results} and \ref{sec: mde} can be found in \cref{sec: proofs}. \cref{sec:summary} provides concluding remarks and discusses future research directions. Finally, the Appendix contains additional proofs.

\subsection{Notation}
Let $| \cdot |$ and $\langle \cdot, \cdot \rangle$ denote the Euclidean norm and inner product, respectively. Let $B(x,r) = \{ y \in \R^{d} : |y-x| \le r \}$ denote the closed ball with center $x$ and radius $r$. We use $dx$ for the Lebesgue measure on $\RR^d$, while $\delta_x$ denotes the Dirac measure at $x \in \R^d$. Given a finite signed Borel measure $\ell$ on $\R^d$, we identify $\ell$ with the linear functional $f \mapsto \ell(f) := \int f d \ell$. Let $\lesssim$ denote inequalities up to some numerical constants. For any $a,b \in \R$, we use the shorthands $a \vee b = \max \{ a,b \}$ and $a \land b = \min \{ a,b \}$. 

For a topological space $S$, $\cB(S)$ and $\cP(S)$ denote, respectively, the Borel $\sigma$-field on $S$ and the class of Borel probability measures on $S$. We write $\cP := \cP(\R^d)$ and for $1 \le p < \infty$, use $\cP_p$ to denote the subset of $\mu \in \cP$ with finite $p$th moment $\int_{\R^d} |x|^{p} d \mu(x) < \infty$. For $\mu,\nu \in \cP$, we write $\mu\ll\nu$ for the absolute continuity of $\mu$ w.r.t. $\nu$, and use $d\mu/d\nu$ for the corresponding Radon-Nikodym derivative. Let $\mu*\nu$ denote the convolution of $\mu,\nu \in \cP$. Likewise, the convolution of two measurable functions $f,g: \R^d \to \R$ is denoted by $f*g$. We write $\gamma_{\sigma} = N(0,\sigma^2 I_d)$ for the centered Gaussian distribution on $\R^d$ with covariance matrix $\sigma^2 I_d$, and use $\phi_{\sigma}(x) = (2\pi\sigma^2)^{-d/2}e^{-|x|^2/(2\sigma^2)}$ ($x \in \R^d$) for the corresponding density. We write $\mu \otimes \nu$ for the product measure of $\mu,\nu \in \cP$. Let $\smash{\stackrel{w}{\to}, \stackrel{d}{\to}}$, and $\smash{\stackrel{\Prob}{\to}}$ denote weak convergence of probability measures, convergence in distribution of random variables, and convergence in probability, respectively. When necessary, convergence in distribution is understood in the sense of Hoffmann-J{\o}rgensen (cf. Chapter 2 in \cite{vandervaart1996book}).

Throughout the paper, we assume that $(X_1,Y_1),(X_2,Y_2),\dots$ are the coordinate projections of the product probability space $\prod_{i=1}^{\infty} \big(\R^{2d}, \cB(\R^{2d}),\mu \otimes \nu\big)$. To generate auxiliary random variables, we extend the probability space as $(\Omega, \cA, \Prob) = \left [ \prod_{i=1}^{\infty} \big(\R^{2d}, \cB(\R^{2d}),\mu \otimes \nu\big)\right] \times \big([0,1],\cB([0,1]), \mathrm{Leb}\big)$, where $\mathrm{Leb}$ denotes the Lebesgue measure on $[0,1]$. For $\beta \in (0,2]$, let $\psi_{\beta}(t) = e^{t^\beta}-1$ for $t \ge 0$, and recall that the corresponding Orlicz (quasi-)norm of a real-valued random variable $\xi$ is defined as $\| \xi \|_{\psi_{\beta}}:= \inf \{ C>0 : \E[\psi_{\beta}(|\xi|/C)] \le 1 \}$. A Borel probability measure $\mu\in\cP$ is called \textit{$\beta$-sub-Weibull} if $\| |X| \|_{\psi_{\beta}} < \infty$ for $X \sim \mu$. We say that $\mu$ is sub-Weibull if it is $\beta$-sub-Weibull for some $\beta \in (0,2]$. Finally, $\mu$ is \textit{sub-Gaussian} if it is $2$-sub-Weibull.

For an open set $\mathcal{O}$ in a Euclidean space,  $C_{0}^{\infty} (\mathcal{O})$  denotes the space of compactly supported, infinitely differentiable, real functions on $\mathcal{O}$. We write $C_0^{\infty} = C_0^{\infty}(\R^d)$ and define $\dot{C}_0^\infty = \{  f+a : f \in C_0^{\infty}, a \in \R \}$.
For any $p \in [1,\infty)$ and $\mu\in\cP(\RR^d)$, let $L^p(\mu;\R^k)$ denote the space of measurable maps $f: \R^d \to \R^k$ such that $\|f\|_{L^p(\mu;\R^k)} = (\int_{\R^d}|f|^p d \mu)^{1/p} < \infty$; when $d=1$ we use the shorthand $L^p(\mu) = L^p(\mu;\R^1)$. Recall that $\big(L^p(\mu;\R^k), \| \cdot \|_{L^p(\mu;\R^k)}\big)$ is a Banach space. Finally, for a subset $A$ of a topological space~$S$, let $\overline{A}^S$ denote the closure of $A$; if the space $S$ is clear from the context, then we simply write $\overline{A}$ for the closure.

\section{Background}
\label{sec:background}

\subsection{Wasserstein distances and their smooth variants}

Recall that, for $1 \le p < \infty$, the $p$-Wasserstein distance $\Wp(\mu,\nu)$ between $\mu,\nu \in \cP_p$ is defined in \eqref{eq:OT-problem}. Some basic properties of $\Wp$ are (cf. e.g., \cite{villani2003,ambrosio2005,villani2008optimal,santambrogio2015}):    
(i) the $\inf$ is attained in the definition of $\Wp$, i.e., there exists a coupling $\pi^\star \in \Pi(\mu,\nu)$ such that $\Wp^{p}(\mu,\nu) = \int_{\R^d \times \R^d} |x-y|^pd\pi^\star(x,y)$, and the optimal coupling $\pi^\star$ is unique when $p>1$ and $\mu \ll dx$;
(ii) $\Wp$ is a metric on $\cP_p$; and
(iii) convergence in $\Wp$ is equivalent to weak convergence plus convergence of $p$th moments: $\Wp(\mu_{n},\mu) \to 0$ if and only if $\mu_{n} \stackrel{w}{\to} \mu$ and $\int |x|^{p} d \mu_{n}(x) \to \int |x|^{p} d \mu(x)$. 

The proof of the limit distribution for empirical $\GWp$ under the alternative hinges on duality theory for $\Wp$, which we summarize below. For a function $g: \R^d \to [-\infty,\infty)$ and a cost function $c: \R^d \times \R^d \to \R$, the \textit{$c$-transform} of $g$ is defined by
\[
g^c(y) = \inf_{x \in \R^d}\big [ c(x,y) -g(x) \big], \quad y \in \R^d.
\]
A function $g: \R^d \to [-\infty,\infty)$, not identically $-\infty$, is called \textit{$c$-concave} if $g = f^c$ for some function $f: \R^d \to [-\infty,\infty)$.

\begin{lemma}[Duality for $\Wp$]
\label{lem: duality}
Let $1 \le p < \infty$, $\mu,\nu \in \cP_p$, and set the cost function to $c(x,y) = |x-y|^p$. 
\begin{enumerate}[(i)]
    \item(Theorem 5.9 in \cite{villani2008optimal}; Theorem 6.1.5 in \cite{ambrosio2005}) The following duality holds, 
\begin{equation}
\label{eq: duality}
\Wp^p(\mu,\nu) = \sup_{g \in L^1(\mu)} \left [ \int_{\R^d} g d\mu + \int_{\R^d} g^c d\nu \right ],
\end{equation}
and there is at least one $c$-concave function $g \in L^1(\mu)$ that attains the supremum in~\eqref{eq: duality}; we call this $g$ an \textit{optimal transport potential} from $\mu$ to $\nu$ for $\Wp^p$.
\item(Theorem 3.3 in \cite{gangbo1996}) Let $1 < p < \infty$, suppose that $g:\R^d \to [-\infty,\infty)$ is $c$-concave, and take $K$ as the convex hull of $\{ x : g (x) > -\infty \}$. Then $g$ is locally Lipschitz on the interior~of~$K$. 
\item(Corollary 2.7 in \cite{delbarrio2021}) If $1 < p < \infty$ and $\mu \ll dx$ is supported on an open connected set~$A$, then the optimal transport potential from $\mu$ to $\nu$ for $\Wp^p$ is unique on $A$ up to additive constants, i.e., if $g_1$ and $g_2$ are optimal transport potentials, then there exists $C \in \R$ such that $g_1(x) = g_2(x)+C$ for all $x \in A$. 
\end{enumerate}
\end{lemma}

The \textit{smooth Wasserstein distance} convolves the distributions with an isotropic Gaussian kernel. Gaussian convolution levels out local irregularities in the distributions, while largely preserving the structure of classic $\Wp$. Recalling that $\gamma_{\sigma} = N(0,\sigma^2I_d)$, the smooth $p$-Wasserstein distance is defined as follows.

\begin{definition}[Smooth Wasserstein distance]
Let $1 \le p < \infty$ and $\sigma \ge 0$. For $\mu,\nu \in \cP_{p}$,
the \emph{smooth $p$-Wasserstein distance} between $\mu$ and $\nu$ with smoothing parameter $\sigma$ is%
\[
\GWp(\mu,\nu) := \Wp(\mu*\gamma_{\sigma}, \nu*\gamma_{\sigma}).
\]
\end{definition}

The smooth Wasserstein distance was studied in \cite{goldfeld2019,Goldfeld2020GOT,goldfeld2020asymptotic,sadhu2021,nietert21} for structural properties and empirical convergence rates. We recall two basic properties: (i) $\GWp$ is a metric on $\cP_p$ that generates the same topology as classic $\Wp$; (ii) for $\mu,\nu \in \cP_p$ and $0 \le \sigma_1 \leq \sigma_2 < \infty$, we have $\Wp^{(\sigma_2)}(\mu,\nu) \leq \Wp^{(\sigma_1)}(\mu,\nu) \leq \Wp^{(\sigma_2)}(\mu,\nu) +C_{p,d}\sqrt{\sigma_2^2 - \sigma_1^2}$ for a constant $C_{p,d}$ that depends only on $p,d$. In particular, $\GWp(\mu,\nu)$ is continuous and monotonically non-increasing in $\sigma \in [0,+\infty)$ with $\lim_{\sigma \downarrow 0} \GWp(\mu,\nu) = \Wp(\mu,\nu)$. See \cite{nietert21} for additional structural results, including an explicit expression for $C_{p,d}$ and weak convergence of smooth optimal couplings. %
For empirical convergence, it was shown in \cite{nietert21} that under appropriate moment assumptions $\EE\big[\GWp(\hat{\mu}_n,\mu)\big]=O(n^{-1/2})$ for $p >1$ in any dimension $d$. Versions of this result for $p=1$ and $p=2$ were derived earlier in \cite{goldfeld2019,goldfeld2020asymptotic,sadhu2021}.

\subsection{Sobolev spaces and their duals}

Our proof strategy for the limit distribution results is to regard $\Wp$ as a functional defined on a subset of a certain dual Sobolev space. We will show that the smooth empirical process converges weakly in the dual Sobolev space and that $\Wp$ is Hadamard (directionally) differentiable w.r.t. the dual Sobolev norm. Given these, the limit distributions in \cref{thm: main theorem} follow via the functional delta method. Here we briefly discuss (homogeneous) Sobolev spaces and their duals.

\begin{definition}[Homogeneous Sobolev spaces and their duals]\label{def:Sobolev}
Let $\rho\mspace{-1.5mu}\in\mspace{-1.5mu}\cP$ and $1\mspace{-1.5mu}\le\mspace{-1.5mu} p\mspace{-1.5mu}<\mspace{-1.5mu}\infty$. 
\begin{enumerate}[(i)]
\item For a differentiable function $f:\R^{d}\to\R$, let
\[
\| f \|_{\dot{H}^{1,p}(\rho)}:= \| \nabla f \|_{L^{p}(\rho;\R^d)} = \left (\int_{\R^{d}} |\nabla f |^p  d\rho \right)^{1/p}
\]
be the \textit{Sobolev seminorm}. 
We define the \textit{homogeneous Sobolev space} $\dot{H}^{1,p}(\rho)$ by the completion of $\dot{C}_{0}^\infty$ w.r.t. $\| \cdot \|_{\dot{H}^{1,p}(\rho)}$. 
\item Let $q$ be  the conjugate index of $p$, i.e., $1/p+1/q = 1$.
Let $\dot{H}^{-1,p}(\rho)$ denote the topological dual of $\dot{H}^{1,q}(\rho)$.
The \emph{dual Sobolev norm} $\| \cdot \|_{\dot{H}^{-1,p}(\rho)}$ (dual to $\| \cdot \|_{\dot{H}^{1,q}(\rho)}$) of a continuous linear functional $\ell:\dot{H}^{1,q}(\rho) \to \R$ is defined by
\[
\| \ell \|_{\dot{H}^{-1,p}(\rho)} = \sup \left\{  \ell (f) : f \in \dot{C}_{0}^{\infty}, \| f \|_{\dot{H}^{1,q}(\rho)} \le 1 \right\}.
\]
\end{enumerate}
\end{definition}

The restriction $f \in \dot{C}_0^\infty$ can be replaced with $f \in C_0^\infty$ in the definition of the dual norm $\| \cdot \|_{\dot{H}^{-1,p}(\rho)}$ since $\ell (f+a) = \ell (f)$ for any $\ell \in \dot{H}^{-1,p}(\rho)$.

\medskip

We have defined the homogeneous Sobolev space $\dot{H}^{1,p}(\rho)$ as the completion of $\dot{C}_{0}^\infty$ w.r.t. $\| \cdot \|_{\dot{H}^{1,p}(\rho)}$. It is not immediately clear that the so-constructed space is a function space over $\R^d$. Below we present an explicit construction of $\dot{H}^{1,p}(\rho)$ %
when $d\rho/d\kappa$ is bounded away from zero for some reference measure $\kappa \gg dx$ satisfying the $p$-Poincar\'{e} inequality. To that end, we first define the Poincar\'e inequality.

\begin{definition}[Poincar\'{e} inequality]
\label{def: Poincare}
For $1 \le p < \infty$, a probability measure $\kappa \in \cP$ is said to satisfy the \textit{$p$-Poincar\'{e} inequality} if there exists a finite constant $\mathsf{C}$ such that 
\[
\| \varphi - \kappa(\varphi) \|_{L^{p}(\kappa)} \le \mathsf{C} \| \nabla \varphi \|_{L^{p}(\kappa; \R^d)}, \quad \forall \varphi \in C_{0}^{\infty}.
\label{eq:Poincare}
\]
The smallest constant satisfying the above is denoted by $\mathsf{C}_{p}(\kappa)$. 
\end{definition}

The standard Poincar\'{e} inequality refers to the $2$-Poincar\'{e} inequality. It is known that any log-concave distribution (i.e., a distribution $\kappa$ of the form $d\kappa = e^{-V} d x$ for some convex function $V:\R^d \to \R \cup \{ +\infty \}$; cf. \cite{lovasz2007,saumard2014}) satisfies the $p$-Poincar\'e inequality for any $1\leq p<\infty$ \cite{bobkov1999,milman2009}. In particular, the Gaussian distribution $\gamma_{\sigma}$ satisfies every $p$-Poincar\'e inequality (see also \cite[Corollary 1.7.3]{bogachev1998}).

\begin{remark}[Explicit construction of $\dot{H}^{1,p}(\rho)$]
\label{rem:explicit-construction}

Suppose that there exists a reference measure $\kappa \in \cP$, with $\kappa \gg dx$, that satisfies the $p$-Poincar\'{e} inequality.
Assume that $d\rho/d\kappa \ge c$ for some constant $c > 0$ (in our applications, $\rho = \gamma_{\sigma}$ or $\mu*\gamma_{\sigma}$ for some $\mu \in \cP_p$; in either case, the stated assumption is satisfied with $\kappa=\gamma_\sigma$ or $\gamma_{\sigma/\sqrt{2}}$).
Let $\cC=\{ f \in \dot{C}_{0}^{\infty} : \kappa(f) = 0 \}$. Then, $\| \cdot \|_{\dot{H}^{1,p}(\rho)}$ is a proper norm on $\cC$, and the map $\iota: f \mapsto \nabla f$ is an isometry from $(\cC,\| \cdot \|_{\dot{H}^{1,p}(\rho)})$ into $(L^{p}(\rho;\R^d),\| \cdot \|_{L^{p}(\rho;\R^d)})$. Let $V$ be the closure of $\iota \cC$ in $L^{p}(\rho;\R^d)$ under $\| \cdot \|_{L^{p}(\rho;\R^d)}$. The inverse map $\iota^{-1}: \iota \cC \to \cC$ can be extended to $V$ as follows. For any $g \in V$, choose $f_n \in \cC$ such that $\| \nabla f_n - g \|_{L^p(\rho;\R^d)} \to 0$. Since $\nabla f_n$ is Cauchy in $L^p(\rho;\R^d)$ and thus in $L^p(\kappa;\R^d)$ (as $\| \cdot \|_{L^p(\kappa;\R^d)} \lesssim \| \cdot \|_{L^p(\rho;\R^d)}$), $f_n$ is Cauchy in $L^{p}(\kappa)$ by the $p$-Poincar\'{e} inequality, so $\| f_n - f \|_{L^p(\kappa)} \to 0$ for some $f \in L^p(\kappa)$. Set $\iota^{-1} g = f$ and extend $\| \cdot \|_{\dot{H}^{1,p}(\rho)}$ by $\| f \|_{\dot{H}^{1,p}(\rho)} = \lim_{n \to \infty} \| f_n \|_{\dot{H}^{1,p}(\rho)} = \| g \|_{L^p(\rho;\R^d)}$.  The space $(\iota^{-1}V,\|\cdot \|_{\dot{H}^{1,p}(\rho)})$ is a Banach space of functions over $\R^d$.
Finally, the homogeneous Sobolev space $\dot{H}^{1,p}(\rho)$ can be constructed as $\dot{H}^{1,p}(\rho) = \big\{ f+a : a \in \R, f \in \iota^{-1}V \big\}$ with $\| f+a \|_{\dot{H}^{1,p}(\rho)} = \| f \|_{\dot{H}^{1,p}(\rho)}$. 
\end{remark}

The next lemma summarizes some basic results about the space $\dot{H}^{-1,p}(\rho)$ and $\dot{H}^{-1,p}(\rho)$-valued random variables that we use in the sequel. The proof can be found in \cref{sec: dual space}. 

\begin{lemma}
\label{lem: dual space}
Let $1 < p < \infty$ and $\rho \in \cP$. 
The dual space $\dot{H}^{-1,p}(\rho)$ is a separable Banach space. The Borel $\sigma$-field on $\dot{H}^{-1,p}(\rho)$ coincides with the cylinder $\sigma$-field (the smallest $\sigma$-field that makes the coordinate projections, $\dot{H}^{-1,p}(\rho)\ni \ell \mapsto \ell (f)\in \R$, measurable).
\end{lemma}

Consider a stochastic process $Y = (Y(f))_{f \in \dot{H}^{1,q}(\rho)}$ indexed by $\dot{H}^{1,q}(\rho)$, i.e.,  $\omega \mapsto Y(f,\omega)$ is measurable for each $f \in \dot{H}^{1,q}(\rho)$. The process can be thought of as a map from $\Omega$ into $\dot{H}^{-1,p}(\rho)$ as long as $Y$ has paths in $\dot{H}^{-1,p}(\rho)$, i.e., for each fixed $\omega \in \Omega$, the map $f \mapsto Y(f,\omega)$ is continuous and linear. The fact that the Borel $\sigma$-field on $\dot{H}^{-1,p}(\rho)$ coincides with the cylinder $\sigma$-field guarantees that a stochastic process indexed by $\dot{H}^{1,q}(\rho)$ with paths in $\dot{H}^{-1,p}(\rho)$ is Borel measurable as a map from $\Omega$ into $\dot{H}^{-1,p}(\rho)$.

\subsection{\(\Wp\) and dual Sobolev norm}

In \cref{sec: main results}, we will explore limit distributions for empirical $\GWp$. One of the key technical ingredients there is a comparison of the Wasserstein distance with a certain dual Sobolev norm, which we present next. 
\begin{proposition}[Comparison between $\Wp$ and dual Sobolev norm; Theorem 5.26 in \cite{dolbeault2009}]
\label{prop:duality-in-Wp}
Let $1 < p  < \infty$, and suppose that $\mu_0, \mu_1 \in \cP_p$ with $\mu_0,\mu_1 \ll \rho$ for some reference measure $\rho \in \cP$.
Denote their respective densities by $f_i = d \mu_i/d \rho$, $i=0,1$. 
If $f_0$ or $f_1$ is bounded from below by some $c>0$, then
\begin{equation}
\label{eq:Wp-Sobolev-comparison}
\Wp(\mu_0,\mu_1) \leq p\, c^{-1/q} \, 
\| \mu_{1} - \mu_{0} \|_{\dot{H}^{-1,p}(\rho)}.
\end{equation}
\end{proposition}

\cref{prop:duality-in-Wp} follows directly from Theorem 5.26 of \cite{dolbeault2009}. Similar comparison inequalities appear in \cite{peyre2018, ledoux2019, weed2019}. We include a self-contained proof of \cref{prop:duality-in-Wp} in \cref{app: proof of comparison} as some elements of the proof are key to our derivation of the null limit distribution for empirical $\GWp$. The proof builds on the Benamou-Brenier dynamic formulation of optimal transport \cite{benamou2000}, which shows that $\Wp(\mu_0,\mu_1)$ is bounded from above by the length of any absolutely continuous path from $\mu_0$ to $\mu_1$ in $(\cP_p,\Wp)$. The dual Sobolev norm emerges as a bound on the length of the linear interpolation $t \mapsto t\mu_1 + (1-t)\mu_0$.

\section{Limit distribution theory}
\label{sec: main results}

The goal of this section is to establish \cref{thm: main theorem}. The proof relies on two key steps: (i) establish weak convergence of the smooth empirical process $\sqrt{n} (\hat{\mu}_n - \mu)*\gamma_{\sigma}$ in the dual Sobolev space $\dot{H}^{-1,p}(\mu*\gamma_{\sigma})$; and (ii) regard $\GWp$ as a functional defined on a subset of $\dot{H}^{-1,p}(\mu*\gamma_{\sigma})$ and characterize its Hadamard directional derivative w.r.t. the corresponding dual Sobolev norm. Given (i) and (ii), the limit distribution results follow from the functional delta method, and the asymptotic normality under the alternative further follows from linearity of the Hadamard directional derivative.

\subsection{Preliminaries}

Throughout this section, we fix  $1 < p < \infty$, take $q$ as the conjugate index of $p$, and let $\sigma>0$. %
For $\mu,\nu\in\cP_p$, let $X_1,\ldots,X_n \sim \mu$ and $Y_1,\ldots,Y_n \sim \nu$ be independent observations  and denote the associated empirical distributions by $\hat{\mu}_n:= n^{-1}\sum_{i=1}^n \delta_{X_i}$ and $\hat{\nu}_n:= n^{-1}\sum_{i=1}^n \delta_{Y_i}$, respectively. %

\subsubsection{Weak convergence of smooth empirical process in dual Sobolev spaces}
\label{sec: weak convergence}

The first building block of our limit distribution results is the following weak convergence of the smoothed empirical process $\sqrt{n} (\hat{\mu}_n - \mu)*\gamma_{\sigma}$ in $\dot{H}^{-1,p}(\gamma_{\sigma})$ and $\dot{H}^{-1,p}(\mu*\gamma_{\sigma})$. 

\begin{proposition}[Weak convergence of smooth empirical process]
\label{thm:limit-distribution}
\hspace{-2mm}Suppose that  $X\sim\mu$ satisfies
\begin{equation}
\label{eq:moment-condition}
\int_{0}^{\infty} e^{\frac{\theta r^2}{2\sigma^2}} \sqrt{\Prob (|X| > r)} d r < \infty \quad \text{for some $\theta > p-1$}. 
\end{equation}
Then, the smoothed empirical process $\sqrt{n} (\hat{\mu}_n - \mu)*\gamma_{\sigma}$ converges in distribution as $n \to \infty$ both in $\dot{H}^{-1,p}(\gamma_{\sigma})$ and $\dot{H}^{-1,p}(\mu*\gamma_{\sigma})$. The limit process in each case is a centered Gaussian process, indexed by $\dot{H}^{1,q}(\gamma_\sigma)$ or $\dot{H}^{1,q}(\mu*\gamma_\sigma)$, respectively, with covariance function $(f_1,f_2) \mapsto \Cov_{\mu} (f_1*\phi_{\sigma},f_2*\phi_{\sigma})$. Here $\Cov_{\mu}$ denotes the covariance under $\mu$. 
\end{proposition}

The proof of \cref{thm:limit-distribution} relies on the prior work \cite{nietert21} by a subset of the authors, where it was shown that the smoothed function class $\cF * \phi_\sigma = \{ f* \phi_{\sigma} : f \in \cF \}$ with $\cF =  \{  f \in \dot{C}_0^\infty : \|f\|_{\dot{H}^{1,q}(\gamma_{\sigma})} \leq 1 \}$ is $\mu$-Donsker. The weak convergence in $\dot{H}^{-1,p}(\gamma_\sigma)$ then follows from a similar argument to Lemma 1 in \cite{nickl2009}. This, in turn, implies weak convergence in $\dot{H}^{-1,p}(\mu*\gamma_\sigma)$ when $\mu$ has mean zero, since in that case $\dot{H}^{-1,p}(\gamma_{\sigma})$ is continuously embedded into $\dot{H}^{-1,p}(\mu*\gamma_\sigma)$. To account for non-centered distributions, we use a reduction to the mean zero case via translation. See also \cref{rem: alternative} for an alternative proof for $p=2$ that relies on the CLT in the Hilbert space.
\medskip

Inspection of the proof of \cref{thm:limit-distribution} shows that Condition \eqref{eq:moment-condition} implies 
\[
\int_{\R^d} e^{(p-1)|x|^2/\sigma^2} d\mu(x) < \infty, 
\]
which requires $\mu$ to be sub-Gaussian. It is not difficult to see that Condition \eqref{eq:moment-condition} is satisfied if $\mu$ is compactly supported or sub-Gaussian with $\| |X| \|_{\psi_2} < \sigma/\sqrt{p-1}$ for $X \sim \mu$.

A natural question is whether a condition in the spirit of \eqref{eq:moment-condition} is necessary for the conclusion of \cref{thm:limit-distribution} to hold. Indeed, we show that some form of sub-Gaussianity is necessary for the smooth empirical process to converge to zero in $\dot{H}^{-1,p}(\gamma_{\sigma})$.

\begin{proposition}[Necessity of sub-Gaussian condition]
\label{prop:glivenko-cantelli}
The following hold. 
\begin{enumerate}[(i)]
    \item If $(\hat{\mu}_n - \mu)*\gamma_{\sigma} \to 0$ in $\dot{H}^{-1,p}(\gamma_{\sigma})$ as $n \to \infty$ a.s., then 
$\int_{\R^d} e^{\theta |x|^2/(2\sigma^2)} d \mu(x) < \infty$
for any $\theta < p-1$.
\item Conversely, if  $\int_{\R^d} e^{(p-1) |x|^2/(2\sigma^2)} d \mu(x) < \infty$, then $(\hat{\mu}_n - \mu)*\gamma_{\sigma} \to 0$ in $\dot{H}^{-1,p}(\gamma_{\sigma})$ as $n \to \infty$ a.s.
\end{enumerate} 
\end{proposition}

\subsubsection{Functional delta method}

Another ingredient of our limit distribution results is the (extended) functional delta method \cite{shapiro1991,dumbgen1993,romisch2004,fang2019}. 
Let $D$ be a normed space and $\Phi: \Xi \subset D \to \R$ be a function. Following \cite{shapiro1990,romisch2004}, we say that $\Phi$ is 
\textit{Hadamard directionally differentiable}  at $\theta \in \Xi$ if there exists a map $\Phi'_{\theta}: T_{\Xi}(\theta) \to \R$ such that 
\[
\lim_{n \to \infty} \frac{\Phi(\theta + t_n h_n) - \Phi(\theta)}{t_n} = \Phi'_{\theta}(h)
\]
for any $h \in T_{\Xi}(\theta)$, $t_n \downarrow 0$, and 
$h_n \to h$ in $D$ such that $\theta + t_n h_n \in \Xi$. Here $T_{\Xi}(\theta)$ is the \textit{tangent cone} to $\Xi$ at $\theta$ defined as 
\[
T_{\Xi}(\theta):= \left \{ h \in D : h = \lim_{n \to \infty} \frac{\theta_n - \theta}{t_n}  \ \text{for some $\theta_n \to \theta$ in $\Xi$ and $t_n \downarrow 0$} \right \}. 
\]
The tangent cone $T_{\Xi}(\theta)$ is closed, and if $\Xi$ is convex, then $T_{\Xi}(\theta)$ coincides with the closure in $D$ of $\{ (\tilde{\theta} - \theta)/t: \tilde{\theta} \in \Xi, t > 0 \}$ (cf. Proposition 4.2.1 in \cite{aubin2009}). The derivative $\Phi_{\theta}'$ is positively homogeneous (i.e., $\Phi_{\theta}'(ch) = c\Phi_{\theta}'(h)$ for any $c\geq 0$) and continuous, but need not be linear.

\begin{lemma}[Extended functional delta method \cite{shapiro1991,dumbgen1993,romisch2004,fang2019}]
\label{lem: functional delta method}
Let $D$ be a normed space and $\Phi: \Xi \subset D \to \R$ be a function that is Hadamard directionally differentiable at $\theta \in \Xi$ with derivative $\Phi_{\theta}': T_{\Xi}(\theta) \to \R$. Let $T_n: \Omega \to \Xi$ be maps such that $r_n (T_n - \theta) \stackrel{d}{\to} T$ for some $r_n \to \infty$ and Borel measurable map $T: \Omega \to D$ with values in $T_{\Xi}(\theta)$. Then, $r_n\big(\Phi(T_n) - \Phi(\theta)\big) \stackrel{d}{\to} \Phi_{\theta}'(T)$. Further, if $\Xi$ is convex, then we have the expansion  $r_n\big(\Phi(T_n) - \Phi(\theta)\big) = \Phi_{\theta}'\big(r_n(T_n-\theta)\big)   + o_{\Prob}(1)$. 
\end{lemma}

\begin{remark}[Choice of domain $\Xi$] \hspace{-2mm}The domain $\Xi$ is arbitrary as long as it contains the ranges of $T_n$ for all $n$, and the tangent cone $T_{\Xi}(\theta)$ contains the range of the limit variable~$T$.
\end{remark}

\subsection{Limit distributions under the null (\(\mu = \nu\))}
\label{sec: null case}

We shall apply the extended functional delta method to derive the limit distributions of $\sqrt{n}\GWp(\hat{\mu}_n,\mu)$ and $\sqrt{n}\GWp(\hat{\mu}_n,\hat{\nu}_n)$ as $n \to \infty$, namely, proving Parts (i) and (ii) of Theorem \ref{thm: main theorem}. To set up the problem over a (real) vector space, we regard $\rho \mapsto \GWp(\rho,\mu) = \Wp(\rho*\gamma_\sigma,\mu*\gamma_\sigma)$ as a map $h \mapsto \Wp(\mu*\gamma_\sigma+h,\mu*\gamma_{\sigma})$ defined on a set of finite signed Borel measures. %
The comparison result from \cref{prop:duality-in-Wp} implies that the latter map is Lipschitz in $\| \cdot \|_{\dot{H}^{-1,p}(\mu*\gamma_\sigma)}$, and \cref{thm:limit-distribution} shows that $\sqrt{n}(\hat{\mu}_n-\mu)*\gamma_\sigma$ is weakly convergent in $\dot{H}^{-1,p}(\mu*\gamma_\sigma)$. These suggest choosing the ambient space to be $\dot{H}^{-1,p}(\mu*\gamma_\sigma)$.

To cover the one- and two-sample cases in a unified manner, consider the same map but in two variables. Take $D_{\mu} = \dot{H}^{-1,p}(\mu*\gamma_{\sigma})$, set $\Xi_{\mu}:= D_{\mu} \cap \big \{ h = (\rho - \mu)*\gamma_\sigma : \rho \in \cP_p \big\}$, and define the function $\Phi: \Xi_{\mu} \times \Xi_\mu \subset D_{\mu} \times D_\mu \to \R$ as
\[
\Phi (h_1,h_2):= \Wp (\mu*\gamma_{\sigma}+h_1,\mu*\gamma_{\sigma}+h_2), \quad (h_1,h_2) \in \Xi_{\mu} \times \Xi_\mu.
\]
We endow $D_\mu \times D_\mu$ with a product norm (e.g., $\| h_1 \|_{\dot{H}^{-1,p}(\mu*\gamma_\sigma)} + \| h_2 \|_{\dot{H}^{-1,p}(\mu*\gamma_\sigma)}$). 
Since the set $\Xi_{\mu}$ (and thus $\Xi_\mu \times \Xi_\mu$) is convex, the tangent cone $T_{\Xi_{\mu}\times \Xi_\mu}(0,0)$ coincides with the closure in $D_{\mu} \times D_\mu$ of $\{ (h_1,h_2)/t : (h_1,h_2) \in \Xi_{\mu} \times \Xi_\mu, t > 0 \}$. We next verify that $\Phi$ is Hadamard directionally differentiable at $(0,0)$.

\begin{proposition}[Hadamard directional derivative of $\Wp$ under the null]
\label{prop: Haradard derivative null}
Let $1 < p < \infty$ and $\mu \in \cP_p$. %
Then, the map $\Phi: (h_1,h_2) \mapsto \Wp (\mu*\gamma_{\sigma}+h_1,\mu*\gamma_{\sigma}+h_2), \Xi_{\mu} \times \Xi_\mu \subset D_{\mu} \times D_\mu \to \R$, is Hadamard directionally differentiable at $(h_1,h_2)=(0,0)$ with derivative $\Phi'_{(0,0)}(h_1,h_2) = \| h_1-h_2 \|_{\dot{H}^{-1,p}(\mu*\gamma_{\sigma})}$, i.e., 
for any $(h_1,h_2) \in T_{\Xi_{\mu} \times \Xi_\mu}(0)$, $t_n \downarrow 0$ and $(h_{n,1},h_{n,2}) \to (h_1,h_2)$ in $D_{\mu}\times D_\mu$ such that $ (t_nh_{n,1},t_nh_{n,2}) \in \Xi_{\mu} \times \Xi_\mu$, we have 
\[
\lim_{n \to \infty} \frac{\Phi(t_n h_{n,1},t_nh_{n,2})}{t_n} = \| h_1-h_2 \|_{\dot{H}^{-1,p}(\mu*\gamma_{\sigma})}.
\]
\end{proposition}

Proposition \ref{prop: Haradard derivative null} follows from the next G\^{a}teaux differentiability result for $\Wp$, which may be of independent interest, combined with Lipschitz continuity of $\Phi$ w.r.t. $\| \cdot \|_{\dot{H}^{-1,p}(\mu*\gamma_{\sigma})}$ (cf. Proposition \ref{prop:duality-in-Wp}).

\begin{lemma}[G{\^a}teaux directional derivative of $\Wp$]
\label{lem: Haradard derivative null} 
Let $\mu \in \cP_p$ and $h_i \in \dot{H}^{-1,p}(\mu), i=1,2$ be finite signed Borel measures with total mass $0$ such that $h_i \ll \mu$ and $\mu+h_i \in \cP_p$. Then,
\[
\frac{d}{dt^+}\Wp(\mu+th_1,\mu+th_2) \Big|_{t=0} = \| h_1-h_2 \|_{\dot{H}^{-1,p}(\mu)},
\]
where $d/dt^+$ denotes the right derivative. 
\end{lemma}

\begin{remark}[Comparison with Exercise 22.20 in \cite{villani2008optimal}]
Exercise 22.20 in \cite{villani2008optimal} states that (in our notation)
\begin{equation}
\lim_{\epsilon \downarrow 0} \frac{\mathsf{W}_2\big((1+\epsilon h)\mu,\mu\big)}{\epsilon} = \| h\mu \|_{\dot{H}^{-1,2}(\mu)},
\label{eq: derivative dual}
\end{equation}
for any sufficiently regular function $h$ with $\int h d\mu = 0$ ($h\mu$ is understood as a signed measure $hd\mu$). Theorem 7.26 in  \cite{villani2003} provides a proof of the one-sided inequality that the liminf of the left-hand side above is at least $\| h \mu \|_{\dot{H}^{-1,2}(\mu)}$, when $\mu \in \cP_2$ satisfies $\mu \ll dx$ and $h$ is bounded. The subsequent Remark 7.27 states that ``We shall not consider the converse of this inequality, which requires more assumptions and more effort." However, we could not find references that establish rigorous conditions applicable to our problem under which the derivative formula \eqref{eq: derivative dual} holds. Lemma \ref{lem: Haradard derivative null} provides a rigorous justification for this formula and extends it to general $p > 1$.
\end{remark}

Given these preparations, the proof of Theorem \ref{thm: main theorem} Parts (i) and (ii) is immediate. 

\begin{proof}[Proof of Theorem \ref{thm: main theorem}, Parts (i) and (ii)]
Let $G_{\mu}$ denote the weak limit of $\sqrt{n}(\hat{\mu}_n-\mu)*\gamma_\sigma$ in $\Hu$; cf. \cref{thm:limit-distribution}. 
Recall that $D_\mu=\dot{H}^{-1,p}(\mu*\gamma_{\sigma})$ is separable (cf. Lemma \ref{lem: dual space}), so $(T_{n,1},T_{n,2}) :=\big((\hat{\mu}_n-\mu)*\gamma_{\sigma},(\hat{\nu}_n-\mu)*\gamma_{\sigma} \big)$ is a Borel measurable map from~$\Omega$ into the product space $D_\mu \times D_\mu$ \cite[Lemma 1.4.1]{vandervaart1996book}. Since $T_{n,1}$ and $T_{n,2}$ are independent, by Example 1.4.6 in \cite{vandervaart1996book} and  \cref{thm:limit-distribution}, $(T_{n,1},T_{n,2}) \stackrel{d}{\to} (G_\mu,G_\mu')$ in $D_\mu \times D_\mu$, where $G_\mu'$ is an independent copy of $G_\mu$. Since $(T_{n,1},T_{n,2}) \in T_{\Xi_\mu \times \Xi_\mu}(0,0)$ and $T_{\Xi_\mu \times \Xi_\mu}(0,0)$ is closed in $D_\mu \times D_\mu$, we see that $(G_\mu,G_\mu') \in T_{\Xi_\mu \times \Xi_\mu}(0,0)$ by the portmanteau theorem.

Applying the functional delta method (\cref{lem: functional delta method}) and \cref{prop: Haradard derivative null}, we conclude that 
\[
\sqrt{n}\GWp(\hat{\mu}_n,\hat{\nu}_n) = \sqrt{n}\big (\Phi (T_{n,1},T_{n,2})-\Phi(0,0)\big) \stackrel{d}{\to} \Phi_{(0,0)}'(G_{\mu},G_\mu') = \| G_{\mu}-G_\mu' \|_{\dot{H}^{-1,p}(\mu*\gamma_{\sigma})}.
\]
Likewise, we also have
\[
\sqrt{n}\GWp(\hat{\mu}_n,\mu) = \sqrt{n}\big (\Phi (T_{n,1},0)-\Phi(0,0)\big) \stackrel{d}{\to} \Phi_{(0,0)}'(G_{\mu},0) = \| G_{\mu} \|_{\dot{H}^{-1,p}(\mu*\gamma_{\sigma})}.
\]
This completes the proof. 
\end{proof}

\subsection{Limit distributions under the alternative \((\mu \ne \nu)\)}
\label{sec: alternative}

\subsubsection{One-sample case} We start from the simpler situation where $\nu$ is known and prove Part (iii) of Theorem \ref{thm: main theorem}. Our proof strategy is to first establish asymptotic normality of the $p$th power of $\GWp$, from which Part (iii) follows by applying the delta method for $s \mapsto s^{1/p}$. For notational convenience, define
\[
\SWp (\mu,\nu):= \big[\GWp(\mu,\nu)\big]^p,
\]
for which one-sample asymptotic normality under the alternative is stated next.

\begin{proposition}
\label{thm: limit distribution alternative}
Suppose that $\mu \in \cP$ satisfies Condition \eqref{eq:moment-condition}, $\nu \in \cP$ is sub-Weibull, and $\mu\neq\nu$. Let $g$ be an optimal transport potential from $\mu*\gamma_{\sigma}$ to $\nu*\gamma_{\sigma}$ for $\Wp^p$. Then, we have 
\[
\sqrt{n}\big(\SWp(\hat{\mu}_n,\nu) - \SWp(\mu,\nu)\big) \stackrel{d}{\to}  N\big(0,\Var_{\mu}(g*\phi_{\sigma})\big).
\]
\end{proposition}

We again use the functional delta method to prove this proposition, but with a slightly different setting. Set $D_{\mu}= \dot{H}^{-1,p}(\mu*\gamma_{\sigma})$ as before, and consider the function $\Psi: \Lambda_{\mu} \subset D_{\mu} \to \R$ defined by 
\[
\Psi (h) := \Wp^p(\mu*\gamma_{\sigma}+h,\nu*\gamma_{\sigma}), \quad h \in \Lambda_{\mu},
\]
where 
\begin{equation}
\Lambda_{\mu}:=D_{\mu} \cap  \big \{ h = (\rho-\mu)*\gamma_{\sigma} : \text{$\rho \in \cP$ is sub-Weibull} \big \}.
\label{eq: domain alternative}
\end{equation}
As long as $\mu$ is sub-Weibull (recall that Condition (\ref{eq:moment-condition}) requires $\mu$ to be sub-Gaussian), the set $\Lambda_{\mu}$ contains $0$. This set is also convex, and so the tangent cone $T_{\Lambda_{\mu}}(0)$ coincides with the closure in $D_{\mu}$ of $\{ h/t : h \in \Lambda_{\mu}, t > 0 \}$. The corresponding Hadamard directional derivative of $\Wp^p$ is given next.

\begin{proposition}[Hadamard directional derivative of $\Wp^p$ w.r.t. one argument]
\label{prop: Hadamard derivative alternative}
Let $1 < p < \infty$, and %
suppose that $\mu,\nu \in \cP$ are sub-Weibull. 
Let $g$ be an optimal transport potential from $\mu*\gamma_{\sigma}$ to $\nu*\gamma_{\sigma}$ for $\Wp^p$, which is uniquely determined up to additive constants (see Lemma~\ref{lem: duality}~(iii)). %
Then 
\begin{enumerate}[(i)]
    \item $g \in \dot{H}^{1,q}(\mu*\gamma_{\sigma})$, where $q$ is the conjugate index of $p$; and
    \item the map $\Psi: \Lambda_{\mu} \subset D_{\mu} \to \R, h \mapsto \Wp^p(\mu*\gamma_{\sigma}+h,\nu*\gamma_{\sigma})$, is Hadamard directionally differentiable at $h=0$  with derivative $\Psi_0'(h) = h(g)$, i.e., for any $h \in T_{\Lambda_{\mu}}(0), t_n \downarrow 0$, and $h_n \to h$ in $D_{\mu}$ such that $t_n h_n \in \Lambda_{\mu}$, we have 
\begin{equation}
\lim_{n \to \infty} \frac{\Psi(t_n h_n) - \Psi (0)}{t_n} = h(g).
\label{eq: derivative alternative}
\end{equation}
\end{enumerate}%
\end{proposition}

As in the null case, Part (ii) of Proposition \ref{prop: Hadamard derivative alternative} follows from the following G\^{a}teaux differentiability result for $\Wp^p$, %
combined with local Lipschitz continuity of $\Psi$ w.r.t. $\| \cdot \|_{\dot{H}^{-1,p}(\mu*\gamma_{\sigma})}$. 

\begin{lemma}[G\^{a}teaux directional derivative of $\Wp^p$ w.r.t. one argument]
\label{lem: Hadamard derivative alternative}
Let $1 < p < \infty$ and $\mu,\nu, \rho \in \cP$ be sub-Weibull. Let $g$ be an optimal transport potential from $\mu*\gamma_{\sigma}$ to $\nu$. Then
\[
\frac{d}{dt^+} \Wp^p\big((\mu+t(\rho-\mu))*\gamma_{\sigma},\nu\big)\Big|_{t=0} = \int_{\R^d} g d\big ((\rho-\mu)*\gamma_{\sigma} \big),
\]
where the integral on the right-hand side is well-defined and finite. 
\end{lemma}

\begin{remark}[Comparisons with Theorem 8.4.7 in \cite{ambrosio2005} and Theorem 5.24 in \cite{santambrogio2015}]
Theorem 8.4.7 in \cite{ambrosio2005} derives the following differentiabiliy result for $\Wp^p$. 
Let $\mu_t: I \to (\cP_p,\Wp)$ be an absolutely continuous curve for some open interval $I$, and let $v_t$ be an ``optimal'' velocity field satisfying the continuity equation for $\mu_t$ (see Theorem 8.4.7 in \cite{ambrosio2005} for the precise meaning). Then, for any $\nu \in \cP_p$, we have that 
\begin{equation}
\frac{d}{dt} \Wp^p(\mu_t,\nu) = \int_{\R^d \times \R^d} p |x-y|^{p-2} \langle x-y,v_t(x) \rangle d\pi_t(x,y)
\label{eq: generic differentiability}
\end{equation}
for almost every (a.e.) $t \in I$, where $\pi_t \in \Pi (\mu_t,\nu)$ is an optimal coupling for $\Wp(\mu_t,\nu)$. See also Theorem 5.24 in \cite{santambrogio2015}. Since %
\eqref{eq: generic differentiability} only holds for a.e. $t \in I$, while we need the (right) differentiability at a specific point, the result of \cite[Theorem 8.4.7]{ambrosio2005} (or  \cite[Theorem 5.24]{santambrogio2015}) does not directly apply to our problem. %
We overcome this difficulty by establishing regularity of optimal transport potentials (see \cref{lem: potential} ahead), for which Gaussian smoothing plays an essential role. 
\end{remark}

We are now ready to prove \cref{thm: limit distribution alternative} and obtain Part (iii) of Theorem \ref{thm: main theorem} combined with the delta method for the map $s \mapsto s^{1/p}$.

\begin{proof}[Proof of Proposition \ref{thm: limit distribution alternative}]
By \cref{thm:limit-distribution}, $T_n := (\hat{\mu}_n-\mu)*\gamma_{\sigma} \in \Lambda_{\mu}$ and $\sqrt{n}T_n \stackrel{d}{\to} G_{\mu}$ in~$D_{\mu}$. Also $G_{\mu} \in T_{\Lambda_{\mu}}(0)$ with probability one by the portmanteau theorem.  Applying the functional delta method (Lemma \ref{lem: functional delta method}) and Proposition \ref{prop: Hadamard derivative alternative}, we have
\[
\sqrt{n}\big(\SWp(\hat{\mu}_n,\nu) - \SWp(\mu,\nu)\big) = \sqrt{n}\big (\Psi (T_n)-\Psi(0)\big) \stackrel{d}{\to} 
G_{\mu}(g) \sim N\big (0,\Var_{\mu}(g*\phi_{\sigma}) \big),
\]
as desired. 
\end{proof}

\subsubsection{Two-sample case}
Finally, we consider the two-sample case and prove the following, from which Part (iv) of Theorem \ref{thm: main theorem} follows. 

\begin{proposition}
\label{thm: limit distribution alternative two sample}
Let $1 < p < \infty$. 
Suppose that $\mu,\nu \in \cP$ satisfy Condition (\ref{eq:moment-condition}) and $\nu \ne \mu$. Let $g$ be an optimal transport potential from $\mu*\gamma_{\sigma}$ to $\nu*\gamma_{\sigma}$ for $\Wp^p$. Then, we have 
    \[
\sqrt{n}\big(\SWp(\hat{\mu}_n,\hat{\nu}_n) - \SWp(\mu,\nu)\big) \stackrel{d}{\to} N\big( 0, \Var_{\mu}(g*\phi_{\sigma})+\Var_{\nu}(g^c*\phi_{\sigma})\big).
    \]
\end{proposition}

Set $D_{\mu} = \dot{H}^{-1,p}(\mu*\gamma_{\sigma})$ and $D_{\nu} = \dot{H}^{-1,p}(\nu*\gamma_{\sigma})$.  Consider the function $\Upsilon: \Lambda_{\mu} \times \Lambda_{\nu} \subset D_{\mu} \times D_{\nu} \to \R$ defined by   
\[
\Upsilon (h_1,h_2):= \Wp^p(\mu*\gamma_{\sigma}+h_1,\nu*\gamma_{\sigma}+h_2), \quad (h_1,h_2) \in \Lambda_{\mu} \times \Lambda_{\nu},
\]
where $\Lambda_{\mu}$ is given in (\ref{eq: domain alternative}) and $\Lambda_{\nu}$ is defined analogously. Here  we endow $D_{\mu} \times D_{\nu}$ with a product norm (e.g. $\| h_1 \|_{\dot{H}^{-1,p}(\mu*\gamma_\sigma)} + \| h_2 \|_{\dot{H}^{-1,p}(\nu*\gamma_\sigma)}$).

We note that if $g$ is an optimal transport potential from $\mu*\gamma_{\sigma}$ to $\nu*\gamma_{\sigma}$, then $g^c$ is an optimal transport potential from $\nu*\gamma_{\sigma}$ to $\mu*\gamma_{\sigma}$, as $g^{cc} = g$. With this in mind, Proposition \ref{prop: Hadamard derivative alternative} immediately yields the following proposition. %

\begin{proposition}[Hadamard directional derivative of $\Wp^p$ w.r.t. two arguments]
\label{prop: Hadamard derivative alternative two sample}
Let $1 < p < \infty$, and %
suppose that $\mu,\nu \in \cP$ are sub-Weibull.
Let $g$ be an optimal transport potential from $\mu*\gamma_{\sigma}$ to $\nu*\gamma_{\sigma}$ for $\Wp^p$. Then, $(g,g^c) \in \dot{H}^{1,q}(\mu*\gamma_\sigma) \times \dot{H}^{1,q}(\nu*\gamma_\sigma)$, and the map $\Upsilon: \Lambda_\mu \times \Lambda_\nu \subset D_\mu \times D_\nu \to \R, (h_1,h_2) \mapsto \Wp^p(\mu*\gamma_{\sigma}+h_1,\nu*\gamma_{\sigma}+h_2)$, is Hadamard directionally differentiable at $(h_1,h_2)=(0,0)$  with derivative $\Upsilon_{(0,0)}'(h_1,h_2) = h_1(g) + h_2(g^c)$ for $(h_1,h_2) \in T_{\Lambda_\mu \times \Lambda_\nu}(0,0)$. 
\end{proposition}

Given \cref{prop: Hadamard derivative alternative two sample}, the proof of \cref{thm: limit distribution alternative two sample} is analogous to that of \cref{thm: limit distribution alternative}, and is thus omitted for brevity. As before, Part (iv) of Theorem \ref{thm: main theorem} follows via the delta method for $s \mapsto s^{1/p}$.

\subsection{Bootstrap}
\label{sec: bootstrap}
The limit distributions in Theorem \ref{thm: main theorem} are non-pivotal, as they depend on the population distributions $\mu$ and/or $\nu$, which are unknown in practice. To overcome this and facilitate statistical inference using $\GWp$, we apply the bootstrap to estimate the limit distributions of empirical $\GWp$. 

We start from the one-sample case. Given the data $X_1,\dots,X_n$, let $X_1^B,\dots,X_n^B$ be an independent sample from $\hat{\mu}_n$, and set $\hat{\mu}_n^B:= n^{-1}\sum_{i=1}^n \delta_{X_i^B}$ as the bootstrap empirical distribution. Let $\Prob^B$ denote the conditional probability given $X_1,X_2,\dots$. The next proposition shows that the bootstrap consistently estimates the limit distribution of empirical $\GWp$ under both the null and the alternative.

\begin{proposition}[Bootstrap consistency: one-sample case]
\label{prop: bootstrap}
Suppose that $\mu$ satisfies Condition (\ref{eq:moment-condition}). 
\begin{enumerate}
    \item[(i)] (Null case) We have 
\begin{equation}
\sup_{t \ge 0} \left | \Prob^B \left ( \sqrt{n}\GWp(\hat{\mu}^B_n,\hat{\mu}_n) \le t \right ) - \Prob \left ( \| G_\mu \|_{\dot{H}^{-1,p}(\mu*\gamma_{\sigma})}\le t \right ) \right | \stackrel{\Prob}{\to} 0.
\label{eq: bootstrap consistency null}
\end{equation}
    \item[(ii)] (Alternative case) Assume in addition that $\nu$ is sub-Weibull with $\nu \ne \mu$. Let $\mathfrak{v}_1^2$ denote the asymptotic variance of $\sqrt{n}\big(\GWp(\hat{\mu}_n,\nu) - \GWp(\mu,\nu)\big)$ given in Part (iii) of Theorem~\ref{thm: main theorem}.  Then, we have
\[
\sup_{t \in \R} \left | \Prob^B \left ( \sqrt{n}\big(\GWp(\hat{\mu}^B_n,\nu) - \GWp(\hat{\mu}_n,\nu)\big) \le t \right ) - \Prob\big( N ( 0, \mathfrak{v}_1^2) \le t \big) \right | \stackrel{\Prob}{\to} 0.
\]
\end{enumerate}
\end{proposition}

Part (ii) of the proposition is not surprising given that the Haramard directional derivative of the function $\Psi$ in Proposition \ref{prop: Hadamard derivative alternative} is $\Psi_0' (h) = h(g)$, which is linear in $h$. Part~(i)~is less obvious since the function $h_1 \mapsto \Phi (h_1,0)$ from Proposition \ref{prop: Haradard derivative null} has a nonlinear Hadamard directional derivative, $\Psi_{(0,0)}'(h_1,0) = \| h_1 \|_{\dot{H}^{-1,p}(\mu*\gamma_\sigma)}$. Recall that \cite[Proposition 1]{dumbgen1993} or \cite[Corollary 3.1]{fang2019} show that the bootstrap is inconsistent for functionals with nonlinear derivatives, but these results do not collide with Part (i) of \cref{prop: bootstrap} since our application of the bootstrap differs from theirs. For instance, \cite[Proposition 1]{dumbgen1993} specialized to our setting states that the conditional law of $\sqrt{n}\big(\Phi(\hat{\mu}_n^B-\mu,0) - \Phi(\hat{\mu}_n-\mu,0)\big) = \sqrt{n}\big(\GWp (\hat{\mu}_n^B,\mu) - \GWp (\hat{\mu}_n,\mu)\big)$ does not converge weakly to $\| G_\mu \|_{\dot{H}^{-1,p}(\mu*\gamma_\sigma)}$ in probability. Heuristically, $\sqrt{n}\GWp(\hat{\mu}_n,\mu)$ is nonnegative while $\sqrt{n}\big(\GWp (\hat{\mu}_n^B,\mu) - \GWp (\hat{\mu}_n,\mu)\big)$ can be negative, so the conditional law of the latter cannot mimic the distribution of the former. %
Further, when $\mu$ is unknown, the conditional law of $\sqrt{n}\big ( \GWp (\hat{\mu}_n^B,\mu) - \GWp (\hat{\mu}_n,\mu)\big)$ is infeasible. The correct bootstrap analog for $\GWp(\hat{\mu}_n,\mu)$ is $\GWp(\hat{\mu}_n^B,\hat{\mu}_n) = \Phi (\hat{\mu}_n^B-\mu,\hat{\mu}_n-\mu)$, and the proof of \cref{prop: bootstrap} shows that it can be approximated by $\| \hat{\mu}_n^B-\mu - (\hat{\mu}_n - \mu) \|_{\dot{H}^{-1,p}(\mu*\gamma_\sigma)} = \| \hat{\mu}_n^B-\hat{\mu}_n \|_{\dot{H}^{-1,p}(\mu*\gamma_\sigma)}$, whose conditional law (after scaling) converges weakly to $\| G_\mu \|_{\dot{H}^{-1,p}(\mu*\gamma_\sigma)}$ in probability.

\medskip
Next, consider the two-sample case. In addition to $X_1^B,\dots,X_n^B$ and $\hat{\mu}_n^B$, given $Y_1,\dots,Y_n$, let $Y_1^B,\dots,Y_n^B$ be an independent sample from $\hat{\nu}_n$, and set $\hat{\nu}_n^B:= n^{-1}\sum_{i=1}^n \delta_{Y_i^B}$. Slightly abusing notation, we reuse $\Prob^B$ for the conditional probability given $(X_1,Y_1),(X_2,Y_2),\dots$. 

\begin{proposition}[Bootstrap consistency: two-sample under the alternative]
\label{prop: bootstrap case 2}
Suppose that $\mu$ and $\nu$ satisfy Condition (\ref{eq:moment-condition}) and $\mu \ne \nu$.
 Let $\mathfrak{v}_2^2$ denote the asymptotic variance of $\sqrt{n}\big(\GWp(\hat{\mu}_n,\hat{\nu}_n) - \GWp(\mu,\nu)\big)$ given in Part (iv) of Theorem \ref{thm: main theorem}.  Then, we have
\[
\sup_{t \in \R} \left | \Prob^B \left ( \sqrt{n}\big(\GWp(\hat{\mu}^B_n,\hat{\nu}_n^B) - \GWp(\hat{\mu}_n,\hat{\nu}_n)\big) \le t \right ) - \Prob \left ( N( 0, \mathfrak{v}_2^2) \le t \right ) \right | \stackrel{\Prob}{\to} 0.
\]
\end{proposition}

\begin{example}[Confidence interval for $\GWp$]
Consider constructing confidence intervals for $\Wp(\mu,\nu)$. For $\alpha \in (0,1)$, let $\hat{\zeta}_\alpha$ denote the conditional $\alpha$-quantile of $\GWp(\hat{\mu}^B_n,\hat{\nu}_n^B)$ given the data. Then, by \cref{prop: bootstrap case 2} above and Lemma 23.3 in \cite{vandervaart1998}, the interval
\[
\left[ 2\GWp(\hat{\mu}_n,\hat{\nu}_n) - \hat{\zeta}_{1-\alpha/2},2\GWp(\hat{\mu}_n,\hat{\nu}_n) - \hat{\zeta}_{\alpha/2} \right],
\]
contains $\GWp(\mu,\nu)$ with probability approaching $1-\alpha$.
\end{example}

For the two-sample case under the null, instead of separately sampling bootstrap draws from $\hat{\mu}_n$ and $\hat{\nu}_n$ (see \cref{REM:bootstrap_inconsistent} below), we use the pooled empirical distribution $\hat{\rho}_n = (2n)^{-1} \sum_{i=1}^n (\delta_{X_i} + \delta_{Y_i})$ (cf. Chapter 3.7 in \cite{vandervaart1996book}). Given $(X_1,Y_1),\dots,(X_n,Y_n)$, let $Z_1^B,\dots,$ $Z_{2n}^B$ be an independent sample from $\hat{\rho}_n$, and set 
\[
\hat{\rho}_{n,1}^B = \frac{1}{n}\sum_{i=1}^n \delta_{Z_i^B} \quad \text{and} \quad \hat{\rho}_{n,2}^B = \frac{1}{n}\sum_{i=n+1}^{2n} \delta_{Z_i^B}. 
\]
The following proposition shows that this two-sample bootstrap is consistent for the null limit distribution of empirical $\GWp$. 
\begin{proposition}[Bootstrap consistency: two-sample under the null]
\label{prop: bootstrap case 3}
Suppose that $\mu$ and $\nu$ satisfy Condition (\ref{eq:moment-condition}). 
Then, for $\rho = (\mu+\nu)/2$, we have
\[
\sup_{t \ge 0} \left | \Prob^B \left ( \sqrt{n}\GWp(\hat{\rho}^B_{n,1},\hat{\rho}_{n,2}^B)  \le t \right ) - \Prob \left ( \| G_{\rho}-G_\rho' \|_{\dot{H}^{-1,p}(\rho*\gamma_{\sigma})} \le t \right ) \right | \stackrel{\Prob}{\to} 0,
\]
where $G_\rho'$ is an independent copy of $G_\rho$. In particular, if $\mu = \nu$, then 
\[
\sup_{t \ge 0} \left | \Prob^B \left ( \sqrt{n}\GWp(\hat{\rho}^B_{n,1},\hat{\rho}_{n,2}^B)  \le t \right ) - \Prob \left ( \| G_{\mu}-G_\mu' \|_{\dot{H}^{-1,p}(\mu*\gamma_{\sigma})} \le t \right ) \right | \stackrel{\Prob}{\to} 0.
\]
\end{proposition}

\begin{remark}[Inconsistency of naive bootstrap]\label{REM:bootstrap_inconsistent}
One may consider using $\GWp(\hat{\mu}^B_{n},\hat{\nu}_{n}^B)$ (rather than $\GWp(\hat{\rho}^B_{n,1},\hat{\rho}_{n,2}^B)$) to approximate the distribution of $\GWp(\hat{\mu}_n,\hat{\nu}_n)$, but this bootstrap is not consistent. Indeed, from the proof of \cref{prop: bootstrap case 3}, we may deduce that, if $\mu = \nu$, then $\sqrt{n}\GWp(\hat{\mu}^B_{n},\hat{\nu}_{n}^B)$ is expanded as
\[
\big\| \sqrt{n}(\hat{\mu}_n^B-\hat{\mu}_n)*\gamma_\sigma - \sqrt{n}(\hat{\nu}_n^B-\hat{\nu}_n)*\gamma_\sigma + \sqrt{n}(\hat{\mu}_n-\hat{\nu}_n)*\gamma_\sigma \big\|_{\dot{H}^{-1,p}(\mu*\gamma_\sigma)} + o_{\Prob}(1), 
\]
which converges in distribution to $\| G_\mu^1-G_\mu^2+G_\mu^3-G_\mu^4 \|_{\dot{H}^{-1,p}(\mu*\gamma_{\sigma})}$ unconditionally, where $G_\mu^1,\dots,G_\mu^4$ are independent copies of $G_\mu$. This implies that the conditional law of $\sqrt{n}\GWp(\hat{\mu}^B_{n},\hat{\nu}_{n}^B)$ does not converge weakly to the law of $\| G_{\mu}-G_\mu' \|_{\dot{H}^{-1,p}(\mu*\gamma_{\sigma})}$ in probability. 
\end{remark}

\begin{example}[Testing the equality of distributions]
Consider testing the equality of distributions, i.e., $H_0: \mu = \nu$ against $H_1: \mu \ne \nu$. We shall use $\sqrt{n}\GWp(\hat{\mu}_n,\hat{\nu}_n)$ as a test statistic and reject $H_0$ if $\sqrt{n}\GWp(\hat{\mu}_n,\hat{\nu}_n) > c$ for some critical value $c$. \cref{prop: bootstrap case 3} implies that, if we choose $c = \hat{c}_{1-\alpha}$ to be the conditional $(1-\alpha)$-quantile of $\sqrt{n}\GWp(\hat{\rho}^B_{n,1},\hat{\rho}_{n,2}^B)$ given the data, then the resulting test is asymptotically of level $\alpha$, 
\[
\lim_{n \to \infty}\Prob\left(\sqrt{n}\GWp(\hat{\mu}_n,\hat{\nu}_n) > \hat{c}_{1-\alpha}\right) = \alpha \quad \text{if} \ \mu=\nu.
\]
Here $\alpha \in (0,1)$ is the nominal level. To see that the test is consistent, note that if $\mu \ne \nu$, then $\GWp(\hat{\mu}_n,\hat{\nu}_n) \ge \GWp(\mu,\nu)-\GWp(\hat{\mu}_n,\mu)-\GWp(\hat{\nu}_n,\nu) \ge \GWp(\mu,\nu)/2$ with probability approaching one, %
while $\hat{c}_{1-\alpha} = O_{\Prob}(1)$ by \cref{prop: bootstrap case 3}. 

Testing the equality of distributions using Wasserstein distances was considered in \cite{ramdas2017}, but their theoretical analysis is focused on the $d=1$ case, partly because of the lack of null limit distribution results for empirical $\Wp$ in higher dimensions. We overcome this obstacle by using the smooth Wasserstein distance. 
\end{example}

\section{Minimum distance estimation with \(\GWp\)}
\label{sec: mde}

We consider the application of our limit distribution theory to MDE with $\GWp$. Given an independent sample $X_1, \dots, X_n$ from a distribution $\mu \in \cP$, MDE aims to learn a generative model from a parametric family $\{\nu_\theta\}_{\theta \in \Theta} \subset \cP$ that approximates $\mu$ under some statistical divergence. We use $\GWp$ as the proximity measure and the empirical distribution $\hat{\mu}_n$ as an estimate for $\mu$, which leads to the following MDE problem
\[
\inf_{\theta \in \Theta} \GWp(\hat{\mu}_n,\nu_\theta).
\]
MDE with classic $\mathsf{W}_1$ is called the Wasserstein GAN, which continues to underlie state-of-the-art methods in generative modeling \cite{arjovsky_wgan_2017,gulrajani2017improved}. MDE with $\GWp$ was previously examined for $p=1$ in \cite{goldfeld2020asymptotic} and for $p > 1$ in \cite{nietert21}. Specifically, \cite{nietert21} established measurability, consistency, and parametric convergence rates for MDE with $\GWp$ for $p > 1$, but did not derive limit distribution results. We will expand on this prior work by providing limit distributions for the $\GWp$ MDE problem. %

\medskip
Analogously to the conditions of Theorem 4 in \cite{goldfeld2020asymptotic}, we assume the following. %

\begin{assumption} Let $1<p<\infty$, and assume that the following conditions hold. 
(i) The distribution $\mu$ satisfies Condition \eqref{eq:moment-condition}.
(ii) The parameter space $\Theta\subset \R^{d_0}$ is compact with nonempty interior.
(iii) The map $\theta\mapsto \nu_{\theta}$ is continuous w.r.t. the weak topology.
(iv) There exists a unique $\theta\opt$ in the interior of $\Theta$ such that $\nu_{\theta\opt}=\mu$.
(v) There exists a neighborhood $N_0$ of $\theta\opt$ such that $\left(\nu_{\theta}-\nu_{\theta\opt}\right)*\gamma_{\sigma}\in \Hu$ for every $\theta\in N_0$. 
(vi) The map $N_0\ni \theta\mapsto \left(\nu_{\theta}-\nu_{\theta\opt}\right)*\gamma_{\sigma}\in \Hu$ is norm differentiable with non--singular derivative $\mathfrak{D}$ at $\theta\opt$. That is, there exists $\mathfrak{D}=\left(\mathfrak{D}_1,\dots,\mathfrak{D}_{d_0}\right)\in \big(\Hu\big)^{d_0}$, where $\mathfrak{D}_1,\dots,\mathfrak{D}_{d_0}$ are linearly independent elements of $\Hu$, such that
            \[
                \norm{ (\nu_{\theta}-\nu_{\theta\opt})*\gamma_{\sigma}-\inp{\theta-\theta\opt}{\mathfrak{D}}}_{\Hu}=o(\abs{\theta-\theta\opt}),%
            \]   
            as $\theta\to \theta \opt$ in $N_0$, where $\inp{t}{\mathfrak{D}}=\sum_{i=1}^{d_0}t_i\mathfrak{D}_i$ for $t=(t_1,\dots, t_{d_0})\in\R^{d_0}$.
    \label{as:MSE}
\end{assumption}

We derive limit distributions for the optimal value function and MDE solution, following the methodology of \cite{pollard_min_dist_1980,bernton2019,goldfeld2020asymptotic}.

\begin{theorem}[Limit distributions for MDE with $\GWp$]
\label{thm: MDE limit}
Suppose that \cref{as:MSE} holds. Let $\mathbb{G}_n^{(\sigma)}:=\sqrt n(\hat \mu_n-\mu)*\gamma_{\sigma}$ be the smooth empirical process, and $G_{\mu}$ its weak limit in $\dot{H}^{-1,p}(\mu*\gamma_{\sigma})$; cf. \cref{thm:limit-distribution}. Then, the following hold. 
\begin{enumerate}
    \item[(i)] We have $\inf_{\theta\in \Theta}\sqrt n{\sf W}_p^{(\sigma)}(\hat\mu_n,\nu_{\theta}) \stackrel{d}{\to} \inf_{t\in\R^{d_0}}\norm{G_{\mu}-\inp{t}{\mathfrak{D}}}_{\dot H^{-1,p}(\mu*\gamma_{\sigma})}$. 
    \item[(ii)] Let $(\hat \theta_n)_{n\in\N}$ be a sequence of measurable estimators satisfying 
\[
{\sf W}_{p}^{(\sigma)}(\hat \mu_n,\nu_{\hat \theta_n})\leq \inf_{\theta\in \Theta}{\sf W}_{p}^{(\sigma)}(\hat \mu_n,\nu_{\theta})+o_{\mathbb{P}}(n^{-1/2}).
\]
Then, provided that $\argmin_{t\in\R^{d_0}}\norm{G_{\mu}-\inp{t}{\mathfrak{D}}}_{\dot H^{-1,p}(\mu*\gamma_{\sigma})}$ is almost surely unique, we have $\sqrt n(\hat \theta_n-\theta\opt) \stackrel{d}{\to} \argmin_{t\in\R^{d_0}}\norm{G_{\mu}-\inp{t}{\mathfrak{D}}}_{\dot H^{-1,p}(\mu*\gamma_{\sigma})}$. 
\end{enumerate}
\end{theorem}

In general, it is nontrivial to verify that $\argmin_{t\in\R^{d_0}}\norm{G_{\mu}-\inp{t}{\mathfrak{D}}}_{\dot H^{-1,p}(\mu*\gamma_{\sigma})}$ is almost surely unique. However, for $p=2$, the Hilbertian structure of $\dot H^{-1,2}(\mu*\gamma_{\sigma})$ guarantees this uniqueness. %
Let $E$ be an isometric isomorphism between $\dot H^{-1,2}(\mu*\gamma_{\sigma})$ and a closed subspace of $L^2(\mu*\gamma_\sigma;\R^{d})$; cf. \cref{lem:vector-field}. Setting $\bunderline{G}_{\mu}:=E(G_{\mu})$ and $\bunderline{\mathfrak{D}}=\left(    \bunderline{\mathfrak{D}}_1,\dots,\bunderline{\mathfrak{D}}_{d_0}\right):=\big(E(\mathfrak{D}_1),\dots,E(\mathfrak{D}_{d_0})\big)$, we have
\[
\norm{G_{\mu}-\inp{t}{\mathfrak{D}}}_{\dot H^{-1,2}(\mu*\gamma_{\sigma})}=\norm{\bunderline G_{\mu}-\inp{t}{\bunderline{\mathfrak{D}}}}_{L^{2}(\mu*\gamma_{\sigma};\R^{d})}.
\]
The unique minimizer in $t$ of the above display is given by
\begin{equation}
    \hat{t}_\mu = \left [ \big (\inp{\bunderline{\mathfrak{D}}_{j}}{\bunderline{\mathfrak{D}}_{k}}_{L^2(\mu*\gamma_{\sigma};\R^{d})} \big )_{1 \le j,k \le d_0} \right ]^{-1} \big ( \inp{G_{\mu}}{\bunderline{\mathfrak{D}}_{j}}_{L^2(\mu*\gamma_{\sigma};\R^{d})} \big)_{j=1}^{d_0}.
    \label{eq:optMWSEp2}
\end{equation}
Since $\bunderline{G}_{\mu}$ is a centered Gaussian random variable in $L^2(\mu*\gamma_{\sigma};\R^d)$, $\hat t_{\mu}$ is a mean--zero Gaussian vector in $\R^{d_0}$.  

\begin{corollary}[Asymptotic normality for MDE solutions when $p=2$]
\label{cor:simpMSWE}
Consider the setting of \cref{thm: MDE limit} Part (ii) and let $p=2$. 
Then $\sqrt n(\hat \theta_n-\theta\opt) \stackrel{d}{\to} \hat t_{\mu}$, the mean--zero Gaussian vector in \eqref{eq:optMWSEp2}.  
\end{corollary}

Without assuming the uniqueness of $\argmin_{t\in\R^{d_0}}\norm{G_{\mu}-\inp{t}{\mathfrak{D}}}_{\dot H^{-1,p}(\mu*\gamma_{\sigma})}$, limit distributions for MDE solutions can be stated in terms of set-valued random variables.  Consider the set of approximate minimizers
\begin{equation}
    \hat \Theta_n:= \left\{\theta\in \Theta: {\sf W}_p^{(\sigma)}(\hat\mu_n,\nu_{\theta})\leq \inf_{\theta'\in \Theta} {\sf W}_p^{(\sigma)}(\hat\mu_n,\nu_{\theta'})+n^{-1/2}\lambda_n \right\},
    \label{eq:appmin}
\end{equation}
where $\lambda_n$ is any nonnegative sequence with $\lambda_n = o_{\Prob}(1)$.
We will show that $\hat \Theta_n\subset \theta\opt+n^{-1/2}K_n$ with inner probability approaching one for some sequence $K_n$ of random, convex, and compact sets; cf.  \cite[Section 2]{pollard_min_dist_1980}. 
To describe the sets $K_n$, for any $\beta \ge 0$ and $h \in \Hu$, define 
{\small
\[
    K(h,\beta):=\left\{ t\in\R^{d_0}:\norm{h-\inp{t}{\mathfrak{D}}}_{\dot H^{-1,p}(\mu*\gamma_{\sigma})}\leq \inf_{t'\in\R^{d_0}}\norm{h-\inp{t'}{\mathfrak{D}}}_{\dot H^{-1,p}(\mu*\gamma_{\sigma})}+\beta \right\}\in\mathfrak{K},
\]
}
\noindent where $\mathfrak{K}$ is the class of compact, convex, and nonempty subsets of $\R^{d_0}$ endowed~with~the Hausdorff topology. That is, the topology induced by the Hausdorff metric $d_H(K_1,K_2):=\inf\left\{\delta>0:K_2\subset K_1^{\delta},K_1\subset K_2^{\delta}  \right\}$, where $K^{\delta}:=\bigcup_{x\in K}\left\{ y\in\R^{d_0}:\norm{x-y}\leq\delta \right\}$. Lemma~7.1 in \cite{pollard_min_dist_1980} shows that the map $h\mapsto K(h,\beta)$ is measurable from $\Hu$ into $\mathfrak K$ for any $\beta\geq0$.

\begin{proposition}[Limit distribution for set of approximate minimizers]
\label{prop: MDE limit}
Under \cref{as:MSE},   there exists a sequence of nonnegative real numbers $\beta_n\dn 0$ such that (i) $\mathbb P_*\big(\hat \Theta_n\subset \theta\opt + n^{-1/2}K(\mathbb{G}_n^{(\sigma)},\beta_n)\big)\to 1$, where $\mathbb{P}_*$ denotes inner probability;  and (ii) $K\big(\mathbb{G}_n^{(\sigma)},\beta_n\big) \stackrel{d}{\to} K(G_{\mu},0)$ as $\mathfrak{K}$--valued random variables.    
\end{proposition}

The proof of this proposition is an adaptation of that of Theorem 7.2 in \cite{pollard_min_dist_1980}. For completeness, a self-contained argument is provided in Appendix \ref{sec: MDE additional proof}. 

\section{Remaining proofs}
\label{sec: proofs}

\subsection{Proofs for \cref{sec: weak convergence}}

We fix some notation. For a nonempty set $S$, let $\ell^\infty(S)$ denote the space of bounded real functions on $S$ endowed with the sup-norm $\| \cdot \|_{\infty,S} = \sup_{s \in S}| \cdot |$. The space $(\ell^\infty(S),\| \cdot \|_{\infty,S})$ is a Banach space. 

\subsubsection{Proof of \cref{thm:limit-distribution}}
We divide the proof into three steps. In Steps 1 and 2, we will establish weak convergence of  $\sqrt{n}(\hat{\mu}_n-\mu)*\gamma_{\sigma}$ in $\dot{H}^{-1,p}(\gamma_\sigma)$. Step 3 is devoted to weak convergence of  $\sqrt{n}(\hat{\mu}_n-\mu)*\gamma_{\sigma}$ in $\dot{H}^{-1,p}(\mu*\gamma_\sigma)$.

\smallskip

\underline{Step 1}.  Observe that
\begin{equation}
\big ((\hat{\mu}_n - \mu)*\gamma_{\sigma} \big) (f) = (\hat{\mu}_n - \mu)(f*\phi_{\sigma}). 
\label{eq: convolution identity}
\end{equation}
Consider the function classes
\[
\label{eq:sobolev-ball}
\cF =\big \{ f \in \dot{C}_{0}^{\infty} :  \| f \|_{\dot{H}^{1,q}(\gamma_{\sigma})} \le 1 \big\} \quad \text{and} \quad \cF*\phi_{\sigma} = \big \{ f*\phi_{\sigma} : f \in \cF \big \}. 
\]
The proof of Theorem 3 in \cite{nietert21} shows that the function class $\cF*\phi_{\sigma}$ is $\mu$-Donsker. For completeness, we provide an outline of the argument. Since for any constant $a \in \R$ and any function $f \in \cF$, $(\hat{\mu}_n - \mu)(f * \phi_{\sigma}) = (\hat{\mu}_n - \mu)\big((f-a)* \phi_{\sigma} \big)$, it suffices to show that $\cF_0*\phi_{\sigma}$ with $\cF_0:= \{ f \in \cF: \gamma_{\sigma}(f) = 0 \}$ is $\mu$-Donsker. To this end, we will apply Theorem 1 in \cite{vanderVaart1996} or its simple adaptation, Lemma 8 in \cite{nietert21}. 

Fix any $\eta \in (0,1)$.
We first observe that, for any $f \in \cF_0$ and any multi-index $k = (k_1,\dots,k_d) \in \N_{0}^{d}$, we have
\begin{equation}
\big|\partial^{k} (f*\phi_\sigma)(x)\big| \lesssim (\mathsf{C}_q(\gamma_{\sigma}) \vee \sigma^{-\bar{k}+1})
\exp \left ( \frac{(p-1)|x|^2}{2\sigma^2(1-\eta)} \right)
\label{eq: derivative}
\end{equation}
up to constants independent of $f, x$, and $\sigma$, where $\bar{k} = \sum_{j=1}^{d}k_j$. Here $\partial^k = \partial_1^{k_1} \cdots \partial_d^{k_d}$ is the differential operator  and $\mathsf{C}_q(\gamma_\sigma)$ is the $q$-Poincar\'{e} constant for the Gaussian measure $\gamma_\sigma$. To see this, observe that
\[
(f*\phi_\sigma)(x)=  \int_{\R^d} \frac{\phi_{\sigma}(x-y)}{\phi_{\sigma}(y)} f(y) \phi_{\sigma}(y) d y. 
\]
Applying H\"{o}lder's inequality and using the fact that $\| f \|_{L^q(\gamma_{\sigma})} \le \mathsf{C}_q(\gamma_{\sigma}) \| f \|_{\dot{H}^{1,q}(\gamma_\sigma)} \le \mathsf{C}_q(\gamma_{\sigma})$ (recall that $\gamma_\sigma (f) = 0$), we obtain 
\[
|(f*\phi_\sigma)(x)| \le \mathsf{C}_q(\gamma_{\sigma}) \left [\int_{\R^d} \frac{\phi_{\sigma}^{p}(x-y)}{\phi_{\sigma}^{p-1}(y)} d y \right ]^{1/p}.
\]
A direct calculation further shows that
\[
\int_{\R^d} \frac{\phi_{\sigma}^{p}(x-y)}{\phi^{p-1}_{\sigma}(y)} d y 
=  \exp \left ( \frac{p(p-1)|x|^2}{2\sigma^2} \right ),
\]
which implies
\[
|(f*\phi_\sigma)(x)|\le  \mathsf{C}_q(\gamma_{\sigma}) \exp \left ( \frac{(p-1)|x|^2}{2\sigma^2} \right ),
\]
establishing (\ref{eq: derivative}) when $\bar{k} = 0$. Derivative bounds follow similarly; see \cite{nietert21} for details.

Next, we construct a cover $\{ \cX_j \}_{j=1}^{\infty}$ of $\R^d$. Let $B_r = B(0,r)$. For $\delta > 0$ fixed and $r =2,3,\dots$, let $\{ x_{1}^{(r)},\dots,x_{N_{r}}^{(r)} \}$ be a minimal $\delta$-net of $B_{r\delta} \setminus B_{(r-1)\delta}$.
Set $x_{1}^{(1)} = 0$ with $N_1 = 1$. It is not difficult to see from a volumetric argument that $N_{r} = O(r^{d-1})$.
Set $\cX_j = B(x_j^{(r)},\delta)$ for $j=\sum_{k=1}^{r-1}N_{k}+1,\dots,\sum_{k=1}^{r}N_{k}$. 
By construction, $\{ \cX_j \}_{j=1}^{\infty}$ forms a cover of $\R^{d}$ with diameter $2\delta$.
Set $\alpha = \lfloor d/2 \rfloor +1$ and $M_j = \sup_{f \in \cF_0} \max_{\bar{k} \le \alpha} \sup_{x \in \mathrm{int}(\cX_j)} |\partial^{k}(f*\phi_{\sigma})(x)|$. By Theorem 1 in \cite{vanderVaart1996} combined with Theorem 2.7.1 in \cite{vandervaart1996book} (or their simple adaptation; cf. Lemma 8 in \cite{nietert21}), $\cF_0*\phi_{\sigma}$ is $\mu$-Donsker if $\sum_{j=1}^\infty M_j \mu(\mathcal{X}_j)^{1/2} < \infty$.
By inequality (\ref{eq: derivative}), 
\[
\max_{\sum_{k=1}^{r-1}N_{k}+1 \le j \le \sum_{k=1}^{r} N_{j}} M_j \lesssim \sigma^{-\lfloor d/2 \rfloor}  \exp \left ( \frac{(p-1)r^2 \delta^2}{2\sigma^2 (1-\eta)} \right )
\]
up to constants independent of $r$ and $\sigma$. 
Hence, $\sum_{j=1}^\infty M_j \mu(\mathcal{X}_j)^{1/2}$ is finite if
\[
\sum_{r=1}^{\infty} r^{d-1}\exp \left ( \frac{(p-1)r^2 \delta^2}{2\sigma^2 (1-\eta)} \right ) \sqrt{\Prob(|X| > (r-1)\delta)} <\infty.
\]
By Riemann approximation, the sum on the the left-hand side above can be bounded by 
\[
\delta^{-d-1} \int_{1}^{\infty}t^{d-1}\exp \left ( \frac{(p-1)t^2}{2\sigma^2 (1-\eta)} \right ) \sqrt{\Prob(|X| > t-2\delta)} d t,
\]
which is finite under our assumption by choosing $\eta$ and $\delta$ sufficiently small, and absorbing $t^{d-1}$ into the exponential function.

\smallskip

\underline{Step 2}. 
Let $\cU = \{ f \in \dot{H}^{1,q}(\gamma_{\sigma}) : \| f \|_{\dot{H}^{1,q}(\gamma_{\sigma})} \le 1 \}$.
Recall from \cref{rem:explicit-construction} that $\dot{H}^{1,q}(\gamma_\sigma) \subset L^q(\gamma_\sigma)$. From Step 1, we know that $\cF*\phi_{\sigma}$ is $\mu$-Donsker.
The same conclusion holds with $\cF$ replaced by $\cU$. This can be verified as follows. From the proof of (\ref{eq: derivative}) when $\bar{k}=0$, we see that for $f_1,f_2 \in \dot{H}^{1,q}(\gamma_{\sigma})$ with $\gamma_{\sigma}$-mean zero, 
\[
\big|(f_1*\phi_{\sigma})(x) - (f_2*\phi_{\sigma})(x)\big| \le  \mathsf{C}_q(\gamma_{\sigma})\| f_1-f_2 \|_{\dot{H}^{1,q}(\gamma_{\sigma})} \exp \left ( \frac{(p-1)|x|^2}{2\sigma^2} \right), \quad \forall\, x \in \R^d. 
\]
Since the exponential function on the right-hand side is square-integrable w.r.t. $\mu$ under Condition (\ref{eq:moment-condition}) and $\cF_0$ is dense in $\cU_0 := \{ f \in \cU : \gamma_{\sigma}(f) = 0 \}$ for $\| \cdot \|_{\dot{H}^{1,q}(\gamma_\sigma)}$ by construction (cf. Remark \ref{rem:explicit-construction}), we see that 
\[
\cU_0*\phi_\sigma \subset \big \{ g : \exists\, g_m \in \cF_0*\phi_\sigma \ \text{such that $g_m \to g$ poinwise and in $L^2(\mu)$}\big \}.
\]
Thus, by Theorem 2.10.2 in \cite{vanderVaart1996}, $\cU_0*\phi_\sigma$ (or equivalently, $\cU*\phi_\sigma$) is $\mu$-Donsker. 
Since the map $ \ell^{\infty}(\cU*\phi_{\sigma}) \ni L \mapsto (L(f*\phi_{\sigma}))_{f \in \cU} \in \ell^{\infty}(\cU)$ is isometrically isomorphic, in view of (\ref{eq: convolution identity}), we have $\sqrt{n}(\hat{\mu}_n-\mu)*\gamma_{\sigma} \stackrel{d}{\to} G_{\mu}^{\circ}$ in $\ell^{\infty}(\cU)$ for some tight Gaussian process~$G_{\mu}^{\circ}$. %

Let $\lin^{\infty}(\cU)$ denote all bounded real functionals $L$ on $\cU$, such that $L(0)=0$ and %
\[
L(\alpha f + (1-\alpha) g) = \alpha L(f)+(1-\alpha)L(g), \quad 0 \le \alpha \le 1, \ f,g \in \cU. 
\]
Equip $\lin^{\infty}(\cU)$ with the norm $\| \cdot \|_{\infty,\cU} = \sup_{f \in \cU}| \cdot |$.  
Each element in $\lin^{\infty}(\cU)$ extends uniquely to  the corresponding element in $\dot{H}^{-1,p}(\gamma_{\sigma})$, and the extension, denoted by  $\iota: \lin^{\infty}(\cU) \to \dot{H}^{-1,p}(\gamma_{\sigma})$, is isometrically isomorphic. 
This follows from the same argument as the proof of Lemma 1 in \cite{nickl2009}. We omit the details. 

Since $\lin^{\infty}(\cU)$ is a closed subspace of $\ell^{\infty}(\cU)$ and $\sqrt{n}(\hat{\mu}_n-\mu)*\gamma_{\sigma}$ has paths in $\lin^{\infty}(\cU)$, we see that $G_{\mu}^{\circ} \in \lin^{\infty}(\cU)$ with probability one by the portmanteau theorem and $\sqrt{n}(\hat{\mu}_n-\mu)*\gamma_{\sigma} \stackrel{d}{\to} G_{\mu}^{\circ}$ in $\lin^{\infty}(\cU)$.  Now, since $\sqrt{n} (\hat{\mu}_n-\mu)*\gamma_\sigma$ is a (random) signed measure that is bounded on $\cU$ with probability one, we can regard  $\sqrt{n}(\hat{\mu}_n-\mu)*\gamma_\sigma$ as a random variable with values in $\dot{H}^{-1,p}(\gamma_\sigma)$. Conclude that $\sqrt{n}(\hat{\mu}_n-\mu)*\gamma_{\sigma} \stackrel{d}{\to} \iota\circ G_{\mu}^{\circ}$ in $\dot{H}^{-1,p}(\gamma_{\sigma})$ by the continuous mapping theorem. For notational convenience, redefine $G_{\mu}^{\circ}$ by $\iota \circ G_{\mu}^{\circ}$. The limit variable $G_\mu^{\circ} = (G_\mu^{\circ}(f))_{f \in \dot{H}^{1,q}(\gamma_\sigma)}$ is a centered Gaussian process with covariance function $\Cov(G^{\circ}_{\mu}(f),G^{\circ}_{\mu}(g) \big) = \Cov_{\mu}(f*\phi_{\sigma},g*\phi_{\sigma})$.

\smallskip

\underline{Step 3}. 
We will show that $\sqrt{n}(\hat{\mu}_n-\mu)*\gamma_{\sigma}$ converges in distribution to a centered Gaussian process in $\dot{H}^{-1,p}(\mu*\gamma_{\sigma})$. 
For $X \sim \mu$ with $a = \E[X]$, let $\mu^{-a}$ denote the distribution of $X-a$, and let $\hat{\mu}_{n}^{-a} = n^{-1}\sum_{i=1}^n \delta_{(X_i-a)}$. It is not difficult to see that $\mu^{-a}$ satisfies Condition (\ref{eq:moment-condition}).
Applying the result of Step 2 with $\mu$ replaced by $\mu^{-a}$, we have $\sqrt{n}(\hat{\mu}_n^{-a}-\mu^{-a})*\gamma_{\sigma} \stackrel{d}{\to} G_{\mu^{-a}}^{\circ}$ in $\dot{H}^{-1,p}(\gamma_{\sigma})$. 
Since $\| \cdot \|_{\dot{H}^{1,q}(\gamma_{\sigma})} \lesssim \| \cdot \|_{\dot{H}^{1,q}(\mu^{-a}*\gamma_{\sigma})}$ (as $d(\mu^{-a}*\gamma_{\sigma})/d\gamma_{\sigma} \ge e^{-\E_{\mu}[|X-a|^2]/(2\sigma^2)}$ by Jensen's inequality), we have $\| \cdot \|_{\dot{H}^{-1,p}(\mu^{-a}*\gamma_{\sigma})} \lesssim \| \cdot \|_{\dot{H}^{-1,p}(\gamma_{\sigma})}$, i.e., the continuous embedding $\dot{H}^{-1,p}(\gamma_{\sigma}) \hookrightarrow \dot{H}^{-1,p}(\mu^{-a}*\gamma_{\sigma})$ holds. Thus $\sqrt{n}(\hat{\mu}_n^{-a}-\mu^{-a})*\gamma_{\sigma} \stackrel{d}{\to} (G_{\mu^{-a}}^{\circ}(f))_{f \in \dot{H}^{1,q}(\mu^{-a}*\gamma_{\sigma})}$ in $\dot{H}^{-1,p}(\mu^{-a}*\gamma_{\sigma})$. 

Observe that for $\varphi \in C_0^\infty$, 
\[
\begin{split}
 \| \varphi (\cdot+a) \|_{\dot{H}^{1,q}(\mu^{-a}*\gamma_{\sigma})}^q &= \int_{\R^d} |\nabla \varphi(\cdot+a)|^q d(\mu^{-a}*\gamma_{\sigma}) \\
 &= \int_{\R^d} |\nabla \varphi|^q d(\mu*\gamma_{\sigma}) = \| \varphi \|_{\dot{H}^{1,q}(\mu*\gamma_{\sigma})}^q.
 \end{split}
\]
Thus, the map $\tau_a: \dot{H}^{-1,p}(\mu^{-a}*\gamma_{\sigma}) \to \dot{H}^{-1,p}(\mu*\gamma)$, defined by $\tau_a(h)(f) = h(f(\cdot+a))$, is continuous (indeed, isometrically isomorphic). 
Conclude that 
\[
\sqrt{n}(\hat{\mu}_n-\mu)*\gamma_{\sigma} = \tau_{a}\big (\sqrt{n}(\hat{\mu}_n^{-a}-\mu^{-a})*\gamma_{\sigma} \big) \stackrel{d}{\to} \tau_aG^{\circ}_{\mu^{-a}}=:G_{\mu} \quad \text{in} \ \dot{H}^{-1,p}(\mu*\gamma_{\sigma}).
\] 
The limit variable $G_\mu = (G_\mu(f))_{f \in \dot{H}^{1,q}(\mu*\gamma_\sigma)} = (G_{\mu^{-a}}^{\circ}(f(\cdot+a)))_{f \in \dot{H}^{1,q}(\mu*\gamma_\sigma)}$ is a centered Gaussian process with covariance function
\[
\begin{split}
\Cov(G_\mu(f),G_\mu(g)) &= \Cov\big (G^{\circ}_{\mu^{-a}}(f(\cdot+a),G^{\circ}_{\mu^{-a}}(g(\cdot+a)) \big) \\
&= \Cov_{\mu^{-a}}\big (f(\cdot+a)*\phi_{\sigma},g(\cdot+a)*\phi_{\sigma}\big) \\
&=\Cov_{\mu^{-a}}\big (f*\phi_{\sigma}(\cdot+a),g*\phi_{\sigma}(\cdot+a) \big) \\
&= \Cov_{\mu}(f*\phi_{\sigma},g*\phi_{\sigma}). 
\end{split}
\]
This completes the proof. \qed

\begin{remark}[Alternative proof for $p=2$]
\label{rem: alternative}
Observe that $(\hat{\mu}_n-\mu)*\gamma_\sigma = n^{-1} \sum_{i=1}^n (\delta_{X_i}-\mu)*\gamma_\sigma = n^{-1}\sum_{i=1}^n Z_i$ with $Z_i = (\delta_{X_i}-\mu)*\gamma_\sigma$, and that $Z_1,Z_2,\dots$ are i.i.d. random variables with values in $\dot{H}^{-1,p}(\gamma_\sigma)$ (cf. (\ref{eq: derivative})). Since $\dot{H}^{-1,2}(\gamma_\sigma)$ is isometrically isomorphic to a closed subspace of $L^2(\gamma_\sigma;\R^d)$ (see \cref{lem:vector-field} ahead), we may apply the CLT in the Hilbert space to derive a limit distribution for $\sqrt{n}(\hat{\mu}_n-\mu)*\gamma_\sigma = n^{-1/2}\sum_{i=1}^n Z_i$ in $\dot{H}^{-1,2}(\gamma_\sigma)$. Let $E: \dot{H}^{-1,2}(\gamma_\sigma) \to L^p(\gamma_\sigma;\R^d)$ be the linear isometry given in \cref{lem:vector-field} ahead and $\underline{Z}_i = E(Z_i)$ be the corresponding  $L^2(\gamma_\sigma;\R^d)$-valued random variables. Since $L^2(\gamma_\sigma;\R^d)$ is a Hilbert space, $n^{-1/2}\sum_{i=1}^nZ_i$ obeys the CLT if $\E\big[\| \underline{Z}_1 \|^2_{L^2(\gamma_\sigma;\R^d)}\big] = \E\big[\|Z_1\|_{\dot{H}^{-1,2}(\gamma_\sigma)}^2\big] < \infty$, which is satisfied under Condition (\ref{eq:moment-condition}). Indeed, for $p=2$, it is not difficult to see that the CLT in $\dot{H}^{-1,2}(\gamma_\sigma)$ holds for $n^{-1/2}\sum_{i=1}^nZ_i$ under a slightly weaker moment condition, namely, $\int_{\R^d} e^{|x|^2/\sigma^2} d\mu(x) < \infty$. 
\end{remark}

\subsubsection{Proof of Proposition \ref{prop:glivenko-cantelli}}

\underline{Part (i)}. 
Let
\[
\mathfrak{F} = \{ f*\phi_\sigma : f \in \dot{H}^{1,q}(\gamma_{\sigma}), \| f \|_{L^{q}(\gamma_{\sigma})} \le 1,  \|f\|_{\dot{H}^{1,q}(\gamma_{\sigma})}  \le C \}
\]
for some sufficiently large but fixed constant $C$. It is not difficult to see that $\lim_{n \to \infty} \| (\hat{\mu}_n-\mu)*\gamma_{\sigma}\|_{\dot{H}^{-1,p}(\gamma_{\sigma})} = 0$ a.s. if and only if $\mathfrak{F}$ is $\mu$-Glivenko-Cantelli. 

Suppose first that $\mathfrak{F}$ is $\mu$-Glivenko-Cantelli. Let $F_{\mathfrak{F}}$ denote the minimal envelope for $\mathfrak{F}$, i.e., $F_{\mathfrak{F}}(x) = \sup_{f \in \mathfrak{F}}|f(x)|$. 
By Theorem 3.7.14 in \cite{gine2016}, $F_{\mathfrak{F}}$ must be $\mu$-integrable. We shall bound $F_{\mathfrak{F}}$ from below. Fix any $x \in \R^d$. Consider 
\[
\varphi_{x}(y)=  \frac{g_{x}^{p-1}(y)}{\|g_{x}\|_{L^p(\gamma_{\sigma})}^{p-1}} \quad \text{with} \quad 
g_{x}(y) = \frac{\phi_{\sigma}(x-y)}{\phi_{\sigma}(y)} = e^{-|x|^2/(2\sigma^2)+\langle x,y \rangle/\sigma^2}, \quad y \in \R^d. 
\]
Observe that $\nabla_y \varphi_x (y) = \big((p-1)x/\sigma^2\big) \varphi_{x}(y)$ and thus $\| \varphi_x \|_{\dot{H}^{1,q}(\gamma_{\sigma})} = (p-1)|x|/\sigma^2$.
Thus, for $\tilde{\varphi}_x = \varphi_x/(1+|x|)$, we have $\| \tilde{\varphi}_x \|_{L^q(\gamma_{\sigma})} \le 1, \| \tilde{\varphi}_x \|_{\dot{H}^{1,q}(\gamma_{\sigma})} \le (p-1)/\sigma^2$, and 
\[
(\tilde{\varphi}_x*\phi_\sigma) (x) = \frac{1}{1+|x|} \| g_x \|_{L^p(\gamma_{\sigma})} = \frac{1}{1+|x|} e^{(p-1)|x|^2/(2\sigma^2)}.
\]
Also, from Proposition 1.5.2 in \cite{bogachev1998}, we see that $\tilde{\varphi}_x \in \dot{H}^{1,q}(\gamma_\sigma)$.  Conclude that, as long as $C \ge (p-1)/\sigma^2$, 
\[
F_{\mathfrak{F}}(x) \ge \frac{1}{1+|x|} e^{(p-1)|x|^2/(2\sigma^2)}. 
\]
Now, the left-hand side is $\mu$-integrable, so that $\int_{\R^d} e^{\theta |x|^2/(2\sigma^2)} d \mu(x) < \infty$ for any $\theta < p-1$. 

\underline{Part (ii)}. 
Conversely, suppose that $\int_{\R^d} e^{(p-1) |x|^2/(2\sigma^2)} d \mu(x) < \infty$, which ensures that $F_{\mathfrak{F}}$ is $\mu$-integrable from (\ref{eq: derivative}). From the proof of \cref{thm:limit-distribution}, for any $M > 0$, we see that the restricted function class $\{ f\mathbbm{1}_{F_{\mathfrak{F}} \le M} : f \in \mathfrak{F} \}$ is $\mu$-Donsker and thus $\mu$-Glivenko-Cantelli (cf. Theorem 3.7.14 in \cite{gine2016}). Since the envelope function $F_{\mathfrak{F}}$ is $\mu$-integrable, we conclude that $\mathfrak{F}$ is $\mu$-Glivenko-Cantelli; cf. the proof of Theorem 3.7.14 in \cite{gine2016}.
\qed

\subsection{Proofs for Section \ref{sec: null case}}

Recall that $1 < p < \infty$ and $q$ is its conjugate index, i.e., $1/p+1/q=1$. 

\subsubsection{Proof of Lemma \ref{lem: Haradard derivative null}}
One of the main ingredients of the proof of Lemma \ref{lem: Haradard derivative null} is Theorem 8.3.1 in \cite{ambrosio2005}, which is stated next (see also the Benamou-Brenier formula \cite{benamou2000}). 

\begin{theorem}[Theorem 8.3.1 in \cite{ambrosio2005}]
\label{thm:continuity}
Let $I$ be an open interval, and let $I \ni t \mapsto \mu_{t}$ be a continuous curve in $\cP_{p}(\R^{d})$ (equipped with $\Wp$) such that for some Borel vector field $\R^{d} \times I \ni (x,t) \mapsto v_{t}(x) \in \R^{d}$, the continuity equation
\begin{equation}
\label{eq:continuity}
\partial_{t} \mu_{t} + \nabla \cdot (v_{t} \mu_{t}) = 0 \quad \text{in} \ \  \R^{d} \times I
\end{equation}
holds in the distributional sense, i.e.,
\[
\int_{I} \int_{\R^{d}} (\partial_{t} \varphi (x,t) + \langle v_{t}(x),\nabla_{x} \varphi(x,t) \rangle ) d\mu_{t}(x) d t = 0, \quad \forall \varphi \in C_{0}^{\infty}(\R^{d} \times I). 
\]
If $\| v_{t} \|_{L^{p}(\mu_{t};\R^d)} \in L^1 (I)$, then $\Wp(\mu_{a},\mu_{b}) \le \int_{a}^{b} \| v_{t} \|_{L^{p}(\mu_{t};\R^d)} d t$ for all $a <b$ with $a,b \in I$.
\end{theorem}

For a vector field $v: \R^d \to \R^d$, define
\[
j_{p}(v) := 
\begin{cases}
|v|^{p-2} v & \text{if} \ v \neq 0 \\
0 & \text{otherwise}
\end{cases}.
\]
Observe that $w = j_p(v)$ if and only if $v = j_{q}(w)$, and for any $\rho \in \cP$, 
\[
\| j_{p}(v) \|_{L^{q}(\rho;\R^d)}^{q} =  \| v \|_{L^p(\rho;\R^d)}^p = \int_{\R^d} \langle j_p(v), v \rangle d\rho.
\]
We will also use the following lemma.

\begin{lemma}
\label{lem:vector-field}
Let $\rho \in \cP$ be a reference measure.  For any $h \in \dot{H}^{-1,p}(\rho)$, there exists a unique  vector field $E = E(h)\in L^{p}(\rho;\R^{d})$ such that 
\begin{equation}
\begin{cases}
&\int_{\R^{d}} \langle \nabla \varphi,E \rangle  d \rho =h(\varphi), \quad \forall \varphi \in C_{0}^{\infty}, \\
&j_p(E) \in \overline{\{ \nabla \varphi : \varphi \in C_{0}^{\infty} \}}^{L^q(\rho;\R^d)}.
\end{cases}
\label{eq: vector field}
\end{equation}
The map $h \mapsto E(h)$ is homogeneous (i.e., $E(ah) = a E(h)$ for all $a \in \R$ and $h \in \dot{H}^{-1,p}(\rho)$) and such that $\| E(h) \|_{L^p(\rho;\R^d)} = \| h \|_{\dot{H}^{-1,p}(\rho)}$ for all $h \in \dot{H}^{-1,p}(\rho)$. If $p=2$, then the map $h \mapsto E(h)$ is a linear isometry from $\dot{H}^{-1,2}(\rho)$ into $L^2(\rho;\R^d)$. 
\end{lemma}

The proof of Lemma \ref{lem:vector-field} in turn relies on the following existence result of optimal solutions in Banach spaces. We provide its proof for the sake of completeness.

\begin{lemma}
\label{lem:optimization}
Let $(V,\| \cdot \|)$ be a reflexive real Banach space, and let $J: V \to \R \cup \{ +\infty \}$ ($J \not  \equiv +\infty$) be weakly lower semicontinuous (i.e., $J(v) \le \liminf_{n} J(v_n) $ for any $v_n \to v$ weakly) and coercive (i.e., $J(v) \to \infty$ as $\| v \| \to \infty$).
Then there exists $v_0 \in V$ such that $J(v_0) = \inf_{v \in V} J(v)$. 
\end{lemma}

\begin{proof}[Proof of \cref{lem:optimization}]
Let $v_n \in V$ be such that $J(v_n) \to \inf_{v \in V}J(v) =: \underline{J}$.
By coercivity, $v_n$ is bounded, so by reflexivity and the Banach-Alaoglu theorem, there exists a weakly convergent subsequence $v_{n_k}$ such that $v_{n_k} \to v_0$ weakly.
Since $J$ is weakly lower semicontinuous, we conclude $J(v_0) \le \liminf_{k} J(v_{n_k}) = \underline{J}$. 
\end{proof}

We turn to the proof of  \cref{lem:vector-field}, which is inspired by the first part of the proof of Theorem 8.3.1 in \cite{ambrosio2005}.

\begin{proof}[Proof of Lemma \ref{lem:vector-field}]

 Let $V$ denote the closure in $L^{q}(\rho;\R^{d})$ of the subspace $V_0 = \{ \nabla \varphi : \varphi \in C_{0}^{\infty} \}$. Endowing $V$ with $\| \cdot \|_{L^q(\rho;\R^d)}$ gives a reflexive Banach space because any closed subspace of a reflexive Banach space is reflexive. Define the linear functional $L: V_0 \to \R$ by 
$L(\nabla \varphi):=h(\varphi)$. 
To see that $L$ is well-defined, observe that 
\[
\begin{split}
|h(\varphi)|&\le \| \varphi \|_{\dot{H}^{1,q}(\rho)} \| h \|_{\dot{H}^{-1,p}(\rho)} \\
&=\| \nabla \varphi \|_{L^q(\rho;\R^d)} \| h \|_{\dot{H}^{-1,p}(\rho)}.
\end{split}
\]
This also shows that $L$ can be extended to a bounded linear functional on $V$.

Consider the optimization problem
\begin{equation}
\label{eq:optimization}
\min_{v \in V} J(v) \quad \text{with} \ \ J(v) := \frac{1}{q} \int_{\R^{d}} | v |^{q}d\rho - L(v).
\end{equation}
The functional $J$ is finite, weakly lower semicontinuous, and coercive. By \cref{lem:optimization} there exists a solution $v_0$ to the optimization problem \eqref{eq:optimization}.
Further, the functional $J$ is G\^{a}teaux differentiable with derivative
\[
J'(v;w) := \lim_{t \to 0} \frac{J(v+tw) - J(v)}{t} = \int_{\R^{d}} \langle w,j_{q}(v)  \rangle d\rho - L(w). 
\]
Thus, for $E=j_{q}(v_0)$, we have $\int_{\R^{d}} \langle \nabla \varphi,E  \rangle d\rho = L(\nabla \varphi)$ for all $\varphi \in C_{0}^{\infty}$ and $j_p(E) = v_0 \in V$.

To show uniqueness of $E$, pick another vector field $E' \in L^p(\rho;\R^d)$ satisfying (\ref{eq: vector field}). Then, $j_p(E') \in V$ satisfies $J'(j_p(E');w) = 0$ for all $w \in V$, so by convexity of $J$, $j_p(E')$ is another optimal solution to (\ref{eq:optimization}). However, since $J$ is strictly convex, the optimal solution to (\ref{eq:optimization}) is unique, so that $j_p(E') = j_p(E)$, i.e., $E'=E$.

Now, the map $h \mapsto E(h)$ is homogeneous, as $aE(h)$ clearly satisfies the first equation in (\ref{eq: vector field}) for $h$ replaced with $ah$ and $j_p(aE(h)) = |a|^{p-2}aj_p(E(h)) \in V$. Further, as $j_p(E(h)) \in \overline{\{ \nabla \varphi : \varphi \in C_{0}^{\infty} \}}^{L^q(\rho;\R^d)}$ by construction, it also satisfies
\[
\| E(h) \|_{L^p(\rho;\R^d)} = \sup \left \{ \int_{\R^d} \langle \nabla \varphi, E(h) \rangle d\rho : \varphi \in C_0^\infty, \| \nabla \varphi \|_{L^q(\rho;\R^d)} \le 1 \right \} = \| h \|_{\dot{H}^{-1,p}(\rho)}.
\]
Finally, if $p=2$, then $j_2(v) = v$, so it is clear that the map $h\mapsto E(h)$ is linear. 
\end{proof}

We are now ready to prove Lemma \ref{lem: Haradard derivative null}.

\begin{proof}[Proof of Lemma \ref{lem: Haradard derivative null}]
Let $\mu_t = \mu + th_1$ and $\nu_t = \mu + th_2$ for $t \in [0,1]$. For notational convenience, let $h=h_1-h_2 \in D_\mu \cap \{ \text{finite signed Borel measures} \}$. We will first show that 
\[
\liminf_{t \downarrow 0} \frac{\Wp(\mu_{t},\nu_t)}{t} \ge \| h_1-h_2 \|_{\dot{H}^{-1,p}(\mu_0)}. 
\]
The proof is inspired by Theorem 7.26 in \cite{villani2003}.  Observe that for any $\varphi \in C_0^{\infty}$ and $t>0$,
\[
h(\varphi) = \int_{\R^d} \varphi dh = \int_{\R^d} \varphi d \left ( \frac{\mu_t - \nu_t}{t} \right) = \frac{1}{t} \int_{\R^d} \varphi d(\mu_t-\nu_t). 
\]
Let $\pi_t \in \Pi (\mu_t,\nu_t)$ be an optimal coupling for $\Wp^p (\mu_t,\nu_t)$, i.e., $\Wp^p(\mu_t,\nu_t) = \iint |x-y|^p d\pi_t(x,y)$. Then
\[
\frac{1}{t} \int_{\R^d} \varphi d(\mu_t-\nu_t) = \frac{1}{t} \iint_{\R^d \times \R^d} \{ \varphi (x)-\varphi(y) \} d\pi_t(x,y). 
\]
Since $\varphi$ is smooth and compactly supported, there exists a constant $C = C_{\varphi,p} < \infty$ such that 
\[
\varphi (x) - \varphi (y) \le \langle \nabla \varphi (y), x-y \rangle + C |x-y|^{2 \wedge p}, \quad \forall x,y \in \R^d.
\]
Indeed, for $p \ge 2$,  we can take $C = C_1 := \sup_{x \in \R^d} \| \nabla^2 \varphi(x) \|_{\mathrm{op}}/2$ (here $\| \cdot \|_{\mathrm{op}}$ denotes the operator norm for matrices). For  $1 < p < 2$, we have
\[
\varphi (x) - \varphi (y) \le \langle \nabla \varphi (y), x-y \rangle + C_1 C_2^{2-p} |x-y|^{p}, \ \forall x,y \in S := \supp(\varphi)
\]
with $C_2 := \sup\{ |x-y| : x,y \in S \}$. Here $\supp (\varphi)$ denotes the support  of $\varphi$, $\supp (\varphi) := \overline{\{ \varphi \ne 0 \}}$. 
If $x \in S$ and $d(y,S) := \inf \{ |y-z| : z \in S \} > 1$, then $\varphi(x)/|x-y|^{p}  \le \| \varphi \|_{\infty}$, so that we have 
\[
\varphi(x) - \varphi(y) = \varphi(x) \le (\| \varphi \|_{\infty} \vee C_1 C_3^{2-p})   |x-y|^p, \ \forall x \in S, y \in S^c
\]
with $C_3 := \sup \{ |x-y| : x\in S, d(y,S) \le 1 \} < \infty$. 
Finally, if $d(x,S) > 1$ and $y \in S$, then 
\[
\frac{-\varphi(y) - \langle \nabla \varphi(y),x-y \rangle}{|x-y|^p} \le \| \varphi \|_{\infty} |x-y|^{-p}+ \| \nabla \varphi \|_{\infty} |x-y|^{1-p} \le \| \varphi \|_{\infty} + \| \nabla \varphi \|_{\infty}, 
\]
so that  we have 
\[
\begin{split}
&\varphi(x) - \varphi(y) = -\varphi(y) \\
&\qquad \le \langle \nabla \varphi (y), x-y \rangle + \left((\| \varphi \|_{\infty} +  \| \nabla \varphi \|_{\infty}) \bigvee C_1 C_4^{2-p} \right ) |x-y|^p, \ \forall x \in S^c,y\in S
\end{split}
\]
with  $C_4 := \sup\{|x-y| : d(x,S) \le 1, y \in S \} < \infty$.

Now, we have
\[
\begin{split}
&\frac{1}{t} \iint_{\R^d \times \R^d} \{ \varphi (x)-\varphi(y) \} d\pi_t(x,y) \\
&\le \frac{1}{t} \left \{ \iint_{\R^d \times \R^d} \langle \nabla \varphi (y), x-y \rangle d\pi_{t}(x,y) + C \iint_{\R^d \times \R^d} |x-y|^{2 \wedge p} d\pi_t(x,y) \right \} \\
&\le \frac{1}{t} \left [ \iint_{\R^d \times \R^d} \langle \nabla \varphi (y), x-y \rangle d\pi_{t}(x,y) + C \left \{ \iint_{\R^d \times \R^d} |x-y|^p d\pi_t(x,y) \right \}^{2/(2 \vee p)}\right ] \\
&= \frac{1}{t} \left \{\iint_{\R^d \times \R^d} \langle \nabla \varphi (y), x-y \rangle d\pi_{t}(x,y) + C \Wp^{2 \wedge p}(\mu_t,\nu_t) \right \}. 
\end{split}
\]
Applying  Proposition \ref{prop:duality-in-Wp} with $\rho = \mu$, we know that $\Wp (\mu_t,\nu_t) \le \Wp(\mu_t,\mu) + \Wp(\mu,\nu_t) \le p t (\| h_1 \|_{\dot{H}^{-1,p}(\mu)} + \| h_2 \|_{\dot{H}^{-1,p}(\mu)})=  O(t)$ as $t \downarrow 0$, so that $\Wp^{2 \wedge p}(\mu_t,\nu_t) = O(t^{2 \wedge p}) = o(t)$ as $t \downarrow 0$. Further, by H\"{o}lder's inequality, with $q$ being the conjugate index of $p$, we have
\[
\iint_{\R^d \times \R^d} \langle \nabla \varphi (y), x-y \rangle d\pi_{t}(x,y) \le \| \nabla \varphi  \|_{L^q(\nu_t;\R^d)}\underbrace{\left \{ \iint_{\R^d \times \R^d} |x-y|^p d\pi_t(x,y) \right \}^{1/p}}_{=\Wp(\mu_t,\nu_t)}. 
\]
Here 
\[
\| \nabla \varphi  \|_{L^q(\nu_t;\R^d)}^q = \int_{\R^d} |\nabla \varphi |^q d\mu + t\int_{\R^d} | \nabla \varphi |^q dh_2 = \| \nabla \varphi  \|_{L^q(\mu;\R^d)}^q + O(t), \ t \downarrow 0. 
\]
Conclude that 
\[
h(\varphi) \le \| \nabla \varphi \|_{L^q(\mu;\R^d)} \liminf_{t \downarrow 0} \frac{\Wp(\mu_{t},\nu_t)}{t} ,
\]
that is,
\[
\liminf_{t \downarrow 0} \frac{\Wp(\mu_{t},\nu_t)}{t}  \ge \sup \left \{ h(\varphi) : \varphi \in C_0^{\infty}, \| \nabla \varphi \|_{L^q(\mu;\R^d)} \le 1 \right \} = \| h \|_{\dot{H}^{-1,p}(\mu)}. 
\]

To prove the reverse inequality,  let $\mathfrak{h} \mapsto E(\mathfrak{h})$ be the map from $\dot{H}^{-1,p}(\mu)$ into $L^p(\mu;\R^d)$ given in Lemma \ref{lem:vector-field}. 
Let $f_t^1 = d\mu_t/d\mu = 1+t dh_1/d\mu$.
Since $\mu_1 = \mu + h_1$ is a probability measure, we have $1+dh_1/d\mu \ge 0$, i.e., $dh_1/d\mu \ge -1$, so that $f_t^1 \ge 1/2$ for $t \in [0,1/2]$. Likewise, $f_t^2 := d\nu_t/d\mu \ge 1/2$ for $t \in [0,1/2]$. 

Fix $t \in [0,1/2]$ and consider the curve $\rho_s =(1-s) \mu_t + s\nu_t = \mu_t - sth$ for $s \in [0,1]$. Then $\rho_s$ satisfies the continuity equation (\ref{eq:continuity}) with $v_s = E(-th)/((1-s)f_t^1+sf_t^2))$. By Theorem \ref{thm:continuity} (Theorem 8.3.1 in \cite{ambrosio2005}), we have
\[
\Wp(\mu_t,\nu_t) \le \int_{0}^1 \| v_s \|_{L^p(\rho_s;\R^d)} ds = \int_0^1 \left ( \int_{\R^d} \frac{|E(-th)|^p}{\big [(1-s)f_t^1+sf_t^2\big]^{p-1}} d\mu \right)^{1/p} ds.
\]
Since $E(-th)=-tE(h)$ by homogeneity and $1/2 \le f_t^i \to 1$ as $t \downarrow 0$, the dominated convergence theorem yields that, as $t \downarrow 0$, 
\[
\begin{split}
\frac{\Wp(\mu_t,\nu_t)}{t} &\le \int_0^1 \left ( \int_{\R^d} \frac{|E(h)|^p}{\big [(1-s)f_t^1+sf_t^2\big]^{p-1}} d\mu \right)^{1/p} ds \\
&=\| E(h) \|_{L^p(\mu;\R^d)} + o(1) \\
&=\| h \|_{\dot{H}^{-1,p}(\mu)} + o(1).
\end{split}
\]
This completes the proof. 
\end{proof}

\subsubsection{Proof of Proposition \ref{prop: Haradard derivative null}}

Pick arbitrary %
$(h_1,h_2) \in T_{\Xi_{\mu} \times \Xi_\mu}(0,0)$, $t_n \downarrow 0$ and $(h_{n,1},h_{n,2})\to (h_1,h_2)$ in $D_{\mu} \times D_\mu$ such that $(t_n h_{n,1},t_nh_{n,2}) \in \Xi_{\mu} \times \Xi_\mu$.
By density, for any $\epsilon > 0$, there exist $c > 0$ and $\rho_i \in \cP_p$ for $i=1,2$ such that $\| h_i-\tilde{h}_i \|_{\dot{H}^{-1,p}(\mu*\gamma_{\sigma})} < \epsilon$ for $\tilde{h}_i = c (\rho_i - \mu)*\gamma_{\sigma}$. By scaling, Lemma \ref{lem: Haradard derivative null} holds with $(h_1,h_2)$ replaced by $(\tilde{h}_1,\tilde{h}_2)$.
Assume without loss of generality that  $n$ is large enough such that $\| h_{n,i} - h_i \|_{\dot{H}^{-1,p}(\mu*\gamma_{\sigma})} < \epsilon$ for $i=1,2$ and $c t_n \le 1/2$. The density of $\mu*\gamma_{\sigma}+t_n \tilde{h}_i = \big ((1-ct_n) \mu + ct_n \rho_i\big )*\gamma_{\sigma}$ w.r.t. $\mu*\gamma_{\sigma}$ is 
\[
\frac{d(\mu*\gamma_{\sigma}+t_n \tilde{h}_i)}{d(\mu*\gamma_{\sigma})} \ge (1-ct_n) \ge \frac{1}{2}, \ i=1,2.
\]
Thus, by Proposition \ref{prop:duality-in-Wp}, we have
\[
\begin{split}
&\left | \frac{\Phi(t_n h_{n,1},t_nh_{n,2})}{t_n} - \frac{\Phi(t_n \tilde{h}_1,t_n\tilde{h}_2)}{t_n} \right| \\
&\le \sum_{i=1}^2 \frac{\Wp(\mu*\gamma_{\sigma}+t_n h_{n,i},\mu*\gamma_{\sigma}+t_n\tilde{h}_i)}{t_n} \\
&\lesssim \sum_{i=1}^2 \| h_{n,i} - \tilde{h}_i \|_{\dot{H}^{-1,p}(\mu*\gamma_{\sigma})} \\
&\le \sum_{i=1}^2 \big ( \| h_{n,i} - h_i \|_{\dot{H}^{-1,p}(\mu*\gamma_{\sigma})} + \| h_i-\tilde{h}_i \|_{\dot{H}^{-1,p}(\mu*\gamma_{\sigma})} \big)\\
&< 4\epsilon.
\end{split}
\]
Further, 
\[
\big| \| h_1-h_2 \|_{\dot{H}^{-1,p}(\mu*\gamma_{\sigma})} - \| \tilde{h}_1-\tilde{h}_2 \|_{\dot{H}^{-1,p}(\mu*\gamma_{\sigma})}\big| \le \sum_{i=1}^2 \| h_i-\tilde{h}_i \|_{\dot{H}^{-1,p}(\mu*\gamma_{\sigma})} < 2\epsilon. 
\]
Thus, using the result of Lemma \ref{lem: Haradard derivative null}, we conclude that
\[
\limsup_{n \to \infty} \left | \frac{\Phi(t_n h_{n,1},t_nh_{n,2})}{t_n} - \| h_1-h_2 \|_{\dot{H}^{-1,p}(\mu*\gamma_{\sigma})}\right | \lesssim  \epsilon. 
\]
Since $\epsilon > 0$ is arbitrary, we obtain the desired conclusion. 
\qed

\subsection{Proofs for Section \ref{sec: alternative}}
\subsubsection{Proof of Lemma \ref{lem: Hadamard derivative alternative}}

The proof of Lemma \ref{lem: Hadamard derivative alternative} relies on the following technical lemma concerning regularity of optimal transport potentials.  
Recall that any locally Lipschitz function on $\R^d$ is differentiable a.e. by the Rademacher theorem (cf. \cite{evans2018measure}). Here and in what follows a.e. is taken w.r.t. the Lebesgue measure. 

\begin{lemma}[Regularity of optimal transport potential]
\label{lem: potential}
Let $1 < p < \infty$. 
Suppose that $\mu \in \cP_p$ and $\nu \in \cP$ is $\beta$-sub-Weibull for some $\beta \in (0,2]$. Let $g$ be an optimal transport potential from $\mu*\gamma_{\sigma}$ to $\nu$ for $\Wp^p$. Then there exists a constant $C$ that depends only on $p,d,\sigma,\beta$, upper bounds on $\E_{\mu}[|X|]$ and $\| |Y| \|_{\psi_{\beta}}$ for $Y \sim \nu$, and a lower bound on $\int \phi_{\sigma} d\mu$, such that 
 \[
 \begin{cases}
&\text{$g$ is locally Lipschitz}, \\
&| g (x) - g(0) | \le C (1+|x|^{\frac{2p}{\beta}})|x|, \ \forall x \in \R^d,  \\
&| \nabla g(x) | \le C (1+|x|^{\frac{2p}{\beta}}) \ \text{for a.e.} \ x \in \R^d.
 \end{cases}
 \]
\end{lemma}

The proof of Lemma \ref{lem: potential} borrows ideas from Lemmas 9 and 10 and Theorem 11 in the recent work by \cite{manole2021}, which in turn build on \cite{gangbo1996,colombo2021}.

\begin{proof}[Proof of \cref{lem: potential}]
By Theorem 11 in \cite{manole2021}, there exists a constant $C_1$ depending only on $p,d,\beta$ and an upper bound on $\| |Y| \|_{\psi_{\beta}}$ for $Y \sim \nu$, such that 
\[
\sup_{y \in \partial^c g(x)} |y| \le C_1 \left \{ (|x|+1)^{\frac{p}{p-1}} \bigvee \sup_{y: |x-y| \le 2} \left [ \log \left (\frac{1}{(\mu*\gamma_{\sigma})(B_y)} \right)\right]^{\frac{p}{\beta (p-1)}} \right\}, \quad x \in \R^d, 
\]
where $\partial^c g (x) = \{ y \in \R^d : c(z,y) - g(z) \ge c(x,y) - g(x), \ \forall z \in \R^d \}$ is the $c$--superdifferential of $g$ at $x$ for the cost function $c(x,y) = |x-y|^p$, and $B_y = B(y,1) = \{ x \in \R^d: |x-y| \le 1 \}$.

Next, by Proposition 2 in \cite{polyanskiy2016},  $\mu*\gamma_{\sigma}$ has Lebesgue density $f_{\mu}$ that is $(c_1,c_2)$-regular with $c_1 = 3/\sigma^2$ and $c_2 = 4\E_{\mu}[|X|]/\sigma^2$, i.e.,
\[
\big | \nabla \log f_{\mu}(x) \big| \le c_1 |x| + c_2, \quad \forall x \in \R^d. 
\]
From the proof of Lemma 10 in \cite{manole2021}, we have 
\begin{equation}
f_{\mu}(x) \ge e^{-c_2^2}f_{\mu}(0)e^{-(1+c_1)|x|^2}, \quad \forall x \in \R^d. 
\label{eq: density lower bound}
\end{equation}
Thus, whenever $|x-y| \le 2$, 
\[
(\mu*\gamma_{\sigma})(B_y) = \int_{B_y} f_{\mu}(z) dz  \ge \inf_{z \in B_y} f_\mu(z) \times \int_{B_y} dz \ge c_3 e^{-c_2^2}f_{\mu}(0)e^{-2(1+c_1)(|x|^2+9)},
\]
where $c_3$ is a constant that depends only on $d$. Conclude that there exists a  constant $C_2$ depending only on $p,d,\sigma,\beta$, upper bounds on $\E_{\mu}[|X|]$ and $\| |Y| \|_{\psi_{\beta}}$ for $Y \sim \nu$, and a lower bound on $f_{\mu}(0)$, such that
\[
\sup_{y \in \partial^c g(x)} |y| \le C_2 (1+|x|^{\frac{2p}{\beta(p-1)}}), \quad \forall x \in \R^d. 
\]

The rest of the proof  mirrors the latter half of the proof of Lemma 9 in \cite{manole2021}. Since $g \in L^1(\mu*\gamma_\sigma)$ and $\mu*\gamma_{\sigma}$ is equivalent to the Lebesgue measure (i.e., $\mu*\gamma_\sigma \ll dx$ and $dx \ll \mu*\gamma_\sigma$), $g(x) > -\infty$ for a.e. $x \in \R^d$. Since any open convex set in $\R^d$ agrees with the interior of its closure (cf. Proposition 6.2.10 in \cite{dudley2002}), the convex hull of $\{ x : g(x) > -\infty \}$ agrees with $\R^d$. Thus, by Lemma \ref{lem: duality} (ii) (Theorem 3.3 in \cite{gangbo1996}), $g$ is locally Lipschitz on $\R^d$.
Further, by Propositon C.4 in \cite{gangbo1996}, 
$\partial^c g(x)$ is nonempty for all $x \in \R^d$. For any $x \in \R^d$ and $y \in \partial^c g(x)$,  
\[
g(x) = c(x,y) - g^c (y).
\]
Thus, for any $x' \in \R^d$, 
\[
\begin{split}
g(x') - g(x) &\le c(x',y) - g^c(y) - [c(x,y) - g^c(y)] \\
&=c(x',y) - c(x,y) \\
&=|x'-y|^p - |x-y|^p \\
&\le p (|x-y|^{p-1} \vee |x'-y|^{p-1})|x-x'| \\
&\le C_3 \big [ 1+(|x|\vee |x'|)^{\frac{2p}{\beta}} \big ] |x-x'|,
\end{split}
\]
where  $C_3$ depends only on $p,\beta,C_2$. Interchanging $x$ and $x'$, we conclude that
\begin{equation}
\label{eq: local modulus}
|g(x)-g(x')| \le C_3 \big [ 1+(|x|\vee |x'|)^{\frac{2p}{\beta}} \big ] |x-x'|, \ x,x' \in \R^d,
\end{equation}
which implies the desired conclusion. 
\end{proof}

\begin{proof}[Proof of \cref{lem: Hadamard derivative alternative}]
 Let $\mu_t = (\mu+t(\rho-\mu))*\gamma_{\sigma} = (1-t)\mu*\gamma_{\sigma} + t\rho*\gamma_{\sigma}$ for $t \in [0,1]$, and let $g_t$ be an optimal transport potential from $\mu_t$ to $\nu$. Without loss of generality, we may normalize $g_t$ in such a way that $g_t (0) = 0$ for $t \in [0,1]$. 
 
We will apply Lemma \ref{lem: potential} with $(\mu,\nu)$ replaced with $\big ((1-t)\mu+t\rho,\nu \big)$ for $t \in [0,1/2]$. It is not difficult to see that, as long as $t \in [0,1/2]$,
\[
 \E_{(1-t) \mu + t \rho}[|X|] \le \E_{\mu}[|X|] + \E_{\rho}[|X|] \quad \text{and} \quad \int_{\R^d} \phi_{\sigma} d\big((1-t) \mu + t \rho \big) \ge \frac{1}{2} \int_{\R^d} \phi_{\sigma} d\mu. 
 \]
 Thus, by Lemma \ref{lem: potential}, there exist constants $C$ and $K$ independent of $t$ such that for every $t \in [0,1/2]$,
 \[
 \begin{cases}
&\text{$g_t$ is locally Lipschitz}, \\
&| g_t (x) | \le C (1+|x|^{K})|x|, \ \forall x \in \R^d,  \\
&| \nabla g_t(x) | \le C (1+|x|^{K}) \ \text{for a.e.} \ x \in \R^d.
 \end{cases}
 \]

 By duality (Lemma \ref{lem: duality} (i)), we have with $h=(\rho-\mu)*\gamma_\sigma$,
 \[
 \begin{split}
 \Wp^p (\mu_t,\nu) 
 &\ge \int_{\R^d} g_0 d\mu_t + \int_{\R^d} g_0^c d\nu \\
 &=\int_{\R^d} g_0 d\mu_0 + \int_{\R^d} g_0^c d\nu + t \int_{\R^d} g_0 dh \\
 &=\Wp^p(\mu_0,\nu) + t \int_{\R^d} g_0 dh, 
 \end{split}
 \]
 so that 
 \[
 \liminf_{t \downarrow 0}\frac{\Wp^p (\mu_{t},\nu) - \Wp^p(\mu_0,\nu)}{t} \ge \int_{\R^d} g_0 dh.
 \]
 
 Second, by construction, 
 \[
  \begin{split}
 \Wp^p (\mu_t,\nu) &= \int_{\R^d} g_t d\mu_t + \int_{\R^d} g_t^c d\nu \\
 &=\int_{\R^d} g_t d\mu_0 + \int_{\R^d} g_t^c d\nu + t \int_{\R^d} g_t dh \\
 &\le \int_{\R^d} g_0 d\mu_0 + \int_{\R^d} g_0^c d\nu + t \int_{\R^d} g_t dh \\
 &=\Wp^p(\mu_0,\nu) + t \int_{\R^d} g_t dh.
 \end{split}
 \]
 Pick any $t_n \downarrow 0$. Since $\mu_0 = \mu*\gamma_{\sigma} \ll dx$, $\mu_0$ has full support $\R^d$, and $\mu_{t_n} \stackrel{w}{\to} \mu_0$, we have by Theorem 3.4 in \cite{delbarrio2021} that there exists some sequence of constants $a_n$ such that $
 g_{t_n} -a_n \to g_0$ pointwise. 
 Since we have normalized $g_t$ in such a way that $g_t(0)=0$, we have $a_n \to 0$, i.e., $g_{t_n} \to g_0$ pointwise. 
 Further, since $|g_t(x)| \le C(1+|x|^K)|x|$ for all $t \in [0,1/2]$, the dominated convergence theorem yields that 
 \[
 \int_{\R^d} g_{t_n} dh \to  \int_{\R^d} g_0 dh. 
 \]
 Conclude that 
 \[
  \limsup_{n \to \infty}\frac{\Wp^p (\mu_{t_n},\nu) - \Wp^p(\mu_0,\nu)}{t_n} \le \int_{\R^d} g_0 dh.
 \]
 This completes the proof
\end{proof}

\subsubsection{Proof of Proposition \ref{prop: Hadamard derivative alternative}}
\underline{Part (i)}. We first note that $\dot{H}^{1,q}(\mu*\gamma_\sigma)$ is a function space over $\R^d$. To see this, observe that if we choose a reference measure $\kappa$ to be an isotropic Gaussian distribution with sufficiently small variance parameter, then the relative density $d(\mu*\gamma_\sigma)/d\kappa$ is bounded away from zero. Indeed, for $\kappa = \gamma_{\sigma/\sqrt{2}}$, we have 
\[
\begin{split}
\frac{d(\mu*\gamma_\sigma)}{d\gamma_{\sigma/\sqrt{2}}}(x) &= 2^{-d/2} \int_{\R^d} e^{-|x-y|^2/(2\sigma^2) +|x|^2/\sigma^2} d\mu(y) \\
&=2^{-d/2} \int_{\R^d} e^{|x+y|^2/(2\sigma^2) -|y|^2/\sigma^2} d\mu(y) \\
&\ge 2^{-d/2} e^{-\E_{\mu}[|X|^2]/\sigma^2}
\end{split}
\]
by Jensen's inequality, 
which guarantees that $\dot{H}^{1,q}(\mu*\gamma_\sigma)$ is a function space over $\R^d$ in view of \cref{rem:explicit-construction}. 

By regularity of $g$ from Lemma \ref{lem: potential}, we know that $g$ is locally Lipschitz and $\| g \|_{L^q(\mu*\gamma_\sigma)} \vee \| \nabla g \|_{L^{q}(\mu*\gamma_\sigma;\R^d)} < \infty$ (the latter alone does not automatically guarantee $g \in \dot{H}^{1,q}(\mu*\gamma_\sigma)$). As in Proposition 1.5.2 in \cite{bogachev1998}, choose a sequence $\zeta_j \in C_0^\infty$ with the following property: 
\[
0 \le \zeta_j \le 1, \ \zeta_j(x)=1 \ \text{if $|x| \le j$}, \ \sup_{j,x}|\nabla \zeta_j(x)| < \infty.
\]
Let $\varphi_j = \zeta_j g$.  Each $\varphi_j$ belongs to the ordinary Sobolev $(1,q)$-space w.r.t. the Lebesgue measure, so $\nabla \varphi_j$ can be approximated by gradients of $C_0^\infty$ functions under $\| \cdot \|_{L^q(dx;\R^d)}$ (cf. \cite{adams1975}, Corollary 3.23).  Since $\mu*\gamma_{\sigma}$ has a bounded Lebesgue density, this shows that $\varphi_j \in \dot{H}^{1,q}(\mu*\gamma_\sigma)$. 
Now, 
\[
\| \nabla \varphi_j - \nabla g \|_{L^p(\mu*\gamma_\sigma;\R^d)} \le \| (\nabla \zeta_j) g \|_{L^q(\mu*\gamma_{\sigma};\R^d)} + \| (\zeta_j-1) \nabla g \|_{L^q(\mu*\gamma_{\sigma};\R^d)} \to 0
\]
as $j \to \infty$, implying that $g \in \dot{H}^{1,q}(\mu*\gamma_\sigma)$.

\underline{Part (ii)}. 
Pick any $h \in T_{\Lambda_{\mu}}(0), t_n \downarrow 0$, and $h_n \to h$ in $D_{\mu}$ such that $t_n h_n \in \Lambda_{\mu}$. For any $\epsilon > 0$, there exist some constant $c>0$ and sub-Weibull $\rho \in \cP$ such that $\| h-\tilde{h} \|_{\dot{H}^{-1,p}(\mu*\gamma_{\sigma})} < \epsilon$ for $\tilde{h}=c(\rho-\mu)*\gamma_{\sigma}$.

Observe that 
\[
\begin{split}
&|\Psi (t_n h_n) - \Psi(t_n\tilde{h})| = |\Wp^p(\mu*\gamma_{\sigma}+t_nh_n,\nu*\gamma_\sigma) - \Wp^p(\mu*\gamma_{\sigma}+t_n\tilde{h},\nu*\gamma_\sigma)| \\
&\le p \big (\Wp^{p-1}(\mu*\gamma_{\sigma}+t_nh_n,\nu*\gamma_\sigma) \vee \Wp^{p-1}(\mu*\gamma_{\sigma}+t_n\tilde{h},\nu*\gamma_\sigma) \big)  \\ &\quad \times \Wp(\mu*\gamma_{\sigma}+t_nh_n,\mu*\gamma_{\sigma}+t_n\tilde{h}).
\end{split}
\]
Assume that $n$ is large enough so that $ct_n \le 1/2$ and $\| h_n - h \|_{\dot{H}^{-1,p}(\mu*\gamma_\sigma)} < \epsilon$. The density of $\mu*\gamma_{\sigma}+t_n\tilde{h} = \big ((1-ct_n)\mu + ct_n \rho\big)*\gamma_\sigma$ w.r.t. $\mu*\gamma_{\sigma}$ is 
\[
\frac{d(\mu*\gamma_{\sigma}+t_n\tilde{h})}{d(\mu*\gamma_{\sigma})} \ge 1-ct_n  \ge \frac{1}{2}.
\]
Thus, by Proposition \ref{prop:duality-in-Wp}, 
\[
 \Wp(\mu*\gamma_{\sigma}+t_nh_n,\mu*\gamma_{\sigma}+t_n\tilde{h}) \lesssim t_n \| h_n - \tilde{h} \|_{\dot{H}^{-1,p}(\mu*\gamma_\sigma)} < 2t_n \epsilon. 
\]
Also, by Proposition \ref{prop:duality-in-Wp}, 
\[
\begin{split}
\Wp(\mu*\gamma_{\sigma}+t_n\tilde{h},\nu*\gamma_\sigma) &\le \Wp(\mu*\gamma_{\sigma}+t_n\tilde{h},\mu*\gamma_\sigma) + \Wp(\mu*\gamma_\sigma,\nu*\gamma_\sigma)\\
&\lesssim t_n \| \tilde{h} \|_{\dot{H}^{-1,p}(\mu*\gamma_\sigma)} + \Wp(\mu*\gamma_\sigma,\nu*\gamma_\sigma) = O(1).
\end{split}
\]
Likewise, $\Wp(\mu*\gamma_{\sigma}+t_nh_n,\nu*\gamma_\sigma) = O(1)$. Conclude that 
\[
\limsup_{n \to \infty} |\Psi (t_n h_n) - \Psi(t_n\tilde{h})|/t_n \lesssim \epsilon. 
\]
Further, $|h(g)-\tilde{h}(g)| \le \| g \|_{\dot{H}^{1,q}(\mu*\gamma_\sigma)} \| h-\tilde{h} \|_{\dot{H}^{-1,p}(\mu*\gamma_\sigma)} \lesssim \epsilon$. Combining Lemma \ref{lem: Hadamard derivative alternative}, we conclude that 
\[
\limsup_{n \to \infty} \left |\frac{\Psi (t_n h_n) - \Psi(0)}{t_n} - h(g) \right | \lesssim \epsilon. 
\]
This completes the proof. 
\qed

\subsection{Proofs for Section \ref{sec: bootstrap}}

\subsubsection{Proof of \cref{prop: bootstrap}}
We first prove the following lemma. We note that the empirical distributions $\hat{\mu}_n^B$ and $\hat{\mu}_n$ are finitely discrete, so $\sqrt{n}(\hat{\mu}_n^B-\hat{\mu}_n)*\gamma_\sigma$ defines a random variable with values in $\dot{H}^{-1,p}(\mu*\gamma_\sigma)$  (cf. (\ref{eq: derivative}) and Step 3 of the proof of \cref{thm:limit-distribution}). Let $\mathfrak{L}_n^B = \mathfrak{L}_n^B(X_1,\dots,X_n)$ denote  its (regular) conditional law given the data (which exists as $\dot{H}^{-1,p}(\mu*\gamma_\sigma)$ is a separable Banach space; cf. Chapter 11 in \cite{dudley2002}). 
\begin{lemma}
\label{lem: bootstrap}
Suppose that $\mu$ satisfies Condition (\ref{eq:moment-condition}). Then,  we have $\mathfrak{L}_n^B  \stackrel{w}{\to} \Prob \circ G_\mu^{-1}$
almost surely. 
\end{lemma}

\begin{proof}[Proof of \cref{lem: bootstrap}]
From the proof of \cref{thm:limit-distribution}, the function class $\mathcal{U}*\phi_{\sigma}$ is $\mu$-Donsker with a $\mu$-square integrable envelope. The rest of the proof follows from the Gin\'e-Zinn theorem for the bootstrap (cf. Theorem 3.6.2 in \cite{vandervaart1996book}) and repeating the arguments in Steps 2 and 3 in the proof of \cref{thm:limit-distribution}. 
\end{proof}

\begin{proof}[Proof of \cref{prop: bootstrap}]
\underline{Part (i)}. Assume without loss of generality that $\mu$ is not a point mass.
We first note that the limit variable $\| G_\mu \|_{\dot{H}^{-1,p}(\mu*\gamma)}$ has  a continuous distribution function. This can be verified, e.g., analogously to the proof of Lemma 1 in \cite{sadhu2021}. Thus, it suffices to prove the convergence in probability (\ref{eq: bootstrap consistency null}) for each fixed $t \ge 0$ (cf. Problem 23.1 in \cite{vandervaart1998}).

Let $T_n = (\hat{\mu}_n-\mu)*\gamma_\sigma$ and $T_n^B = (\hat{\mu}^B_n - \mu)*\gamma_\sigma$. By \cref{thm:limit-distribution} and Lemma \ref{lem: bootstrap}, we know that 
\[
\begin{split}
\big(\sqrt{n}T_n^B,\sqrt{n}T_n\big) &= \big(\sqrt{n}(\hat{\mu}_n^B - \hat{\mu}_n)*\gamma_\sigma + \sqrt{n}(\hat{\mu}_n - \mu)*\gamma_\sigma, \sqrt{n}(\hat{\mu}_n - \mu)*\gamma_\sigma \big)\\
&\stackrel{d}{\to} \big(G_{\mu}'+G_{\mu},G_\mu \big) \quad \text{in} \quad \dot{H}^{-1,p}(\mu*\gamma_\sigma) \times \dot{H}^{-1,p}(\mu*\gamma_\sigma)
\end{split}
\]
unconditionally, where $G_\mu'$ is an independent copy of $G_\mu$ (cf. Theorem 2.2 in \cite{kosorok2008}).  Thus, by \cref{prop: Hadamard derivative alternative} and the second claim of the functional delta method (Lemma \ref{lem: functional delta method}), we see that  
\[
\begin{split}
\sqrt{n}\GWp(\hat{\mu}^B_n,\hat{\mu}_n) &= \sqrt{n}\Phi (T_n^B,T_n) = \Phi_{(0,0)}'(\sqrt{n}T_n^B,\sqrt{n}T_n) + R_n \\
&=\| \sqrt{n}(T_n^B-T_n) \|_{\dot{H}^{-1,p}(\mu*\gamma_\sigma)} + R_n \\
&=\| \sqrt{n}(\hat{\mu}_n^B - \hat{\mu}_n)*\gamma_\sigma \|_{\dot{H}^{-1,p}(\mu*\gamma_\sigma)} + R_n.
\end{split}
\]
Here $R_n = o_{\Prob}(1)$ unconditonally. Choose $\epsilon_n \to 0$ such that $\Prob(|R_n| > \epsilon_n) \to 0$. 
By Markov's inequality, we have $\Prob^B(|R_n| > \epsilon_n) \stackrel{\Prob}{\to} 0$. By \cref{lem: bootstrap} and the continuous mapping theorem, we also have 
\[
\sup_{t \ge 0} \left | \Prob^B \left (\| \sqrt{n}(\hat{\mu}_n^B - \hat{\mu}_n)*\gamma_\sigma \|_{\dot{H}^{-1,p}(\mu*\gamma_\sigma)} \le t \right ) - \Prob \left ( \| G_\mu \|_{\dot{H}^{-1,p}(\mu*\gamma_{\sigma})}\le t \right ) \right | \stackrel{\Prob}{\to} 0. 
\]
 Thus, for each $t \ge 0$, 
\[
\begin{split}
&\Prob^B\left ( \sqrt{n}\GWp(\hat{\mu}^B_n,\hat{\mu}_n) \le t \right ) \\
&\le \Prob^B\left (\| \sqrt{n}(\hat{\mu}_n^B - \hat{\mu}_n)*\gamma_\sigma \|_{\dot{H}^{-1,p}(\mu*\gamma_\sigma)} \le t + \epsilon_n \right) + \Prob^B(|R_n| > \epsilon_n) \\
&= \Prob\left (\| G_\mu \|_{\dot{H}^{-1,p}(\mu*\gamma)} \le t+\epsilon_n \right) + o_{\Prob}(1) \\
&=\Prob\left (\| G_\mu \|_{\dot{H}^{-1,p}(\mu*\gamma)} \le t\right) + o_{\Prob}(1).
\end{split}
\]
The reverse inequality follows similarly. 

\underline{Part (ii)}. The argument is analogous to Part (i). Observe that, by \cref{prop: Hadamard derivative alternative}, 
\[
\begin{split}
\sqrt{n}\big (\SWp(\hat{\mu}^B_n,\nu)-\SWp(\hat{\mu}_n,\nu) \big)&= \sqrt{n}\big(\Psi (T_n^B) - \Psi (T_n) \big)\\ &= \Psi_{0}'(\sqrt{n}T_n^B) - \Psi_0'(\sqrt{n}T_n) + o_{\Prob}(1) \\
&=\sqrt{n}(T_n^B-T_n)(g) + o_{\Prob}(1) \\
&=\sqrt{n}\big((\hat{\mu}_n^B - \hat{\mu}_n)*\gamma_\sigma\big)(g) + o_{\Prob}(1).
\end{split}
\]
Taking $p$th root and applying the delta method, we have 
\[
\sqrt{n}\big (\GWp(\hat{\mu}^B_n,\nu)-\GWp(\hat{\mu}_n,\nu) \big)
=\frac{1}{p[\GWp(\mu,\nu)]^{p-1}} \cdot \sqrt{n}\big((\hat{\mu}_n^B - \hat{\mu}_n)*\gamma_\sigma\big)(g) + o_{\Prob}(1).
\]
The rest of the proof is completely analogous to Part (i). 
\end{proof}

\begin{proof}[Proof of \cref{prop: bootstrap case 2}]
By Lemma \ref{lem: bootstrap} and Example 1.4.6 in \cite{vandervaart1996book}, the conditional law of $\big(\sqrt{n}(\hat{\mu}_n^B-\hat{\mu}_n)*\gamma_\sigma,\sqrt{n}(\hat{\nu}_n^B-\hat{\nu}_n)*\gamma_\sigma\big)$ given the data converges weakly to the law of $(G_\mu,G_\nu)$ in $\dot{H}^{-1,p}(\mu*\gamma_\sigma) \times \dot{H}^{-1,p}(\nu*\gamma_\sigma)$ almost surely, where $G_\mu$ and $G_\nu$ are independent. By Theorem 2.2 in \cite{kosorok2008}, for  $T_{n,1}^B = (\hat{\mu}_n^B-\mu)*\gamma_\sigma$ and $T_{n,2}^B = (\hat{\nu}_n^B - \nu)*\gamma_\sigma$, we have
\[
\begin{split}
&\big(\sqrt{n}T_{n,1}^B,\sqrt{n}T_{n,2}^B\big) \\
&= \big(\sqrt{n}(\hat{\mu}_n^B - \hat{\mu}_n)*\gamma_\sigma + \sqrt{n}(\hat{\mu}_n - \mu)*\gamma_\sigma, \sqrt{n}(\hat{\nu}_n^B - \hat{\nu}_n)*\gamma_\sigma + \sqrt{n}(\hat{\nu}_n - \nu)*\gamma_\sigma \big)\\
&\stackrel{d}{\to} \big(G_\mu+G_\mu',G_\nu+G_\nu' \big) \quad \text{in} \quad \dot{H}^{-1,p}(\mu*\gamma_\sigma) \times \dot{H}^{-1,p}(\nu*\gamma_\sigma)
\end{split}
\]
unconditionally, where $G_\mu',G_\nu'$ are copies of $G_\mu,G_\nu$, respectively, and $G_\mu,G_\mu',G_\nu,G_\nu'$ are independent. Thus, by \cref{prop: Hadamard derivative alternative two sample} and \cref{lem: functional delta method}, for $T_{n,1} = (\hat{\mu}_n-\mu)*\gamma_\sigma$ and $T_{n,2} = (\hat{\nu}_n - \nu)*\gamma_\sigma$, we have
\[
\begin{split}
&\sqrt{n}\big (\SWp(\hat{\mu}^B_n,\hat{\nu}_n^B)-\SWp(\hat{\mu}_n,\hat{\nu}_n) \big)\\
&= \sqrt{n}\big(\Upsilon (T_{n,1}^B,T_{n,2}^B) - \Upsilon (T_{n,1},T_{n,2}) \big) \\
&= \Upsilon_{(0,0)}'(\sqrt{n}T_{n,1}^B,\sqrt{n}T_{n,2}^B) - \Upsilon_{(0,0)}'(\sqrt{n}T_{n,1},\sqrt{n}T_{n,2}) + o_{\Prob}(1) \\
&=\sqrt{n}(T_{n,1}^B-T_{n,1})(g) + \sqrt{n}(T_{n,2}^B-T_{n,2})(g^c) + o_{\Prob}(1) \\
&=\sqrt{n}\big((\hat{\mu}_n^B - \hat{\mu}_n)*\gamma_\sigma\big)(g) +  \sqrt{n}\big((\hat{\nu}_n^B - \hat{\nu}_n)*\gamma_\sigma\big)(g^c) + o_{\Prob}(1).
\end{split}
\]
The rest of the proof is analogous to \cref{prop: bootstrap} Part (ii). 
\end{proof}

\begin{proof}[Proof of \cref{prop: bootstrap case 3}]
It is not difficult to see that $\rho$ satisfies Condition (\ref{eq:moment-condition}) and  $\sqrt{2n}(\hat{\rho}_n-\rho)*\gamma_\sigma \stackrel{d}{\to} G_{\rho}$ in $\dot{H}^{-1,p}(\rho*\gamma_\sigma)$. 
By Theorem 3.7.7 and  Example 1.4.6 in \cite{vandervaart1996book}, the conditional law of $\big (\sqrt{n}(\hat{\rho}_{n,1}^B-\hat{\rho}_n)*\gamma_\sigma,\sqrt{n}(\hat{\rho}_{n,2}^B-\hat{\rho}_n)*\gamma_\sigma \big)$ given the data converges weakly to the law of $(G_\rho,G_\rho')$ in $\dot{H}^{-1,p}(\rho*\gamma_\sigma) \times \dot{H}^{-1,p}(\rho*\gamma_\sigma)$ almost surely, where $G_\rho'$ is an independent copy of $G_\rho$. Thus, arguing as in the proof of \cref{prop: bootstrap case 2}, for $T_{n,j}^B = (\hat{\rho}_{n,j}^B-\rho)*\gamma_\sigma$ ($j=1,2$), we have 
\[
(\sqrt{n}T_{n,1}^B,\sqrt{n}T_{n,2}^B) \stackrel{d}{\to} (G_\rho^1+G_\rho^2/\sqrt{2},G_\rho^3+G_\rho^4/\sqrt{2}) \quad \text{in} \quad \dot{H}^{-1,p}(\rho*\gamma_\sigma) \times \dot{H}^{-1,p}(\rho*\gamma_\sigma)
\]
unconditionally, where $G_\rho^1,\dots,G_\rho^4$ are independent copies of $G_\rho$.
Define $\Phi$ by replacing $\mu$ with $\rho$ in Section \ref{sec: null case}. 
Then,  by \cref{prop: Haradard derivative null} and the second claim of the functional delta method (Lemma \ref{lem: functional delta method}), we see that  
\[
\begin{split}
\sqrt{n}\GWp(\hat{\rho}^B_{n,1},\hat{\rho}_{n,2}^B) &= \sqrt{n}\Phi (T_{n,1}^B,T_{n,2}^B) = \Phi_{(0,0)}'(\sqrt{n}T_{n,1}^B,\sqrt{n}T_{n,2}^B) + o_{\Prob}(1) \\
&=\| \sqrt{n}(T_{n,1}^B-T_{n,2}^B) \|_{\dot{H}^{-1,p}(\rho*\gamma_\sigma)} + o_{\Prob}(1) \\
&=\| \sqrt{n}(\hat{\rho}_{n,1}^B-\hat{\rho}_n)*\gamma_\sigma - \sqrt{n}(\hat{\rho}_{n,2}^B-\hat{\rho}_n)*\gamma_\sigma \|_{\dot{H}^{-1,p}(\rho*\gamma_\sigma)} + o_{\Prob}(1).
\end{split}
\]
The rest of the proof is analogous to \cref{prop: bootstrap} Part (i). 
\end{proof} 

\subsection{Proof of \cref{thm: MDE limit}}
\subsubsection{Preliminary lemmas}

Recall the notation $\Xi_\mu$ and $D_\mu$ that appeared in \cref{sec: null case}.
\begin{lemma}
Let $\mu\in\cP_p$ for $1<p<\infty$. Under Assumption \ref{as:MSE}, the map 
    \[
        (h,\theta)\in \Xi_{\mu}\times N_0\mapsto{\sf W}_p\left(\mu*\gamma_{\sigma}+h,\nu_{\theta}*\gamma_{\sigma}\right)
    \]
    is Hadamard directionally differentiable at $(0,\theta\opt)$ with derivative
    \[
        (h,\theta)\in \Xi_{\mu}\times N_0\mapsto \norm{h-\inp{\theta}{\mathfrak{D}}}_{\Hu}.
    \]
    Furthermore, the expansion
    \[
     {\sf W}_p(\mu*\gamma_{\sigma}+h,\nu_{\theta}*\gamma_{\sigma})=\norm{h-\inp{\theta-\theta\opt}{\mathfrak{D}}}_{\Hu}+r(h,\theta-\theta\opt),   
 \]
 holds, with remainder $r$ satisfying $r(th,t(\theta-\theta\opt)) = o(t)$ as $t\dn 0$ uniformly w.r.t. $(h,\theta)$ varying in $K\subset \Xi_{\mu}\times N_0$, a compact subset of $D_{\mu}\times \R^{d_0}$. 
\label{lem:HDD}
\end{lemma}

\begin{proof}
    Consider the map $\psi:(h,\theta)\in\Xi_{\mu}\times N_0\mapsto\big(h,(\nu_{\theta}-\nu_{\theta\opt})*\gamma_{\sigma}\big)\in \Xi_{\mu}\times \Xi_{\mu}$.
    The norm differentiability condition, Assumption \ref{as:MSE} (vi), establishes Fr\'echet (hence Hadamard) directional differentiability of $\psi$ at $(0,\theta\opt)$ with
    \[
        \psi'_{(0,\theta\opt)}(h,\theta)=(h,\inp{\theta}{\mathfrak{D}})\in T_{\Xi_{\mu}\times \Xi_{\mu}}(0,0).
    \]
    The chain rule for Hadamard directional derivatives paired with \cref{prop: Haradard derivative null} yields
    \[
        (\Phi\circ \psi)'_{(0,\theta\opt)}(h,\theta)=\Phi'_{\psi(0,\theta\opt)}\circ\psi'_{(0,\theta\opt)}(h,\theta)=\Phi'_{(0,0)}(h,\inp{\theta}{\mathfrak{D}})=\norm{h-\inp{\theta}{\mathfrak{D}}}_{\Hu}.    
    \]
    The final assertion follows from compact directional differentiability of the composition \cite{shapiro1990}.
\end{proof}

\begin{lemma}
    Assume the setting of Lemma \ref{lem:HDD}.
    \begin{enumerate}
        \item[(i)] There exists a neighborhood $N_1$ of $\theta\opt$ with $\overline{N_1}\subset N_0$ such that 
           \[
               {\sf W}_p^{(\sigma)}(\hat \mu_n,\nu_{\theta})\geq \frac{C}{2}\abs{\theta-\theta\opt}-{\sf W}_p^{(\sigma)}(\hat \mu_n,\mu),\quad \forall \theta\in \overline{N_1},
            \]
            where $C>0$ is such that $\norm{\inp{t}{\mathfrak{D}}}_{\Hu}\geq C\abs{t}$ for every $t\in\R^{d_0}$.
        \item[(ii)] Let $\xi_n=O_{\mathbb{P}}(1)$ and $\Theta_n:=\left\{ \theta\in \overline{N_1}: \sqrt n \abs{\theta-\theta\opt}\leq \xi_n\right\}$; then, uniformly in $\theta \in \Theta_n$,   
    \[
        \sqrt n{\sf W}_p^{(\sigma)}\left(\hat\mu_n,\nu_{\theta}\right)=\big \| \mathbb{G}_n^{(\sigma)}-\sqrt n\inp{\theta-\theta\opt}{\mathfrak{D}}\big\| _{\Hu}+o_{\mathbb{P}}(1).  
    \]
\end{enumerate}
\label{lem:errprop}
\end{lemma}

\begin{proof} \underline{Part (i)}. 
    Assumption \ref{as:MSE} (vi) guarantees that there exists a constant $C>0$ such that $\norm{\inp{\theta-\theta\opt}{\mathfrak{D}}}_{\Hu}\geq C\abs{\theta-\theta\opt}$ for every $\theta\in N_0$. Let $N_1$ be an open ball of radius $\bar r$ centered at $\theta\opt$ whose closure is contained in $N_0$; then there exists $t_0>0$ such that, for every $t\leq t_0$, the remainder term $r$ of Lemma \ref{lem:HDD} satisfies $t^{-1}\abs{r(0,t(\theta-{\theta}\opt))}\leq C\bar r/2$ for every
    $\theta\in \partial N_1$. Hence,
    $\abs{r(0,\theta-\theta\opt)}\leq (C/2)\abs{\theta-\theta\opt}$ for every $\theta\in \overline{N_1}$. The triangle inequality yields, for any $\theta\in \overline N_1$,  
    \begin{align*}
        {\sf W}_{p}^{(\sigma)}(\hat \mu_n,\nu_{\theta})&\geq {\sf W}_{p}^{(\sigma)}(\mu,\nu_{\theta})-{\sf W}_{p}^{(\sigma)}(\hat \mu_n,\mu),
\\
&=\norm{\inp{\theta-\theta\opt}{\mathfrak{D}}}_{\Hu}+r(0,\theta-\theta\opt)-{\sf W}_{p}^{(\sigma)}(\hat \mu_n,\mu),
\\
&\geq  \frac{C}{2}\abs{\theta-\theta\opt}-{\sf W}_{p}^{(\sigma)}(\hat \mu_n,\mu).
\end{align*}
 
\underline{Part (ii)}.  Since $\mathbb{G}_n^{(\sigma)} \stackrel{d}{\to} G_{\mu}$ in $\Hu$ and $G_{\mu}$ is tight, the sequence $\mathbb{G}_n^{(\sigma)}$ is uniformly tight. Pick any $\epsilon,\delta>0$. By uniform tightness, there exists a compact set $K_{\epsilon}\subset \Hu$ such that $\Prob(\mathbb{G}_n^{(\sigma)}\in K_{\epsilon})\geq 1-\epsilon/2$ for every $n\in\N$. Further, since $\xi_n=O_{\mathbb{P}}(1)$, there exists $M_{\epsilon}>0$ such that $\Prob(\abs{\xi_n}\leq M_{\epsilon})\geq
1-\epsilon/2$ for every $n\in\N$. Define the event $A_{n,\epsilon}=\big\{ \mathbb{G}_n^{(\sigma)}\in K_{\epsilon} \big\}\cap \big\{ \abs{\xi_n}\leq M_{\epsilon} \big\}$. Observe that $\Prob(A_{n,\epsilon}) \ge 1-\epsilon$ for every $n \in \N$. Then, on this event $A_{n,\epsilon}$, it holds that $\Xi_n\subset \Theta_{n,\epsilon}:=\left\{ \theta\in \overline{N_1}:\sqrt n\abs{\theta-\theta\opt}\leq M_{\epsilon} \right\}$. Since $\Theta_{n,\epsilon}$ is compact,  we have, for every $\theta\in\Theta_{n,\epsilon}$,
 \[
     \sqrt n {\sf W}_p^{(\sigma)}(\hat\mu_n,\nu_{\theta})=\big \| \mathbb{G}_n^{(\sigma)}-\sqrt n\inp{\theta-\theta\opt}{\mathfrak{D}}\big \|_{\Hu}+\sqrt nr(n^{-1/2}\mathbb{G}_n^{(\sigma)},\theta-\theta\opt). %
 \]
Set $\zeta_n:=\sup_{\theta \in \Xi_{n,\epsilon}}|\sqrt nr(n^{-1/2}\mathbb{G}_n^{(\sigma)},\theta_n-\theta\opt)|$. Then, on the event $A_{n,\epsilon}$, 
 \[
     \abs{\zeta_n}\leq \sup_{h \in K_\epsilon, |u| \le M_\epsilon} \sqrt n\big | r(n^{-1/2}h,n^{-1/2}u)\big |,
     \]
     and the right hand side can be made less than $\delta$ for $n$ sufficiently large. 
    Hence, for every sufficiently large $n$,
     \begin{align*}
         \Prob\left( \abs{\zeta_n}>\delta  \right)&=\Prob\left( \left\{\abs{\zeta_n}>\delta\right\}\cap A_{n,\epsilon}  \right)+\Prob\left( \left\{\abs{\zeta_n}>\delta\right\}\cap A_{n,\epsilon}^c  \right)
         \\
         &\leq \Prob\left( \left\{\abs{\zeta_n}>\delta\right\}\cap A_{n,\epsilon}  \right)+\Prob\left(A_{n,\epsilon}^c  \right)
     \\
     &\leq \Prob\left(\abs{\xi_n}>M_{\epsilon}\right)
     \\
     &\leq \epsilon,
     \end{align*}
    that is, $\zeta_n =  o_{\mathbb{P}}(1)$. This implies the desired result. 
\end{proof}

\subsubsection{Proof of \cref{thm: MDE limit}}

\underline{Part (i)}.  Given the above lemmas, the proof follows closely \cite[Theorem 4.2]{pollard_min_dist_1980} or \cite[Appendix B.4]{goldfeld2020asymptotic}.  We first note that, under \cref{as:MSE}, there exists a sequence of  measurable estimators $\hat{\theta}_n$ such that $\GWp(\hat{\mu}_n,\nu_{\hat{\theta}_n}) = \inf_{\theta \in \Theta}\GWp(\hat{\mu}_n,\nu_\theta)$ and  $\hat{\theta}_n \stackrel{\Prob}{\to} \theta^\star$. This follows from a small modification to the proof of Theorems 2 and 3 in \cite{goldfeld2020asymptotic}. Thus, for any neighborhood $N$ of $\theta^\star$, 
    \[
        \inf_{\theta\in \Theta}{\sf W}_p^{(\sigma)}(\hat \mu_n,\nu_{\theta})=\inf_{\theta\in N}{\sf W}_p^{(\sigma)}(\hat \mu_n,\nu_{\theta})
    \]
with probability approaching one.

By Assumption \ref{as:MSE} (vi), there exists $C>0$ such that $\norm{\inp{t}{\mathfrak{D}}}_{\Hu}\geq C\abs{t}$ for every $t\in\R^{d_0}$. Thus, by Lemma \ref{lem:errprop} (i),  there exists a neighborhood $N_1$ of $\theta\opt$ with $\overline{N_1}\subset N_0$ such that
           \[
               {\sf W}_p^{(\sigma)}(\hat \mu_n,\nu_{\theta})\geq \frac{C}{2}\abs{\theta-\theta\opt}-{\sf W}_p^{(\sigma)}(\hat \mu_n,\mu),\quad \forall \theta\in \overline{N_1}. 
           \]
Set $\Theta_n:=\left\{ \theta\in \Theta:\sqrt n\abs{\theta-\theta\opt}\leq \xi_n \right\}$ with $\xi_n:=(4/C)\big \| \mathbb{G}_n^{(\sigma)} \big \|_{\Hu} = O_{\Prob}(1)$. By Lemma \ref{lem:errprop} (ii), the expansion    
\begin{equation}
        \sqrt n{\sf W}_p^{(\sigma)}\left(\hat\mu_n,\nu_{\theta}\right)=\big \| \mathbb{G}_n^{(\sigma)}-\sqrt n\inp{\theta-\theta\opt}{\mathfrak{D}}\big \| _{\Hu}+o_{\mathbb{P}}(1),
        \label{eq: expansion 2}
\end{equation}
    holds uniformly in $\theta\in N_1\cap \Theta_n$. Then, for arbitrary $\theta\in N_1\cap \Theta_n^c$,
    \begin{align*}
        {\sf W}_p^{(\sigma)}(\hat \mu_n,\nu_{\theta})&> \frac{C}{2}\frac{\xi_n}{\sqrt n}-{\sf W}_p^{(\sigma)}(\hat \mu_n,\mu)
        \\&= {\sf W}_p^{(\sigma)}(\hat \mu_n,\mu) + o_{\mathbb{P}}(n^{-1/2}),
\end{align*}
so that
\[
\begin{split}
    \inf_{\theta\in N_1\cap \Theta_n^c}{\sf W}_p^{(\sigma)}(\hat \mu_n,\nu_{\theta})&>{\sf W}_p^{(\sigma)}(\hat \mu_n,\mu)+o_{\mathbb P}(n^{-1/2})\\
    &\geq\inf_{\theta\in N_1\cap \Theta_n}{\sf W}_p^{(\sigma)}(\hat \mu_n,\nu_\theta)+o_{\mathbb P}(n^{-1/2}).
    \end{split}
\]
This shows that $\inf_{\theta\in \Theta}{\sf W}_p^{(\sigma)}(\hat \mu_n,\nu_{\theta}) = \inf_{\theta\in N_1\cap \Theta_n}{\sf W}_p^{(\sigma)}(\hat \mu_n,\nu_\theta)+o_{\mathbb P}(n^{-1/2})$.

Now, reparametrizing by $t=\sqrt n(\theta-\theta\opt)$ and setting $T_n:=\big \{ t\in\R^{d_0}:\abs{t}\leq \xi_n,\theta\opt+t/\sqrt n\in N_1 \big\}$ in (\ref{eq: expansion 2}),  we have
\[
    \inf_{\theta\in N_1\cap \Theta_n}\sqrt n{\sf W}_p^{(\sigma)}(\hat\mu_n,\nu_{\theta})=\inf_{t\in T_n}\big \| \mathbb{G}_n^{(\sigma)}-\inp{t}{\mathfrak{D}}\big \|_{\dot H^{-1,p}(\mu*\gamma_{\sigma})}+o_{\mathbb P}(1).
\]
Set $\mathfrak{g}_n:=\big \| \mathbb{G}_n^{(\sigma)}-\inp{\cdot}{\mathfrak{D}}\big \|_{\Hu}$. For any $t \in \R^{d_0}$ such that $\abs{t}>\xi_n$, we have
\[
    \mathfrak{g}_n(t)\geq C\abs{t}-\big \| \mathbb{G}_n^{(\sigma)}\big \|_{\Hu}>3\mathfrak{g}_n(0)\geq 3\inf_{\abs{t'}\leq \xi_n}\mathfrak{g}_n(t').
\]
Since $\left\{t\in\R^{d_0}:\abs{t}\leq \xi_n  \right\}\subset T_n$ with probability approaching one (as $\xi_n=O_{\mathbb P}(1)$), we have $\inf_{t \in T_n}\mathfrak{g}_n(t)=\inf_{t \in \R^{d_0}} \mathfrak{g}_n(t)$ with probability approaching one.  Conclude that
\[
    \inf_{\theta\in \Theta}\sqrt n{\sf W}_p^{(\sigma)}(\hat\mu_n,\nu_{\theta})=\inf_{t\in \R^{d_0}}\big \| \mathbb{G}_n^{(\sigma)}-\inp{t}{\mathfrak{D}}\big \|_{\dot H^{-1,p}(\mu*\gamma_{\sigma})}+o_{\mathbb P}(1).
\]
Finally, since the map $h\in\Hu\mapsto \inf_{t\in\R^{d_0}}\norm{h-\inp{t}{\mathfrak{D}}}_{\Hu}$ is continuous, the continuous mapping theorem yields
\[
\inf_{\theta\in \Theta}\sqrt n{\sf W}_p^{(\sigma)}(\hat\mu_n,\nu_{\theta})\stackrel{d}{\to}\inf_{t\in \R^{d_0}}\norm{G_{\mu}-\inp{t}{\mathfrak{D}}}_{\dot H^{-1,p}(\mu*\gamma_{\sigma})}.
\]
This completes the proof of Part (i). 

\underline{Part (ii)}.  From the proof of Theorem 3 in \cite{goldfeld2020asymptotic}, it is not difficult to show that $\hat \theta_n \stackrel{\Prob}{\to} \theta\opt$. Thus, $\hat \theta_n\in N_1$ with probability approaching one, where $N_1$ is the neighborhood of $\theta^\star$ given in the proof of Part (i), so that by the definition of $\hat{\theta}_n$ and Lemma \ref{lem:errprop} (ii),
\begin{equation}
        \underbrace{\inf_{\theta\in \Theta}\sqrt n {\sf W}_p^{(\sigma)}(\hat\mu_n,\nu_{\theta})}_{=O_{\mathbb{P}}(1)}+o_{\mathbb{P}}(1)\geq \sqrt n {\sf W}_p^{(\sigma)}(\hat \mu_n,\nu_{\hat {\theta}_n})\geq \frac{C}{2}\sqrt n |\hat \theta_n-\theta\opt|-\underbrace{\sqrt n {\sf W}_p^{(\sigma)}(\hat \mu_n,\mu)}_{=O_{\mathbb{P}}(1)},
        \label{eq: expansion}
\end{equation}
    with probability tending to one. This implies that $\sqrt n |\hat \theta_n-\theta\opt|=O_{\mathbb{P}}(1)$. Let $\mathbb{M}_n(t):=\big\|\mathbb{G}_n^{(\sigma)}-\inp{t}{\mathfrak{D}}\big \|_{\Hu}$ and $\mathbb{M}(t):=\norm{G_{\mu}-\inp{t}{\mathfrak{D}}}_{\Hu}$. Observe that $\mathbb{M}_n$ and $\mathbb{M}$ are convex in $t$. Again, from the proof of Part (i), since $\hat t_n:=\sqrt n(\hat \theta_n-\theta\opt)=O_{\mathbb{P}}(1)$, we have
    \[
        \sqrt{n}{\sf W}_p^{(\sigma)}(\hat \mu_n,\nu_{\hat \theta_n})=\mathbb{M}_n(\hat t_n)+o_{\mathbb{P}}(1).
    \]
 Hence,
    \[
    \begin{split}
        \mathbb{M}_n(\hat t_n)&=\sqrt{n}{\sf W}_p^{(\sigma)}(\hat \mu_n,\nu_{\hat \theta_n})+o_{\mathbb{P}}(1)\\
        &\leq \inf_{\theta\in \Theta}\sqrt{n}{\sf W}_p^{(\sigma)}(\hat \mu_n,\nu_\theta)+o_{\mathbb{P}}(1)\\
        &=\inf_{t \in \R^{d_0}}\mathbb{M}_n(t)+o_{\mathbb{P}}(1).
        \end{split}
    \]
    Since $\mathbb{G}_n^{(\sigma)}\stackrel{d}{\to} G_{\mu}$ in $\Hu$, $(\mathbb{M}_n(t_1),\dots,\mathbb{M}_n(t_k))\stackrel{d}{\to} (\mathbb{M}(t_1),\dots,\mathbb{M}(t_k))$ for any finite set of points $(t_i)_{i=1}^k\subset \R^{d_0}$ by the continuous mapping theorem. 
    Applying Theorem 1 in \cite{kato2009asymptotics} (or Lemma 6 in \cite{goldfeld2020asymptotic}) yields $\hat t_n \stackrel{d}{\to} \argmin_{t \in \R^{d_0}}\mathbb{M}(t)$.  
\qed

\section{Concluding remarks}
\label{sec:summary}

In this paper, we have developed a comprehensive limit distribution theory for empirical $\GWp$ that covers general $1 < p< \infty$ and $d \ge 1$, under both the null and the alternative. Our proof technique leveraged the extended functional delta method, which required two main ingredients: (i) convergence of the smooth empirical process in an appropriate normed vector space; and (ii) characterization of the Hadamard directional derivative of $\GWp$ w.r.t. the norm. We have identified the dual Sobolev space $\dot{H}^{-1,p}(\mu*\gamma)$ as the normed space of interest and established the items above to obtain the limit distribution results. Linearity of the Hadamard directional derivative under the alternative enabled establishing the asymptotic normality of the empirical (scaled) $\GWp$. 

To facilitate statistical inference using $\GWp$, we have established the consistency of the nonparametric bootstrap. 
The limit distribution theory was used to study generative modeling via $\GWp$ MDE. We have derived limit distributions for the optimal solutions and the corresponding smooth Wasserstein error, and obtained Gaussian limits when $p=2$ by leveraging the Hilbertian structure of the corresponding dual Sobolev space. Our statistical study, together with the appealing metric and topological structure of $\GWp$ \cite{Goldfeld2020GOT,nietert21}, 
suggest that the smooth Wasserstein framework is compatible with  high-dimensional learning and inference.

An important direction for future research is the efficient computation of $\GWp$. While standard methods for computing $\Wp$ are applicable in the smooth case (by sampling the Gaussian noise), it is desirable to find computational techniques that make use of the structure induced by the convolution with a known smooth kernel. Another appealing direction is to establish Berry-Esseen type bounds for the limit distributions in \cref{thm: main theorem}. Of particular interest is to explore how parameters such as $d$ and $\sigma$ affect the accuracy of the limit distributions in \cref{thm: main theorem}. \cite{sadhu2021} addressed a similar problem for empirical $\mathsf{W}_1^{(\sigma)}$ under the one-sample null case, but their proof relies substantially on the IPM structure of $\mathsf{W}_1$ and finite sample Gaussian approximation techniques developed by \cite{chernozhukov2014,chernozhukov2016}. These techniques do not apply to $p>1$, and thus new ideas, such as the linearization arguments herein, are required to develop Berry-Esseen type bounds for $p > 1$.

\appendix 

\section{Additional proofs}

\subsection{Proof of \cref{lem: dual space}}
\label{sec: dual space}
Completeness of $\dot{H}^{-1,p}(\rho)$ is immediate. To prove separability, let $h \mapsto E(h)$ be the map from $\dot{H}^{-1,p}(\rho)$ into $L^p(\rho;\R^d)$ given in  Lemma \ref{lem:vector-field}. Then, from the first equation in (\ref{eq: vector field}) and the H\"{o}lder's inequality, it holds that $\| h_1-h_2 \|_{\dot{H}^{-1,p}(\rho)} \le \| E(h_1)-E(h_2) \|_{L^p(\rho;\R^d)}$ for all $h_1,h_2 \in \dot{H}^{-1,p}(\rho)$. Since $L^p(\rho;\R^d)$ is separable, we can choose a countable dense subset $W_0$ of the range of $E$. 
As the map $h \mapsto E(h)$ is injective, the set $\{ E^{-1}(w) : w \in W_0 \}$ is (countable and) dense in $\dot{H}^{-1,p}(\rho)$.

Finally, the Borel $\sigma$-field contains the cylinder $\sigma$-field as the coordinate projections are Borel measurable. By separability of $\dot{H}^{-1,p}(\rho)$, to show that the cylinder $\sigma$-field contains the Borel $\sigma$-field, it suffices to show that any closed ball in $\dot{H}^{-1,p}(\rho)$ is measurable relative to the cylinder $\sigma$-field. For any $\ell \in \dot{H}^{-1,p}(\rho)$ and $r > 0$, we have 
\[
\big\{ h \in \dot{H}^{-1,p}(\rho) : \| h-\ell \|_{\dot{H}^{-1,p}(\rho)} \le r \big\} = \bigcap_{\substack{\varphi \in C_0^{\infty}: \\\| \varphi \|_{\dot{H}^{1,q}(\rho)} \le 1}} \big\{ h \in \dot{H}^{-1,p}(\rho) : |(h-\ell)(\varphi)| \le r \big\}.
\]
Since $\dot{H}^{1,q}(\rho)$ (which is isometrically isomorphic to a closed subspace of $L^q(\rho;\R^d)$) is separable, the intersection on the right-hand side can be replaced with the intersection over countably many functions. This shows that the $\dot{H}^{-1,p}(\rho)$-ball on the left-hand side is measurable relative to the cylinder $\sigma$-field.
\qed

\subsection{Proof of \cref{prop:duality-in-Wp}}
\label{app: proof of comparison}
It suffices to prove the proposition when $\| \mu_1 - \mu_0 \|_{\dot{H}^{-1,p}(\rho)} < \infty$. 
Consider the curve $\mu_{t} = (1-t)\mu_{0} +  t\mu_{1}$ for $t \in [0,1]$, and let $E:\R^{d} \to \R^{d}$ be the Borel vector field given in \cref{lem:vector-field} for $h=\mu_1-\mu_0$.
Let $f_{t}$ denote the density of $\mu_{t}$ w.r.t. $\rho$, i.e., $f_{t} = (1-t) f_0 + t f_1$.
By our lower bound of $c$ on one of $f_0$ or $f_1$, we have $f_t > 0$ for all $t$ in the open interval $I = (0,1)$, so the vector field $E/f_t$ is well-defined for $t \in I$. Then the continuity equation \eqref{eq:continuity} holds with $v_{t} = E/f_{t}$ and this choice of $I$.

Now suppose that $f_1 \ge c$ ($f_0 \ge c$ case is similar). Then,
\[
\| v_{t} \|_{L^{p}(\mu_{t};\R^d)}^{p}= \int_{\R^{d}} \frac{|E|^{p}}{f_{t}^{p-1}} d\rho \le c^{-p+1} t^{-p+1} \| E \|_{L^{p}(\rho;\R^d)}^{p},
\]
so that 
$\Wp(\mu_{a},\mu_{b}) \le  c^{-1/q} \| E \|_{L^{p}(\rho;\R^d)} \int_{a}^{b} t^{-1+1/p} d t$ for $0 < a < b < 1$. 
Taking $a \to 0$ and $b \to 1$, we conclude that $\Wp(\mu_{0},\mu_{1}) \le c^{-1/q} p \, \| E \|_{L^{p}(\rho;\R^d)}$.

In view of \eqref{eq: vector field}, the desired conclusion follows once we verify
\[
\| E \|_{L^{p}(\rho;\R^d)} = \sup \left \{ \int_{\R^{d}} \langle \nabla \varphi,E \rangle  d\rho : \varphi \in C_{0}^{\infty}, \| \nabla \varphi \|_{L^{q}(\rho)} \le 1 \right \}, 
\]
but this follows from the fact that $j_p(E) \in \overline{\{ \nabla \varphi : \varphi \in C_{0}^{\infty} \}}^{L^{q}(\rho;\R^d)}$.
\qed

\subsection{Proof of Proposition \ref{prop: MDE limit}}
\label{sec: MDE additional proof}
The proof adopts the approach of \cite[Theorem 7.2]{pollard_min_dist_1980}. 
Let $N_1$ be the neighborhood of $\theta\opt$ appearing in the proof of \cref{thm: MDE limit} Part (i). Set 
        \[
        T_n:=\left\{ t\in \R^{d_0}:\abs{t}\leq C^{-1}(4\big \| \mathbb{G}_n^{(\sigma)}\big \|_{\dot H^{-1,p}(\mu*\gamma_{\sigma})}+2\lambda_n)\right\}.
        \]
        Since $\big \| \mathbb{G}_n^{(\sigma)}\big \|_{\Hu}=\sqrt n {\sf W}_{p}^{(\sigma)}(\hat\mu_n,\mu)+o_{\mathbb{P}}(1)$ (cf. Lemma \ref{lem:errprop}), $\hat\Theta_n\subset \theta\opt+n^{-1/2}T_n$ with inner probability tending to one (cf. equation (\ref{eq: expansion})). 

        Let $\Gamma_n:=\sup_{t\in T_n}\big | \sqrt n{\sf W}_p^{(\sigma)}(\hat \mu_n,\nu_{\theta\opt+n^{-1/2}t})-\big \| \mathbb{G}_n^{(\sigma)}-\inp{t}{\mathfrak{D}}\big \|_{\dot H^{-1,p}(\mu*\gamma_{\sigma})} \big |$. From the proof of \cref{thm: MDE limit} Part (i), we see that $\Gamma_n = o_{\Prob}(1)$. Also, since $\mathbb{G}_n^{(\sigma)} \stackrel{d}{\to} G_{\mu}$ in $\Hu$, by the Skorohod-Dudley-Wichura construction, there exist versions $\tilde{\mathbb{G}}_n^{(\sigma)}$ and $\tilde{G}_\mu$ of $\mathbb{G}_n^{(\sigma)}$ and $G_\mu$ (i.e., $\tilde{\mathbb{G}}_n^{(\sigma)}$ and $\tilde{G}_\mu$ have the same distributions as $\mathbb{G}_n^{(\sigma)}$ and $G_\mu$, respectively, as $\Hu$-valued random variables) such that $\tilde{\mathbb{G}}_n^{(\sigma)} \to \tilde{G}_\mu$ in $\Hu$ almost surely. 
        Choose $\delta_n,\epsilon_n,\gamma_n\dn 0$ in such a way that
        \[
        \mathbb P(\lambda_n>\delta_n)\to 0,\quad \mathbb P\left(\big \| \tilde{\mathbb{G}}_n^{(\sigma)}-\tilde{G}_{\mu}\big \|_{\dot H^{-1,p}(\mu*\gamma_{\sigma})}>\epsilon_n\right)\to 0,\quad \mathbb P(\Gamma_n>\gamma_n)\to 0.
        \]
        Set $\beta_n:=\max\left\{ 2\gamma_n
        +\delta_n,2\epsilon_n\right\}$. 

       \underline{Part (ii).} We first show that $K(\mathbb{G}_n^{(\sigma)},\beta_n) \stackrel{d}{\to} K(G_{\mu},0)$. To this end, it suffices to show that $K(\tilde{\mathbb{G}}_n^{(\sigma)},\beta_n) \stackrel{\Prob}{\to} K(\tilde{G}_\mu,0)$. Observe that for any $h,h'\in \dot H^{-1,p}(\mu*\gamma_{\sigma})$, 
\begin{align*}
    &\abs{\inf_{t\in\R^{d_0}}\norm{h-\inp{t}{\mathfrak{D}}}_{\dot H^{-1,p}(\mu*\gamma_{\sigma})}-\inf_{t\in\R^{d_0}}\norm{h'-\inp{t}{\mathfrak{D}}}_{\dot H^{-1,p}(\mu*\gamma_{\sigma})}}\\&\leq 
 \sup_{t\in\R^{d_0}}\abs{\norm{h-\inp{t}{\mathfrak{D}}}_{\dot H^{-1,p}(\mu*\gamma_{\sigma})}-\norm{h'-\inp{t}{\mathfrak{D}}}_{\dot H^{-1,p}(\mu*\gamma_{\sigma})}}  
 \\&\leq{\norm{h-h'}_{\dot H^{-1,p}(\mu*\gamma_{\sigma})}}.
\end{align*}
Thus, when $\big \| \tilde{\mathbb{G}}_n^{(\sigma)}-\tilde{G}_\mu \big \|_{\dot H^{-1,p}(\mu*\gamma_{\sigma})}\leq \epsilon_n$ and $t\in K(G_{\mu},0)$,
\begin{align*}
    &\big \| \tilde{\mathbb{G}}_n^{(\sigma)}-\inp{t}{\mathfrak{D}}\big \|_{\dot H^{-1,p}(\mu*\gamma_{\sigma})}\leq \norm{\tilde{G}_\mu-\inp{t}{\mathfrak{D}}}_{\dot H^{-1,p}(\mu*\gamma_{\sigma})}+\big \| \tilde{\mathbb{G}}_n^{(\sigma)}-\tilde{G}_\mu\big \|_{\dot H^{-1,p}(\mu*\gamma_{\sigma})} \\
\\
    &\qquad = \inf_{t'\in\R^{d_0}}\big \| \tilde{G}_\mu-\inp{t'}{\mathfrak{D}}\big \|_{\dot H^{-1,p}(\mu*\gamma_{\sigma})} +\big \| \tilde{\mathbb{G}}_n^{(\sigma)}-\tilde{G}_\mu\big \|_{\dot H^{-1,p}(\mu*\gamma_{\sigma})}
\\
&\qquad \leq \inf_{t'\in\R^{d_0}}\big \| \tilde{\mathbb{G}}_n^{(\sigma)}-\inp{t'}{\mathfrak{D}}\big \|_{\dot H^{-1,p}(\mu*\gamma_{\sigma})}+2\epsilon_n.
\end{align*}
This implies that, when $\big \| \tilde{\mathbb{G}}_n^{(\sigma)}-\tilde{G}_\mu \big \|_{\Hu} \leq \epsilon_n$, which occurs with probability approaching one, it holds that $K(\tilde{G}_\mu,0)\subset K(\tilde{\mathbb{G}}_n^{(\sigma)},\beta_n)$.

Now, for any $\alpha>0$ and $\omega \in \Omega$, there exists $\lambda_0(\omega)$ such that $K(\tilde{G}_\mu(\omega),\lambda)\subset K(\tilde{G}_\mu(\omega),0)^{\alpha}$ for every $\lambda\leq \lambda_0(\omega)$, since for any sequence $\kappa_n\dn 0$, $K(\tilde{G}_\mu(\omega),\kappa_n)$ is a decreasing sequence of compact sets whose intersection is contained in the open set $K(\tilde{G}_\mu(\omega),0)^{\alpha}$, and hence $K(\tilde{G}_\mu(\omega),\kappa_n)\subset K(\tilde{G}_\mu(\omega),0)^{\alpha}$ for $n$
sufficiently large. It follows that if $\kappa_n\to0$ almost surely, 
then    $\Prob(K(\tilde{G}_\mu,\kappa_n)\subset K(\tilde{G}_\mu,0)^{\alpha})\to 1$. Applying this result with $\kappa_n=\beta_n+2\big \| \tilde{\mathbb{G}}_n^{(\sigma)}-\tilde{G}_\mu\big \|_{\dot H^{-1,p}(\mu*\gamma_{\sigma})}$, we have that $K(\tilde{\mathbb{G}}_n^{(\sigma)},\beta_n)$ is contained in
\begin{align*}
&\Big\{ t:\big \| \tilde{\mathbb{G}}_n^{(\sigma)}-\inp{t}{\mathfrak{D}}\big \|_{\dot H^{-1,p}(\mu*\gamma_{\sigma})}\leq \inf_{t\in\R^{d_0}}\big \| \tilde{\mathbb{G}}_n^{(\sigma)}-\inp{t}{\mathfrak{D}}\big \|_{\dot H^{-1,p}(\mu*\gamma_{\sigma})}+\lambda_n\Big\}\\
&\subset K(\tilde{G}_{\mu},0)^{\alpha},
\end{align*}
where the final inclusion holds with probability tending to one. Conclude that $K(\tilde{\mathbb{G}}_n^{(\sigma)},\beta_n) \stackrel{\Prob}{\to} K(\tilde{G}_{\mu},0)$ as desired. 

\underline{Part (i).} Finally, suppose that $\Gamma_n\leq \gamma_n$ and $\lambda_n\leq \delta_n$. Observe that if $t\in T_n^c$, then
\begin{align*}
    \big \|\mathbb{G}_n^{(\sigma)}-\inp{t}{\mathfrak{D}} \big \|_{\dot H^{-1,p}(\mu*\gamma_{\sigma})}&
    \geq C\abs{t}-\big \| \mathbb{G}_n^{(\sigma)}\big \|_{\Hu}
    \\
    &>3\big \| \mathbb{G}_n^{(\sigma)}\big \|_{\Hu}+\lambda_n
    \\
    &\geq \big \| \mathbb{G}_n^{(\sigma)}\big \|_{\Hu},  
\end{align*}
so that $\inf_{t\in\R^{d_0}}\big \| \mathbb{G}_n^{(\sigma)}-\inp{t}{\mathfrak{D}} \big \| _{\dot H^{-1,p}(\mu*\gamma_{\sigma})}
= \inf_{t\in T_n}\big \| \mathbb{G}_n^{(\sigma)}-\inp{t}{\mathfrak{D}} \big \|_{\dot H^{-1,p}(\mu*\gamma_{\sigma})}
$.
By repeated applications of the triangle inequality, for $t\in T_n \cap \sqrt n(\hat\Theta_n-\theta\opt)$, we have
\begin{align*}
\inf_{t'\in T_n}\big \| \mathbb{G}_n^{(\sigma)}-\inp{t'}{\mathfrak{D}}\big \|_{\dot H^{-1,p}(\mu*\gamma_{\sigma})}
&\geq \inf_{t'\in T_n}\sqrt{n}{\sf W}_{p}^{(\sigma)}(\hat \mu_n,\nu_{\theta\opt+n^{-1/2}t'})-\gamma_n
\\
&\geq \sqrt{n}{\sf W}_{p}^{(\sigma)}(\hat \mu_n,\nu_{\theta\opt+n^{-1/2}t})-\delta_n-\gamma_n
\\
&\geq \big \| \mathbb{G}_n^{(\sigma)}-\inp{t}{\mathfrak{D}}\big \|_{\dot H^{-1,p}(\mu*\gamma_{\sigma})}-\beta_n.
\end{align*}
Conclude that 
\[
\Prob_{*}\Big( \big [T_n \cap \sqrt n(\hat\Theta_n-\theta\opt)\big] \subset K(\mathbb{G}_n^{(\sigma)},\beta_n) \Big) \to 1.
\]
Since $\hat\Theta_n\subset \theta\opt+n^{-1/2}T_n$ with inner probability tending to one, we obtain the desired conclusion. This completes the proof.  \qed

\bibliographystyle{amsalpha}
\bibliography{references}

\end{document}